\documentclass[a4paper]{amsart}
\usepackage[latin1]{inputenc}
\usepackage{amsfonts}
\usepackage{amsmath,latexsym,amssymb,amsfonts, amsthm}
\usepackage{esint}
\usepackage{cite}
\usepackage{graphicx}
\usepackage{amscd}
\usepackage{xcolor}
\definecolor{earlierrevision}{rgb}{0.05,0.40,0.40}
\usepackage{bm}
\usepackage{tikz}
\usepackage{enumerate}

\usepackage{psfrag}
\usepackage{epsfig}

\usepackage[
  hmarginratio={1:1},
  vmarginratio={1:1},
  textwidth=16cm,
  textheight=22cm,
  heightrounded,
]{geometry}

\usepackage{graphicx,color}
\usepackage[colorlinks]{hyperref}
\hypersetup{linkcolor=blue,citecolor=blue,filecolor=black,urlcolor=blue}

\newtheorem{theorem}{Theorem}
\newtheorem{corollary}[theorem]{Corollary}
\newtheorem{proposition}[theorem]{Proposition}
\newtheorem{lemma}[theorem]{Lemma}
\theoremstyle{definition}
\newtheorem{remark}[theorem]{Remark}

\newtheorem{definition}[theorem]{Definition}

\numberwithin{theorem}{section}

\numberwithin{equation}{section}

\newcommand{\beq}{\begin{equation}}
\newcommand{\eeq}{\end{equation}}

\newcommand{\R}{\mathbb{R}}
\newcommand{\N}{\mathbb{N}}

\newcommand{\mf}[1]{\mathbf{#1}}
\newcommand{\bs}[1]{\boldsymbol{#1}}

\newcommand{\pa}{\partial}

\newcommand{\cC}{{\mathcal C}}

\newcommand{\cF}{{\mathcal F}}

\newcommand{\cH}{{\mathcal H}}

\newcommand{\dist}{{\rm dist}}
\newcommand{\supp}{{\rm supp}}

\newcommand{\eps}{\varepsilon}

\newcommand{\ssubset}{\subset\joinrel\subset}

\DeclareMathOperator{\loc}{loc}

\renewcommand{\epsilon}{\varepsilon}

\author[N. Soave and S. Terracini]{Nicola Soave and Susanna Terracini}\thanks{}
\address{Nicola Soave and Susanna Terracini\newline \indent
 Dipartimento di Matematica ``Giuseppe Peano'', Universit\`a di Torino, \newline \indent
Via Carlo Alberto 10,
10123 Torino, Italy}
\email{nicola.soave@unito.it, susanna.terracini@unito.it}

\title[On partially segregated harmonic maps with free boundary]{On partially segregated harmonic maps: optimal regularity and structure of the free boundary}
\thanks{N.S. is partially supported by the PRIN Project no. 2022R537CS \emph{Nodal Optimization, NOnlinear elliptic equations, NOnlocal geometric problems, with a focus on regularity} ($NO^3$ - CUP D53D23005940006), and S. T. is partially supported by the PRIN 2022 project 20227HX33Z -- \emph{Pattern formation in nonlinear phenomena} (CUP D53D23005680006), funded by European Union - Next Generation EU within the PRIN 2022 program (D.D. 104 - 02$/$02$/$2022 Ministero dell'Universit\`a e della Ricerca, Italy). Both authors are affiliated to the INDAM - GNAMPA group.\\
Declarations of interest: none.\\
Data availability: the authors declare that all data involved in this research are fully available.}

\begin{document}
\begin{abstract}
We consider triplets of densities $(u_1,u_2,u_3)$ minimizing the Dirichlet energy
\[\sum_{j=1}^3 \int_{\Omega} |\nabla u_j|^2\,dx
\]
over a bounded domain $\Omega\subset \R^N$, subject to the partial segregation condition:
\[
u_1\,u_2\,u_3 \equiv 0 \ \text{in $\Omega$.}
\]
We prove optimal regularity of the minimizers in spaces of H\"older continuous functions of exponent $3/4$; furthermore we prove that the free boundary is a collection of a locally finite number of smooth codimension one manifolds up to a residual set of Hausdorff dimension at most $N-2$. Finally we prove uniform-in-$\beta$ a priori bounds for minimal solutions to the penalized energy:
\[
J_\beta(\mf{u}, \Omega) = \int_{\Omega} \sum_{i=1}^3 |\nabla u_i|^2 \,dx+ \beta \int_{\Omega} \prod_{j=1}^3 u_j^2\,dx,
\]
in spaces of H\"older continuous functions of exponent less than $3/4$. The proofs make use of an Almgren-type monotonicity formula, blow-up analysis together with some new Liouville-type theorems.
\end{abstract}

\maketitle

\section{Introduction}

\noindent The goal of this paper is to describe the optimal regularity properties of the solutions, as well as the properties of the free boundary, for the minimizers of the problem
\beq\label{min fix traces intro 1}
\min\left\{ \sum_{j=1}^3 \int_{\Omega} |\nabla u_j|^2\,dx\left|\begin{array}{l}  \mf{u}-\boldsymbol{\psi} \in H_0^1(\Omega,\R^3) \\ u_1\,u_2\,u_3 \equiv 0 \ \text{in $\Omega$}\end{array}\right.\right\},
\eeq
where $\Omega$ is a bounded regular domain of $\R^N$ and $\boldsymbol{\psi}=(\psi_1, \psi_2, \psi_3)$ is a triplet of nonnegative Lipschitz continuous functions in $\overline{\Omega}$ satisfying the \emph{partial segregation condition}
\beq\label{PSC}
\psi_1\,\psi_2\,\psi_3 \equiv 0 \quad \text{in $\overline{\Omega}$};
\eeq
namely, $\psi_1$, $\psi_2$ and $\psi_3$ cannot be positive at the same point $x_0$, but their positivity sets can overlap pairwise. Minimizers of \eqref{min fix traces intro 1} are \emph{harmonic maps} into the singular target represented by the union of the three coordinate hyperplanes in $\R^3$.

This article is the second part of our study on partial segregation models with variational interaction. More precisely, in our previous paper \cite{ST24p1}, we showed that any such minimizer is in the H\"older space $C^{0,\alpha}(\overline{\Omega})$ for any positive exponent $\alpha$ up to some universal threshold $\bar \nu \in (0,1)$, and obtained some preliminary extremality conditions. The exact value of $\bar \nu$ is not explicitly determined at this stage, but can be characterized in terms of an ``overlapping partition problem" on the sphere, see \cite[Remark 1.3]{ST24p1}; this characterization gives the estimate $\bar \nu \le 2/3$, which, as we will see, is suboptimal for the regularity of solutions to \eqref{min fix traces intro 1}. In this paper we prove indeed the \emph{optimal} $C^{0,3/4}$ regularity of minimizers and investigate the structure of the associated free boundary. To this end, it is worthwhile noticing that we have essentially two reasonable ways to introduce the concept of a \emph{free boundary} for a continuous function $\mf{v}$ satisfying the partial segregation condition. At first, we define the \emph{multiplicity} of any $x\in\Omega$ as
\beq\label{def: multiplicity}
m(x):= \# \left\{i \in \{1,2,3\}: \ v_i(x) = 0\right\}
\eeq
(number of components vanishing at $x$), and then we set
\beq\label{def: gamma k}
\Gamma_{\mf{v}}^{k} := \left\{x \in \Omega: \ m(x) = k\right\}, \quad k=0, 1, 2, 3.
\eeq
By the partial segregation condition $v_1 \,v_2\,v_3 \equiv 0$, the set $\Gamma_{\mf{v}}^0$ is empty. If $x \in \Gamma_{\mf{v}}^1$, then two components of $\mf{v}$ are positive at $x$, and hence in a neighborhood of $x$ (by continuity), and the third one vanishes identically in such neighborhood. Therefore  we can define the \emph{nodal set} of $\mf{v}$ as
\beq\label{def fb}
\Gamma_{\mf{v}}:= \Gamma_{\mf{v}}^{2} \cup \Gamma_{\mf{v}}^{3},
\eeq
the set of points with multiplicity $2$ (\emph{double points}) or $3$ (\emph{triple points}).
On the other hand, it is also natural to define the \emph{free boundary} as
\beq\label{def fb2}
\Gamma_{\mf{v}}':= \Omega \cap \bigcup_{i=1}^3 \partial \{v_i>0\}.
\eeq

The main results of this paper can be summarized as follows:

\begin{theorem}\label{thm: main fixed traces}
Let $\mf{v}$ be any minimizer of \eqref{min fix traces intro 1}. Then:
\begin{enumerate}[(i)]
\item $\mf{v} \in C^{0,3/4}(\Omega)$, and the exponent $3/4$ is  optimal for regularity in H\"older spaces.
\item {If $\Gamma_{\mf{v}} \neq \Omega$}, then the sets $\Gamma_{\mf{v}}$ and $\Gamma_{\mf{v}}'$ coincide; $\Gamma_{\mf{v}}$ has Hausdorff dimension at most $N-1$, and more precisely $\Gamma_{\mf{v}}= \mathcal{R}_{\mf{v}} \cup \Sigma_{\mf{v}}$, where:
\begin{itemize}
\item[($a$)] $\mathcal{R}_{\mf{v}} \subset \Gamma_{\mf{v}}^2$ is the union of smooth hypersurfaces of dimension $N-1$.
\item[($b$)] $\Sigma_{\mf{v}}$ has Hausdorff dimension at most $N-2$, and is locally finite in dimension $2$; moreover, $\{v_i=0\}= \overline{\mathrm{int}\{v_i=0\}}$ for every $i$.
\end{itemize}
\end{enumerate}
\end{theorem}

Here optimal means  that there exists a boundary datum $\boldsymbol{\psi}$ and an associated minimizer $\mf{v}$ for \eqref{min fix traces intro 1} such that $\mf{v} \not \in C^{0,3/4+\eps}(\Omega)$, for every $\eps>0$.

The previous theorem will be derived by combining three main results, obtained in a more general setting than the one previously outlined.

\begin{remark}
The existence of a minimizer for problem \eqref{min fix traces intro 1} is easy to show, we refer to \cite[Theorem 1.2]{ST24p1} for the details. The assumption \(\Gamma_{\mf{v}}\neq\Omega\) in part~($ii$)
is necessary. For instance, the constant map \(\mf{v}=(1,0,0)\)
has zero energy and hence minimizes \eqref{min fix traces intro 1}
with its own boundary trace. In this case, \(\Gamma_{\mf{v}}=\Omega\)
and \(\Gamma_{\mf{v}}'=\emptyset\).
\end{remark}

\subsection{Main results in full generality}

To begin, let us precisely recall what we proved in the previous paper \cite{ST24p1}. There, we did consider the minimization with fixed traces described above, but we also considered the following (more flexible) penalised problem: let $\Omega \subset \R^N$ be a domain (not necessarily bounded), and let $\{\mf{u}_\beta=(u_{1,\beta}, u_{2,\beta}, u_{3,\beta})\}_{\beta >1} \subset H^1_{\loc}(\Omega)$ be a family of weak solutions to
\begin{equation}\label{P beta}
\Delta u_{i} = \beta u_i \prod_{\substack{j \neq i}} u_{j}^2, \quad u_{i}>0 \quad \text{in $\Omega$, for $i=1,2,3$},
\end{equation}
under the following assumptions:
\begin{itemize}
\item[(h1)] $\{\mf{u}_\beta\}$ is uniformly bounded in $L^\infty(\Omega)$, namely $\|\mf{u}_\beta\|_{L^\infty(\Omega)} \le C$ for a positive constant $C>0$ independent of $\beta$.
\item[(h2)] $\mf{u}_\beta$ is a minimizer of \eqref{P beta} with respect to variations with compact support, in the sense that for every $\Omega' \ssubset \Omega$
\[
J_\beta(\mf{u}_{\beta}, \Omega') \le J_\beta(\mf{u}_\beta + \bs{\varphi}, \Omega') \qquad \forall \bs{\varphi} \in H_0^1(\Omega', \R^3),
\]
where
\beq\label{def functional}
J_\beta(\mf{u}, \Omega') = \int_{\Omega'} \sum_{i=1}^3 |\nabla u_i|^2 \,dx+ \beta \int_{\Omega'} \prod_{j=1}^3 u_j^2\,dx.
\eeq
\end{itemize}
Such ternary interactions arise in mean field approximations of mixtures when triple collisions between atoms of the three species are taken into account (we refer to \cite{LQZ24, Pet14} for physical models where ternary interactions play a relevant role). In \cite[Theorem 1.1]{ST24p1}, we showed the existence of some universal exponent  $\bar\nu\in(0,2/3]$ such that $\{\mf{u}_\beta\}$ is uniformly bounded in $C^{0,\alpha}_{\loc}(\Omega)$ for every $\alpha \in (0,\bar \nu)$, and thus, up to a subsequence, as $\beta \to +\infty$
\begin{gather*}
\mf{u}_\beta \to \tilde{\mf{u}} \quad \text{in $H^1_{\loc}(\Omega)$ and in $C^{0,\alpha}_{\loc}(\Omega)$, for every $\alpha \in (0, \bar \nu)$, and} \\
\beta \int_{\omega} \prod_{j=1}^3 u_{j,\beta}^2\,dx \to 0 \quad \text{for every $\omega \ssubset \Omega$}.
\end{gather*}
Regarding the limit function $\tilde{\mf{u}}$, at first it is convenient to adopt the following definition from \cite{ST24p1}.

\begin{definition}\label{def: L}
Let $\Omega \subset \R^N$ be a domain. We denote by $\mathcal{L}(\Omega)$ the set of nonnegative functions $\mf{v} \in H^1(\Omega,\R^3) \cap C(\Omega,\R^3)$, $\mf{v} \not \equiv 0$, satisfying the following conditions: there exist sequences $M_n \to +\infty$ and $\{\mf{v}_n\} \subset H^1(\Omega,\R^3)$, with $v_{i,n}>0$ in $\Omega$, such that
\begin{itemize}
\item[($i$)] $\mf{v}_n$ is a minimizer of $J_{M_n}$ with respect to variations with fixed trace in $\Omega$, namely
\[
J_{M_n}(\mf{v}_n, \Omega) \le J_{M_n}(\mf{v}_n + \bs{\varphi}, \Omega) \qquad \forall \bs{\varphi} \in H_0^1(\Omega,\R^3).
\]
\item[($ii$)] $\mf{v}_n \to \mf{v}$ strongly in $H^1(\Omega)$, and in $C^{0}(\overline\Omega)$.
\item[($iii$)] As $n \to \infty$
\[
\int_\Omega M_n \prod_{j=1}^3 v_{j,n}^2\,dx \to 0.
\]
\end{itemize}
We denote by $\mathcal{L}_{\loc}(\Omega)$ the set of functions $\mf{v}$ such that $\mf{v} \in \mathcal{L}(\omega)$ for every $\omega \ssubset \Omega$.
\end{definition}

Theorem 1.1 of \cite{ST24p1} establishes that any non-zero limit function $\tilde{\mf{u}}$ of systems \eqref{P beta} belongs to the class $\mathcal{L}_{\loc}(\Omega)$. As explained in \cite[Remark~1.8]{ST24p1}, one can approximate any given minimizer of \eqref{min fix traces intro 1} by adding a lower-order
penalization term to the approximating functionals. The same arguments
then apply to these approximations and yield the properties of minimizers used below.

One advantage of this approach is that the uniform H\"older estimates for the approximating problems in \cite{ST24p1} give continuity of minimizers from the outset. In Section \ref{sec: ihb} we use the same approach with sharper Liouville theorems to obtain uniform H\"older estimates for every \(0<\alpha<3/4\).

In Theorem 1.7 of \cite{ST24p1}, we demonstrated that if $\mf{v} \in \mathcal{L}_{\loc}(\Omega)$, then $\mf{v} \in C^{0,\alpha}(\Omega)$ for every $\alpha \in (0, \bar \nu)$, each component $v_i$ is harmonic when positive, and the partial segregation condition \eqref{PSC} holds. Additionally, we derived the following domain variation formula:
\beq\label{loc Poh}
\int_{S_r(x_0)} \sum_{i=1}^3 |\nabla \tilde u_i|^2\,d\sigma = \frac{N-2}{r} \int_{B_r(x_0)} \sum_{i=1}^3 |\nabla \tilde u_i|^2\,dx + 2 \int_{S_r(x_0)} \sum_{i=1}^3 (\pa_\nu \tilde u_i)^2\,d\sigma,
\eeq
which holds whenever $B_r(x_0) \subset \Omega$.

\medskip

As a first goal, we focus on the structure of the \emph{nodal set} $\Gamma_{\mf{v}}$ and of the \emph{free boundary} $\Gamma_{\mf{v}}'$ defined by \eqref{def fb} and \eqref{def fb2}, respectively.

\begin{theorem}\label{thm: nodal set}
Let $\mf{v} \in \mathcal{L}_{\loc}(\Omega)$. Then either $\Gamma_{\mf{v}}=\Omega$ (and in this case two components of $\mf{v}$ vanish identically, and the remaining one is a positive harmonic function in $\Omega$), or the following properties hold:
\begin{enumerate}[(i)]
\item $\Gamma_{\mf{v}} = \Gamma_{\mf{v}}'$;
\item $\Gamma_{\mf{v}}$ has Hausdorff dimension at most $N-1$, and more precisely $\Gamma_{\mf{v}}= \mathcal{R}_{\mf{v}} \cup \Sigma_{\mf{v}}$, where:
\begin{itemize}
\item[($a$)] $\mathcal{R}_{\mf{v}} \subset \Gamma_{\mf{v}}^2$ is the union of smooth hypersurfaces of dimension $N-1$.
\item[($b$)] $\Sigma_{\mf{v}}$ has Hausdorff dimension at most $N-2$, and is locally finite in dimension $2$. \end{itemize}
\item $\{v_i=0\}= \overline{\mathrm{int}\{v_i=0\}}$ for every $i$.
\end{enumerate}
\end{theorem}

In the proof of point ($ii$), we need to treat in a different way points of multiplicity $2$ and $3$: in order to study the latter ones, we exploit an Almgren-type monotonicity formula for triplets in $\mathcal{L}_{\loc}(\Omega)$. This is a fundamental ingredient for demonstrating the existence of homogeneous blow-ups. As a subsequent step, the classification of such blow-ups and Federer's reduction principle allow us to deduce that the set of triple points has dimension at most $N-2$. Unfortunately, Almgren's formula proves to be of little use for discussing the regularity of the nodal set at double points. To overcome this drawback, we exploit the minimality properties of the functions in $\mathcal{L}_{\loc}(\Omega)$ to show that, in the neighborhood of each double point, the set $\Gamma_{\mf{v}}$ is the nodal set of a harmonic function. The minimality also plays a crucial role in the proof of points ($i$) and ($iii$).

\begin{remark}\label{rem: part nodal}
It is interesting to observe that positivity sets of different components can overlap, but, by Theorem \ref{thm: nodal set}, when $\Gamma_{\mf{v}} \neq \Omega$ the  zero sets $\mathrm{int}\{v_1=0\}$, $\mathrm{int}\{v_2=0\}$ and $\mathrm{int}\{v_3=0\}$ forms an open partition of $\Omega$, in the sense that they are pairwise disjoint, and the union of their relative closure is $\Omega$.
\end{remark}

Next, we focus on the optimal regularity of functions in $\mathcal{L}_{\loc}$. As already recalled, the main result in \cite{ST24p1} establishes the validity of uniform H\"older estimates for sequences of functions satisfying (h1) and (h2) above, for ``small" exponents $\alpha \in (0, \bar \nu)$, with $\bar \nu \le 2/3$. In this paper, we improve the H\"older estimates up to the exponent $3/4$.

\begin{theorem}\label{thm: improved Holder bounds}
Let $\mf{u}_\beta =(u_{1,\beta}, u_{2,\beta}, u_{3,\beta})$ be a solution of \eqref{P beta} at fixed $\beta>1$. Suppose that (h1) and (h2) hold. Then, for every compact set $K \ssubset \Omega$, we have
\beq\label{eq: imp hol}
\|\mf{u}_\beta\|_{C^{0,\alpha}(K)} \le C, \quad \forall \alpha \in \left(0,\frac34\right).
\eeq
Moreover, any function in $\mathcal{L}_{\loc}(\Omega)$ is of class $C^{0,3/4}(\Omega)$.
\end{theorem}

In proving this theorem, we will exploit some new Liouville-type theorems for entire \emph{minimal} solutions of \eqref{P beta} and entire functions in $\mathcal{L}_{\loc}(\R^N)$, Proposition \ref{prop: imp liou}, in combination with the uniform $\alpha$-H\"older estimates for small $\alpha$ given by \cite[Theorem 1.1]{ST24p1}. It is worthwhile noticing that this can be considered as an ``improvement of regularity", and we could not directly prove Theorem \ref{thm: improved Holder bounds} without the initial a-priori-bounds in \cite{ST24p1}. A similar two-steps procedure was used (in a different context) in \cite{TVZ16}, which concerns H\"older estimates for a full segregation problem driven by the fractional Laplacian. It is also interesting to note that, in the proof of the Liouville-type theorems, we will exploit some properties of the free boundary demonstrated earlier. Thus, ultimately, the regularity of the free boundary proves useful for studying the optimal regularity of the densities.

The next step consists in showing that Theorem \ref{thm: improved Holder bounds} gives the optimal regularity for the class $\mathcal{L}_{\loc}$ in H\"older spaces. To this end, we exhibit a minimizing harmonic map for problem \eqref{min fix traces intro 1} with regularity exactly $C^{0,3/4}(\Omega)$, and no higher. Since, as already pointed out, minimizers of \eqref{min fix traces intro 1} with fixed traces can be regarded as elements of $\mathcal{L}_{\loc}$, such an example establishes that the regularity of the densities in Theorem \ref{thm: improved Holder bounds} cannot be improved.

To be precise, we consider problem \eqref{min fix traces intro 1} in the unit ball $B_1$ of the plane $\R^2$, and we conveniently choose the traces as follows (in polar coordinates): on $\pa B_1$ we set
\beq\label{tr mer}
\psi_1(\Theta)=\begin{cases} \sin\left(\frac34 \Theta\right) & \text{if }0 \le \Theta \le \frac43\pi \\ 0 & \text{if } \frac43\pi < \Theta < 2\pi, \end{cases} \quad \psi_2(\Theta) = \psi_1\left(\Theta-\frac{2}{3}\pi\right), \quad \psi_3(\Theta) = \psi_1\left(\Theta-\frac{4}{3}\pi\right),
\eeq
and we extend these functions inside the ball in a Lipschitz fashion, in such a way to maintain the partial segregation condition.

\begin{theorem}\label{prop: mer min}
The function $\tilde{\mf{v}} \in H^1(B_1) \cap C^{0,3/4}(\overline{B_1})$ defined in polar coordinates by
\[
(\tilde v_1(\mathfrak{r},\Theta), \tilde v_2(\mathfrak{r},\Theta), \tilde v_3(\mathfrak{r},\Theta)) := \mathfrak{r}^{3/4} (\psi_1(\Theta), \psi_2(\Theta), \psi_3(\Theta))
\]
is a minimizer for problem \eqref{min fix traces intro 1} with $(\psi_1, \psi_2, \psi_3)$ defined by \eqref{tr mer}.
\end{theorem}

The proof is based on a preliminary analysis of the free boundary in dimension $2$. We refer the reader to Section \ref{sec: fb dim 2} and to Theorem \ref{thm: fb dim 2 1} for the precise results, which we believe are of independent interest.

\begin{remark}
The previous example shows that functions in $\mathcal{L}_{\loc}$ are $C^{0,3/4}$, but not better in general. Thus, from the point of view of the regularity in $\mathcal{L}_{\loc}$, Theorem \ref{thm: improved Holder bounds} is optimal. One may still ask whether it is possible to derive optimal uniform bounds for the approximating sequence $\{\mf{u}_\beta\}$ in $C^{0,3/4}_{\loc}$. This is an interesting open problem. In connection with full segregation problems, which we will discuss later in the introduction, it is worth mentioning that the transition from nearly optimal to optimal uniform bounds is very delicate (see \cite{SZ15}).
\end{remark}

Theorems \ref{thm: improved Holder bounds} and \ref{prop: mer min} give the optimal regularity for harmonic maps minimizing \eqref{min fix traces intro 1}. It is remarkable that we have a lower regularity compared to the Lipschitz continuity found in \cite{BoBuFo}, regarding singular limits to the partial segregation model
\beq\label{pb farid}
\Delta u_i= \beta \prod_{j =1}^3 u_j, \quad u_i>0 \qquad \text{in $\Omega$}.
\eeq
Problem \eqref{pb farid}, which was previously studied in \cite{CaRo}, is non-variational, and models partial segregation with symmetric-type interaction. The difference in optimal regularity reveals a significant diversity between models of partial or full segregation: indeed, in the latter case both for variational \cite{CL08, CTVind, NTTV10} and symmetric \cite{CKL, CTV05} interaction the optimal regularity is the Lipschitz one.

\subsection{Comparison with the case of  fully segregated  energy minimizers.} Segregation phenomena naturally appear in problems of a purely mathematical nature, and arise from physics, biology, and other applied sciences. As discussed in detail in the introduction of the previous paper \cite{ST24p1}, most of the available results have focused on phenomena of full segregation, i.e., phenomena where any limiting profile $\bar{\mf{u}}$ satisfies the \emph{full segregation condition}
\[
\bar u_{i}  \bar u_j \equiv 0 \quad \text{in $\Omega$, for every $i \neq j$}.
\]
Specifically, this ensures that the positivity sets of the different components of the limiting profile are mutually disjoint. The partial segregation condition, on the other hand, allows the positivity sets of different components to overlap, but only pairwise. Our study is the first contribution concerning problems characterized by variational structure and partial segregation. Previous results, concerning a class of non-variational problems, have been obtained in \cite{BoBuFo, CaRo}.

When condition \eqref{PSC} is replaced by the \emph{(full) segregation condition}
\[
\phi_i\,\phi_j \equiv 0 \quad \text{in $\Omega$ for every $i \neq j$}
\]
(namely the positivity sets of $\phi_i$ and $\phi_j$ are mutually disjoint), problem
\beq\label{FSP}
\begin{cases}
\displaystyle \Delta u_i = \beta u_i \sum_{j \neq i} u_j^2, \quad u_i>0  & \text{in }\Omega\\
u_i = \phi_i & \text{on }\partial \Omega
\end{cases}
\eeq
and its singular limit were studied in \cite{CL08, NTTV10, SZ15, Wa}, see also \cite{CafLin06, CTV03, CTVind, DWZjfa, STTZ16, TT12}. For each $\beta>1$ fixed, one can easily find a solution $\mf{v}_\beta$ by minimizing the associated functional, and prove that a family of (positive) minimizers weakly converges in $H^1$ to a fully segregated nonnegative function $\bar{\mf{u}}$, namely $\bar u_i \, \bar u_j \equiv 0$ a.e. in $\Omega$, for every $i \neq j$. At this point, one can introduce a \emph{limit class} similar to $\mathcal{L}_{\loc}$, and study the same issues considered here and in \cite{ST24p1}, namely \emph{uniform-in-$\beta$ a priori bounds in H\"older spaces}, \emph{optimal regularity of the limit densities}, \emph{regularity of the free boundary}. The picture is now well understood, and optimal results are available. It is not difficult to show that $\bar{\mf{u}}$ is again a \emph{harmonic map into a singular space}, in the sense that it is a minimizer for
\[
\min\left\{ \int_{\Omega} \sum_{i=1}^3 |\nabla u_i|^2\,dx \left| \begin{array}{l} \mf{u} \in H^1(\Omega,\R^3), \ u_i = \phi_i \quad \text{ on $\pa \Omega$} \\ u_i\,u_j \equiv 0 \quad \text{a.e. in $\Omega$, for every $i \neq j$}\end{array}\right.\right\}.
\]
Notice that the target space is
\[
\left\{x \in \R^{3}: \ x_i \,x_j = 0 \quad \text{for every $i \neq j$, and $x_i \,\ge\, 0$ for every $i$}\right\},
\]
which is a singular space with non-positive curvature in the sense of Alexandrov, see \cite{GrSc}. This marks another fundamental difference with the case of partial segregation investigated in this paper; here the target, not being geodesically convex, cannot have non-positive curvature. This may be seen as an easy exercise by exhibiting two different minimal geodesics connecting the two points $(1,1,0)$ and $(0,0,1)$.  The lack of geodesic convexity may reflect on a lack of uniqueness of minimizers of the constrained Dirichlet problem and, ultimately, can be held responsible of the lack of Lipschitz continuity of minimizers.
Indeed, contrary to the present case of partial segregation, in the full segregation one, it is proved that $\{\mf{u}_\beta: \beta>1\}$ is bounded in $C^{0,1}_{\loc}(\Omega)$, see \cite{SZ15}; we also refer to \cite{NTTV10} for uniform H\"older estimates, which are a fundamental ingredient in the proof of \cite{SZ15}. Despite the fact that each function $\mf{v}_\beta$ is smooth, uniform Lipschitz estimates are optimal, since the limit $\bar{\mf{u}}$ is Lipschitz continuous but not $C^1$. The uniform boundedness in $C^{0,1}_{\loc}$ entails a number of consequences (in fact, to derive some of these consequence the uniform H\"older estimates are sufficient): it is clear that, up to a subsequence, $\mf{v}_\beta \to \bar{\mf{u}}$ in $C^{0,\alpha}_{\loc}(\Omega)$, for every $\alpha \in (0,1)$. The convergence is also strong in $H^1_{\loc}$, and the limit $\bar{\mf{u}}$ satisfies
\beq\label{prop lim full}
\Delta \bar u_i = 0 \quad \text{in }\{u_i>0\}, \quad \text{and} \quad u_i \, u_j \equiv 0 \quad \text{in $\Omega$, for every $i \neq j$};
\eeq
furthermore, in view of the variational structure of problem \eqref{FSP}, the \emph{domain variation formula} (or \emph{local Pohozaev identity}) \eqref{loc Poh} holds whenever $B_r(x_0) \subset \Omega$. Notice that the domain variation formula is the same as the one arising in the partial segregation model, but the coupling with the condition of full or partial segregation results in significant differences in the way it can be exploited. In both cases, it is worth to remark that this identity keeps track of the interaction between the different components in the singular limit, an interaction that is lost at the pointwise level, and it has been the key ingredient to study the free boundary regularity in full segregation models \cite{CL08, TT12}: the free boundary $\Gamma:= \bigcup \pa\{\bar u_i>0\}$ and the \emph{nodal set} $\{\bar{\mf{u}}=0\}$ (set of points where all the components vanish) coincide; $\Gamma$ has Hausdorff dimension $N-1$ and, more precisely, it splits into a regular part, which is a collection of $C^{1,\alpha}$-hypersurfaces, and a singular set of dimension at most $N-2$ (optimal); the regular part $\mathcal{R}$ is characterized by the fact that for any point $x_0 \in \mathcal{R}$ there exists a radius $\rho>0$ such that exactly two components, say $\bar u_1$ and $\bar u_2$, do not vanish identically in $B_\rho(x_0)$, and $\bar u_1-\bar u_2$ is harmonic in $B_\rho(x_0)$. Finally, in dimension $N=2$ the singular set is locally finite, and the regular part consists in a locally finite collection of curves meeting with equal angles at singular points.

We also refer to \cite{Alp, SZ16} for further contributions such as sharp pointwise decay estimates for solutions to \eqref{FSP}, and fine properties of the singular set in higher dimension.

The aforementioned results admit an extension into more general contexts: they also hold for non-minimal, sign changing, solutions of systems with reaction terms and \emph{competing groups} of components. We refer the interested reader to \cite{STTZ16} for more details. Moreover, the same issues were considered for system \eqref{FSP} with the Laplacian replaced by the fractional Laplacian \cite{TVZ16, TVZ14, ToZi20}, for systems with \emph{long range interaction} \cite{CPQ, STTZ18, STZ23}, for asymmetric systems with different operators for each component \cite{ST23}, and for systems with non-variational interaction \cite{CTV05, CKL, TT12}, such as
\[
\begin{cases}
\displaystyle \Delta u_i = \beta u_i \sum_{j \neq i} u_j, \quad u_i>0 & \text{in }\Omega\\
u_i = \phi_i & \text{on }\partial \Omega.
\end{cases}
\]
In general, the results obtained depend on the particular problem considered, but a common feature is that all these models describe phenomena of full segregation, namely, at each point at most one component of the limiting profile can be non-zero.

\subsection*{Structure of the paper} In Section \ref{sec: Alm}, we prove some Almgren-type monotonicity formulas tailored for solutions to \eqref{P beta} and for functions in $\mathcal{L}_{\loc}(\Omega)$. The validity of the latter is a rather direct consequence of the domain variation formula \eqref{loc Poh}. In Section \ref{sec: fb initial study}, we start studying the structure of the nodal set of functions in $\mathcal{L}_{\loc}$, exploiting in particular their minimality properties. These properties alone are not sufficient to analyze the nodal set at triple points, but this is where Almgren's formula plays an essential role: thanks to it, we can show in Section \ref{sec: blow-up} the existence of homogeneous blow-ups at triple points. The classification of blow-ups and the results from Section \ref{sec: fb initial study} allow us to prove the structure theorem on the nodal set, Theorem \ref{thm: nodal set}, in Section \ref{sec: fb reg}. The validity of the improved uniform H\"older estimates, and the $C^{0,3/4}$ regularity of functions in $\mathcal{L}_{\loc}$ (i.e. the validity of Theorem \ref{thm: improved Holder bounds}) are the contents of Section \ref{sec: ihb}. In Section \ref{sec: fb dim 2}, we derive additional results regarding the free boundary of minimizing harmonic maps in dimension $N=2$, under suitable assumptions of the boundary data. By using these results, in Section \ref{sec: mer} we prove Theorem \ref{prop: mer min}.

\section{Almgren-type monotonicity formulae}\label{sec: Alm}

In this section we collect some preliminary properties.

\subsection{Almgren monotonicity formulae for solutions of systems}

Let $\Omega \subset \R^N$ be a domain, $M>0$, and let $\mf{v} \in H^1_{\loc}(\Omega) \cap C(\Omega)$, $\mf{v} \not \equiv 0$, satisfy
\beq\label{coex}
\begin{cases}
\Delta v_i = M v_i \prod_{j \neq i} v_j^2 & \text{in $\Omega$} \\
v_i \ge 0 & \text{in $\Omega$}.
\end{cases}
\eeq
For $x_0 \in \Omega$ and $r \in (0,\dist(x_0,\pa \Omega))$, we define
\[
\begin{split}
E(\mf{v},x_0,r) &:= \frac{1}{r^{N-2}} \int_{B_r(x_0)} \bigg(\sum_{j=1}^3 |\nabla v_j|^2 + M\prod_{j=1}^3 v_j^2\bigg)\,dx \\
H(\mf{v},x_0,r) &:= \frac{1}{r^{N-1}} \int_{S_r(x_0)} \sum_{j=1}^3 v_j^2\,d\sigma \\
N(\mf{v},x_0,r) &:= \frac{E(\mf{v},x_0,r)}{H(\mf{v},x_0,r)}.
\end{split}
\]
By standard elliptic theory, any $\mf{v}$ solving \eqref{coex} is smooth, and any non-trivial component is strictly positive by the maximum principle. Therefore, $H(\mf{v},x_0,r)>0$ whenever defined, and $E$, $H$, and $N$ are absolutely continuous with respect to $r$. Moreover,
\beq\label{E =H'}
E(\mf{v},x_0,r) \le \frac{1}{r^{N-2}} \int_{B_r(x_0)} \bigg(\sum_{j} |\nabla v_j|^2 + 3 M\prod_{j} v_j^2\bigg)\,dx = \frac{1}{r^{N-2}} \int_{S_r(x_0)} \sum_i v_i \partial_\nu v_i\,d\sigma,
\eeq
by the divergence theorem. In order to compute the derivative of $E$ with respect to $r$, we recall Lemma 6.1 from \cite{ST24p1}.

\begin{lemma}[Local Pohozaev identity]\label{lem: poh}
For every $x_0 \in \Omega$ and $r \in (0, \dist(x_0,\pa \Omega))$, we have that
\[
\begin{split}
r \int_{S_r(x_0)} \Big( \sum_{i=1}^3 |\nabla v_{i}|^2 + & M \prod_{j=1}^3 v_{j}^2 \Big)d\sigma = (N-2)\int_{B_r(x_0)} \sum_{i=1}^3 |\nabla v_{i}|^2\,dx \\
&+ N \int_{B_r(x_0)} M \prod_{j=1}^3 v_{j}^2\,dx + 2r \int_{S_r(x_0)} \sum_{i=1}^3 (\pa_\nu v_{i})^2\,d\sigma.
\end{split}
\]
\end{lemma}

From now on, we often use the notation $\prime=\frac{d}{dr}$.

\begin{proposition}\label{prop: alm coex}
Let $\mf{v} \not \equiv 0$ satisfy \eqref{coex} for some $M>0$. For every $x_0 \in \Omega$, the function $N(\mf{v},x_0,\,\cdot)$ is monotone non-decreasing. Moreover,
\beq\label{stima res N}
\int_{r_0}^r \frac{2}{s^{N-1}H(\mf{v},x_0,s)} \int_{B_s(x_0)} M \prod_{j=1}^3 v_j^2\,dx \, ds \le N(\mf{v},x_0,r)
\eeq
for every $0<r_0<r <\dist(x_0,\pa \Omega)$, and
\beq\label{der H}
\frac{d}{dr} \log H(\mf{v},x_0,r) = \frac{2}{r} N(\mf{v},x_0,r) + \frac{4}{r^{N-1}H(\mf{v},x_0,r)} \int_{B_r(x_0)} M \prod_{j=1}^3 v_j^2\,dx
\eeq
\end{proposition}

\begin{proof}
By direct computations, and using the subharmonicity of $v_i$, we have that
\[
H'(\mf{v},x_0,r) = \frac{2}{r^{N-1}} \int_{S_r(x_0)} \sum_i v_i \partial_\nu v_i\,d\sigma \ge 0.
\]
Moreover, by Lemma \ref{lem: poh},
\[
E'(\mf{v},x_0,r) = \frac{2}{r^{N-1}} \int_{B_r(x_0)} M \prod_j v_j^2\,dx + \frac2{r^{N-2}} \int_{S_r} \sum_i (\pa_\nu v_i)^2\,d\sigma.
\]
Therefore, recalling \eqref{E =H'},
\[
\begin{split}
N'(\mf{v}&,x_0,r) = \frac{1}{H^2(\mf{v},x_0,r)} \left[ E'(\mf{v},x_0,r) H(\mf{v},x_0,r)- E(\mf{v},x_0,r) H'(\mf{v},x_0,r) \right] \\
& \ge \frac{2}{r^{2N-3}H^2(\mf{v},x_0,r)} \left[  \left(\int_{S_r(x_0)} \sum_i (\pa_\nu v_i)^2\,d\sigma\right) \left(\int_{S_r(x_0)} \sum_i v_i^2\,d\sigma\right) -  \left(\int_{S_r(x_0)} \sum_i v_i \pa_\nu v_i\,d\sigma\right)^2\right] \\
& \quad \quad +\frac{2}{r^{N-1}H(\mf{v},x_0,r)} \int_{B_r(x_0)} M \prod_j v_j^2\,dx \\
& \ge \frac{2}{r^{N-1}H(\mf{v},x_0,r)} \int_{B_r(x_0)} M \prod_j v_j^2\,dx,
\end{split}
\]
by the Cauchy-Schwarz inequality, which proves the monotonicity of $N$. By integrating, we also deduce estimate \eqref{stima res N}. Furthermore, equation \eqref{der H} follows easily comparing the expression of $H'$ and \eqref{E =H'}.
\end{proof}

In what follows we collect some simple consequences of Proposition \ref{prop: alm coex}. Having established Proposition \ref{prop: alm coex}, the proofs of these properties are very similar to those given in \cite{BeTeWaWe, ST15} for systems appearing in total segregation models.

\begin{lemma}\label{lem: doubling coex}
Let $\mf{v}$ be a solution of \eqref{coex} for some $M>0$.
\begin{itemize}
\item[($i$)] Suppose that $N(\mf{v},x_0,r_0) \ge \alpha$. Then, for every $r_0 \le r_1 <r_2$
\[
\frac{H(\mf{v},x_0,r_2)}{r_2^{2\alpha}} \ge \frac{H(\mf{v},x_0,r_1)}{r_1^{2\alpha}}
\]
\item[($ii$)] Suppose that $N(\,\mf{v}\,,x_0,r_0) \le \gamma$. Then, for every $0<r_1<r_2 \le r_0$
\[
\frac{H(\mf{v},x_0,r_2)}{r_2^{2\gamma}} \le e^{2\gamma} \frac{H(\mf{v},x_0,r_1)}{r_1^{2\gamma}}.
\]
\end{itemize}
\end{lemma}

\begin{proof}
The result follows easily by integrating \eqref{der H}, using the monotonicity of $N$ and estimate \eqref{stima res N}.
\end{proof}

\begin{lemma}\label{lem: N infty coex}
Let $\mf{v} \not \equiv 0$ satisfy \eqref{coex} in $\R^N$ for some $M>0$.
\begin{itemize}
\item[($i$)] If $|\mf{v}(x)| \le C(1+|x|^\gamma)$ for every $x \in \R^N$, for some $\gamma>0$, then $N(\mf{v},x_0,r) \le \gamma$ for every $x_0 \in \R^{N\,}$ and $r>0$.
\item[($ii$)] $N(\mf{v},x_0,r) = 0$ for some $x_0 \in \R^N$ and $r>0$ if and only if at least one component $v_i$ vanishes identically.
\end{itemize}
\end{lemma}

\begin{proof}
($i$) The upper bound on $N$ follows easily from Lemma \ref{lem: doubling coex}. Indeed, suppose by contradiction that $N(\mf{v},x_0,\bar r) \ge \gamma +\eps$ for some $x_0 \in \R^N$ and $\eps, \bar r>0$. Then
\[
H(\mf{v},x_0,r) \ge \frac{H(\mf{v},x_0,\bar r)}{\bar r^{2(\gamma+\eps)}} r^{2(\gamma+\eps)} \quad \forall r \ge \bar r,
\]
in contradiction with the fact that $|\mf{v}(x)| \le C(1+|x|^\gamma)$ for every $x \in \R^N$.

\smallskip

\noindent ($ii$) If $N(\mf{v},x_0,r) = 0$, then $\mf{v}$ is constant and $v_1\,v_2\,v_3 =0$ in $B_r(x_0)$. Then, at least one component must vanish in $B_r(x_0)$, and hence in $\R^N$. On the other hand, suppose that one component $v_i \equiv 0$. Then the remaining ones are nonnegative harmonic functions in $\R^N$, thus constants, and hence $N(\mf{v},x_0,r)=0$ for every $x_0$ and $r$.
\end{proof}

\subsection{Almgren-type monotonicity formulae for partially segregated profiles}

In this subsection, we focus on functions in $\mathcal{L}_{\loc}(\Omega)$. Recall that any $\mf{v} \in \mathcal{L}_{\loc}(\Omega)$ is of class $C^{0,\alpha}(\Omega)$, for every $\alpha \in (0,\bar \nu)$, where $\bar \nu \in (0,2/3]$ is characterized in terms of an optimal ``overlapping partition" problem on the sphere, see \cite[Remark 1.3]{ST24p1}; moreover
\beq\label{bas prop 1}
\Delta v_i = 0 \quad \text{in $\{v_i>0\}$}, \qquad \text{and} \qquad
v_1\,v_2\,v_3 \equiv 0  \quad \text{in $\Omega$},
\eeq
and the domain variation formula \eqref{loc Poh} holds for $\mf{v}$:
\[
\int_{S_r(x_0)} \sum_i |\nabla v_i|^2\,d\sigma = \frac{N-2}{r} \int_{B_r(x_0)} \sum_i |\nabla v_i|^2\,dx + 2 \int_{S_r(x_0)} \sum_i (\pa_\nu v_i)^2\,d\sigma,
\]
whenever $B_r(x_0) \subset \Omega$. Starting from \eqref{loc Poh}, we can show that $\mf{v} \in \mathcal{L}_{\loc}(\Omega)$ inherits from $\{\mf{v}_n\}$ the validity of an Almgren-type monotonicity formula. Similarly to what we did in the previous subsection, for $x_0 \in \Omega$ and $r \in (0,\dist(x_0,\pa \Omega))$, we define
\[
\begin{split}
E(\mf{v},x_0,r) &:= \frac{1}{r^{N-2}} \int_{B_r(x_0)} \sum_{j=1}^3 |\nabla v_j|^2 \,dx \\
H(\mf{v},x_0,r) &:= \frac{1}{r^{N-1}} \int_{S_r(x_0)} \sum_{j=1}^3 v_j^2\,d\sigma\\
N(\mf{v},x_0,r) &:= \frac{E(\mf{v},x_0,r)}{H(\mf{v},x_0,r)}
\end{split}
\]
($N(\mf{v},x_0,r)$ being well defined whenever $H(\mf{v},x_0,r) \neq 0$).

\begin{proposition}\label{prop: alm seg}
Let $\mf{v} \in \mathcal{L}_{\loc}(\Omega)$. For every $x_0 \in \Omega$, the function $N(\mf{v},x_0,\,\cdot)$ is well defined for every $r \in (0, \dist(x_0,\pa \Omega))$, absolutely continuous, and monotone non-decreasing. Moreover,
\beq\label{der H seg}
\frac{d}{dr} \log H(\mf{v},x_0,r) = \frac{2}{r} N(\mf{v},x_0,r),
\eeq
and $N(\mf{v},x_0,r)=\gamma$ for every $r \in (r_0,r_1)$ if and only if $\mf{v}$ is $\gamma$-homogeneous with respect to $x_0$ in the annulus $A_{r_0,\,r_1\,}(x_0) :=B_{\,r_1\,}(x_0) \setminus \overline{B_{r_0}(x_0)}$.
\end{proposition}

\begin{proof}
The function $H(\mf{v},x_0,\,\cdot)$ is always nonnegative. For some $x_0 \in \R^N$, it is in fact strictly positive somewhere, say for $r \in (r_1,r_2)$, otherwise $\mf{v} \equiv 0$. Moreover, it is absolutely continuous, and for a.e. $r$
\beq\label{H' e E seg}
H'(\mf{v},x_0,r)= \frac{2}{r^{N-1}} \int_{S_r(x_0)} \sum_i v_i \pa_\nu v_i\,d\sigma = \frac{2}{r} E(\mf{v},x_0,r)
\eeq
(to prove this, one can argue by approximation as in \cite[Proposition 3.9]{NTTV10} by taking the limit in the analogue expressions for $\mf{v}_n$). This gives \eqref{der H seg}. Moreover, the domain variation formula gives
\[
E'(\mf{v},x_0,r) = \frac{2}{r^{N-2}} \int_{S_r(x_0)} \sum_i (\pa_\nu v_i)^2\,d\sigma.
\]
We can then compute the derivative of $N$ in $(r_1,r_2)$: for a.e. $r$ we obtain
\begin{multline*}
N'(\mf{v},x_0,r) = \frac{1}{H^2(\mf{v},x_0,r)} \left[ E'(\mf{v},x_0,r) H(\mf{v},x_0,r)- E(\mf{v},x_0,r) H'(\mf{v},x_0,r) \right] \\
 = \frac{2}{r^{2N-3}H^2(\mf{v},x_0,r)} \left[  \left(\int_{S_r(x_0)} \sum_i (\pa_\nu v_i)^2\,d\sigma\right) \left(\int_{S_r(x_0)} \sum_i v_i^2\,d\sigma\right) -  \left(\int_{S_r(x_0)} \sum_i v_i \pa_\nu v_i\,d\sigma\right)^2\right],
 \end{multline*}
which is nonnegative by the Cauchy-Schwarz inequality. At this point we can show that $H(\mf{v},x_0,r)>0$ for every $r>0$. Indeed, let
\[
r_3:= \inf\left\{r>0: H(\mf{v},x_0,s) >0 \quad \forall s \in \left(r,\frac{r_1+r_2}{2}\right)\right\}.
\]
It is clear that $r_3 \le r_1$. If $r_3>0$, then by continuity $H(\mf{v},x_0,r_3) = 0$, and, by integrating \eqref{der H seg} on $(r_3+\eps,r)$ for $r \in (r_3,r_2)$ and using the monotonicity of $N$, we find
\[
\frac{H(\mf{v},x_0,r)}{H(\mf{v},x_0,r_3+\eps)} \le \left(\frac{r}{r_3+\eps}\right)^{2N(\mf{v},x_0,r)}.
\]
By taking the limit as $\eps \to 0^+$, the left hand side tends to $+\infty$, while the right hand side remains bounded, a contradiction. Thus, $H(\mf{v},x_0,r)>0$ for every $x \in \Omega$ and $r \in (0,\dist(x_0,\pa \Omega))$, and it is not difficult to deduce that $H(\mf{v},x,r)>0$ for every $x \in \Omega$ and $r \in (0,\dist(x,\pa \Omega))$.

It remains to consider the case when $N(\mf{v},x_0,r)=\gamma$ for every $r \in (r_0,r_1)$. Then $N'=0$ on $\,(r_0,r_1)$, so that, by the expression of $N'$, we must have equality in the Cauchy-Schwarz inequality, namely
\[
\nabla v_i \cdot \frac{x-x_0}{|x-x_0|} = \pa_\nu v_i = \lambda(r) v_i \quad \text{on $S_r(x_0)$},
\]
for every $i=1,2,3$ and a.e. $r \in \,(r_0,r_1)$. By substituting this expression in the definition of $N$, thanks to \eqref{H' e E seg} we infer that
\[
\gamma = N(\mf{v},x_0,r) = \frac{r\lambda(r) \int_{S_r(x_0)} \sum_i v_i^2\,d\sigma}{\int_{S_r(x_0)} \sum_i v_i^2\,d\sigma} = r \lambda(r).
\]
Therefore
\[
\nabla v_i \cdot (x-x_0) = r \lambda(r) v_i = \gamma v_i \quad \text{in $A_{\,r_0,r_1}(x_0)$}
\]
and the $\gamma$-homogeneity follows by the Euler theorem for homogeneous functions.
\end{proof}

Useful consequences of the monotonicity formula are collected in the next statements, which are the counterpart of those in the previous subsection. The proofs are analogue, and hence are omitted.

\begin{lemma}\label{lem: doubling seg}
Let $\mf{v} \in \mathcal{L}_{\loc}(\Omega)$.
\begin{itemize}
\item[($i$)] Suppose that $N(\mf{v},x_0,r_0) \ge \alpha$. Then, for every $r_0 \le r_1 <r_2$
\[
\frac{H(\mf{v},x_0,r_2)}{r_2^{2\alpha}} \ge \frac{H(\mf{v},x_0,r_1)}{r_1^{2\alpha}}
\]
\item[($ii$)] Suppose that $N(\mf{v},x_0,r_0) \le \gamma$. Then, for every $0<r_1<r_2 \le r_0$
\[
\frac{H(\mf{v},x_0,r_2)}{r_2^{2\gamma}} \le \frac{H(\mf{v},x_0,r_1)}{r_1^{2\gamma}}.
\]
\end{itemize}
\end{lemma}

\begin{lemma}\label{lem: N infty seg}
Let $\mf{v} \in \mathcal{L}_{\loc}(\R^N)$. If $|\mf{v}(x)| \le C(1+|x|^\gamma)$ for every $x \in \R^N$, for some $\gamma>0$, then $N(\mf{v},x_0,r) \le \gamma$ for every $x_0 \in \R^{N\,}$ and $r>0$.
\end{lemma}

\section{Free boundary regularity: preliminary results}\label{sec: fb initial study}

Let $\mf{v} \in \mathcal{L}_{\loc}(\Omega)$. In this section we start studying the properties of the nodal set $\Gamma_{\mf{v}}$, defined in \eqref{def fb} as the set of points where at least two components vanish.

Our first results are direct consequences of the Almgren-type monotonicity formula.

\begin{lemma}\label{lem: 3 int vuo}
If $\mf{v} \in \mathcal{L}_{\loc}(\Omega)$, then the set $\Gamma_{\mf{v}}^3$ has empty interior.
\end{lemma}
\begin{proof}
It is a consequence of Proposition \ref{prop: alm seg}: if $m(x_0)=3$, since $H(\mf{v},x_0,r)>0$ for every $r>0$, every sphere $S_r(x_0)$ contains points where some component does not vanish.
\end{proof}

Now we prove a lower bound for the frequency at triple points. To this end, we recall from \cite[Theorem 1.7]{ST24p1} that any function in $\mathcal{L}_{\loc}(\Omega)$ is of class $C^{0,\alpha}(\Omega)$ for every $\alpha \in (0, \bar \nu)$, for some $\bar \nu \in (0,2/3]$ depending only on the dimension.

\begin{lemma}\label{lem: N >=}
Let $\mf{v} \in \mathcal{L}_{\loc}(\Omega)$. If $m(x_0) =3$, then $N(\mf{v},x_0,0^+) \ge \bar \nu$.
\end{lemma}

\begin{proof}
Suppose by contradiction that $N(\mf{v},x_0,0^+) <\bar \nu$. Then, by monotonicity, there exist $\alpha < \bar \nu$ and $\bar r, \eps>0$ small enough such that $N(\mf{v},x_0,r) \le\alpha-\eps$ for every $r \in (0, \bar r]$. Thus, by Lemma \ref{lem: doubling seg},
\[
H(\mf{v},x_0,r) \ge \frac{H(\mf{v},x_0,\bar r)}{\bar r^{2(\alpha-\eps)}} r^{2(\alpha-\eps)} = C r^{2(\alpha-\eps)}
\]
for every such $r$. On the other hand, since $m(x_0)=3$ and $\mf{v} \in \mathcal{C}^{0,\alpha}(\Omega)$, we also have that
\[
H(\mf{v},x_0,r) \le C r^{2\alpha} \quad \forall r>0,
\]
and the previous estimates are in contradiction for small $r$.%
\end{proof}

Next, we collect some useful properties coming from the minimality of the approximating sequence $\{\mf{v}_n\}$. At first we establish that, if $v_j$ vanishes identically in a ball, and therein another component $v_i$ is non-trivial, then $v_i$ is in fact strictly positive.

\begin{lemma}\label{lem: no self}
Let $\mf{v} \in \mathcal{L}_{\loc}(\Omega)$. Suppose that there exist $x_0 \in \R^N$ and $R>\rho>0$ such that the following hold: $v_{j} =0$ in $B_R(x_0)$, and $v_i|_{S_{\rho}(x_0)} \not \equiv  0$. Then $v_i >0$ in $B_\rho(x_0)$.
\end{lemma}

\begin{proof}
The proof is completely analogue to the one of \cite[Lemma 4.9]{ST24p1}.
\end{proof}

As a consequence:

\begin{lemma}\label{cor: comp 0 seg}
For $\mf{v} \in \mathcal{L}_{\loc}(\R^N)$, we have that $N(\mf{v},x_0,+\infty) = 0$ for some $x_0 \in \R^N$ if and only if at least one component $v_i$ vanishes identically.
\end{lemma}
\begin{proof}
If $N(\mf{v},x_0,+\infty) = 0$, then $N(\mf{v},x_0,r) = 0$ for every $r>0$, by monotonicity. Hence, $\mf{v}$ is constant in $\R^N$, and to satisfy the partial segregation condition it is necessary that at least one component $v_i$ vanishes identically. Regarding the reverse implication, suppose that one component $v_i \equiv 0$. Let $x_0 \in \R^N$ be arbitrarily chosen. Since $H(\mf{v},x_0,r)>0$ in $\R^N$ for every $r>0$, at least one component $v_j$ has to be nontrivial on $S_r(x_0)$. We claim that $v_j>0$ in $\R^N$. Indeed, since $v_j$ is subharmonic and $v_j \not \equiv 0$ in $B_r(x_0)$, we deduce that $v_j \not \equiv 0$ on $S_\rho(x_0)$, for every $\rho \ge r$. Thus, we can apply Lemma \ref{lem: no self} in any ball $B_\rho(x_0)$, and the claim follows. To sum up, we can divide the components of $\mf{v}$ into two classes: the trivial ones, and the strictly positive ones. The latter are positive harmonic functions in $\R^N$, and hence are constants. Then, all the components of $\mf{v}$ are constants, and $N(\mf{v},x_0,r) = 0$ for every $r>0$ and $x_0 \in \R^N$.
\end{proof}

Secondly, we establish that in a neighborhood of any point of multiplicity $2$, the difference of the two vanishing components is harmonic.

\begin{lemma}\label{prop: diff arm}
Let $\mf{v} \in \mathcal{L}_{\loc}(\Omega)$. Let $x_0 \in \Gamma_{\mf{v}}^{2}$ be such that $v_i(x_0) = 0 = v_j(x_0)$. Then $v_i-v_j$ is harmonic in a neighborhood of $x_0$, and in fact $v_i$ and $v_j$ are the positive and negative part of $v_i-v_j$ in such neighborhood, respectively.
\end{lemma}
\begin{proof}
For concreteness, we suppose that $v_1(x_0) = v_2(x_0) = 0$, and $v_3(x_0)>0$. Then, by continuity, $v_3 \ge c$ in $B_R(x_0)=:B$ for some $c,R>0$, and the partial segregation condition implies that $v_1\,v_2 \equiv 0$ in $B$. We prove that
\beq\label{arm diff}
\int_{B} \nabla (v_1-v_2)\cdot \nabla \varphi \,dx \ge 0 \qquad \forall \varphi \in C^1_c(B), \ \varphi \ge 0.
\eeq
Exchanging the role of $v_1$ and $v_2$, also the opposite inequality holds, and the thesis follows.

Assume by contradiction that \eqref{arm diff} does not hold. Then there exists a nonnegative $\varphi \in C^1_c(B)$, say $\textrm{supp}\, \varphi =B_\rho(x_0)$ with $\rho<R$, such that
\[
\int_{B} \nabla (v_1-v_2)\cdot \nabla \varphi \,dx < 0.
\]
It is not difficult to deduce that
\beq\label{capp}
\int_{B_\rho(x_0)} \left(|\nabla (v_1-v_2+t \varphi)|^2- |\nabla v_1|^2-|\nabla v_2|^2\right) dx =: -2\delta<0
\eeq
for $t=\bar t>0$ sufficiently small (notice that $\nabla v_1 \cdot \nabla v_2 =0$ a.e. in $B$, since $v_1v_2 \equiv 0$). We use this estimate to obtain a contradiction with the minimality of the approximating sequence $\{\mf{v}_n\}$, given by Definition \ref{def: L}. Indeed, let $\eta \in C^1_c(B)$ be such that $0 \le \eta \le 1$, $\eta = 1$ in $S_R(x_0)$, $\eta \equiv 0$ in $B_\rho(x_0)$, and let us consider the following competitor:
\[
\tilde{\mf{v}}_n:= \eta \mf{v}_n + (1-\eta)( (v_{1,n}-v_{2,n}+ \bar t \varphi)_+, (v_{1,n}-v_{2,n}+ \bar t \varphi)_-, v_{3,n}).
\]
Notice that $\tilde{\mf{v}}_n = \mf{v}_n$ on $S_R(x_0)$, so that by minimality
\[
J_{M_n}(\mf{v}_n, B) \le J_{M_n}(\tilde{\mf{v}}_n,B).
\]
On the other hand
\beq\label{capp 2}
\begin{split}
J_{M_n}(\tilde{\mf{v}}_n,B) & - J_{M_n}(\mf{v}_n, B) = \int_{B_R(x_0) \setminus B_\rho(x_0)} \sum_i \left(|\nabla \tilde v_{i,n}|^2  - |\nabla v_{i,n}|^2\right)\,dx \\
& \qquad+  \int_{B_R(x_0) \setminus B_\rho(x_0)} M_n \Big(\prod_j  \tilde v_{j,n}^2-   \prod_j  v_{j,n}^2\Big)dx - M_n \int_{B_\rho(x_0)} \prod_j  v_{j,n}^2\,dx \\
& \qquad + \int_{B_\rho(x_0)} \left(|\nabla (v_{1,n}-v_{2,n}+\bar t \varphi)|^2- |\nabla v_{1,n}|^2-|\nabla v_{2,n}|^2\right) dx .
\end{split}
\eeq
On $B_R(x_0) \setminus B_\rho(x_0)$, where $\varphi \equiv 0$, we have
\begin{align*}
\tilde v_{1,n} &= \eta v_{1,n} + (1-\eta) (v_{1,n}-v_{2,n})_+ \to v_1, \quad \text{and} \\
\tilde v_{2,n} &= \eta v_{2,n} + (1-\eta)(v_{1,n}-v_{2,n})_- \to v_2,
\end{align*}
strongly in $H^1(B_R(x_0) \setminus \overline{B_\rho(x_0)})$, thanks to the segregation condition $v_1 \, v_2 \equiv 0$. Therefore, the first integral on the right hand side in \eqref{capp 2} tends to $0$ as $n \to \infty$. Regarding the second integral, we observe that in $B_R(x_0) \setminus B_\rho(x_0)$
\begin{align*}
\tilde v_{1,n} & = \eta v_{1,n} + (1-\eta) (v_{1,n}-v_{2,n})_+ \le v_{1,n},\\
\tilde v_{2,n} & = \eta v_{2,n} + (1-\eta) (v_{1,n}-v_{2,n})_- \le v_{2,n}, \\
\tilde v_{3,n} & = v_{3,n}.
\end{align*}
Therefore, the second integral on the right hand side in \eqref{capp 2} is nonpositive. Recalling also \eqref{capp}, we conclude that
\[
J_{M_n}(\tilde{\mf{v}}_n,B)   - J_{M_n}(\mf{v}_n, B) \le o(1)-2\delta <-\delta<0
\]
for sufficiently large $n$, which is the desired contradiction.
\end{proof}

\begin{lemma}\label{cor: diff arm}
Let $\mf{v} \in \mathcal{L}_{\loc}(\Omega)$. Then $v_i-v_j$ is harmonic in $\{v_k>0\}$, for all pairwise distinct indices $i, j, k$.
\end{lemma}
\begin{proof}
It is a direct consequence of \eqref{bas prop 1} and Lemma \ref{prop: diff arm}.
\end{proof}

\begin{lemma}\label{cor: on N m12}
If $m(x_0)=2$, then $N(\mf{v},x_0,0^+)=0$.
\end{lemma}

\begin{proof}
For concreteness, suppose that $v_3(x_0) > 0$. By Lemma \ref{prop: diff arm}, we have that $w = v_1-v_2$ is harmonic in a ball $B_R(x_0)$, and the same holds for $v_3$. Thus, both $v_3$ and $w$ are of class $C^1$ in a neighborhood of $x_0$, and the mean value theorem for integrals implies that
\[
\begin{split}
E(\mf{v},x_0,r) &= \frac{1}{r^{N-2}} \int_{B_r(x_0)} \left(|\nabla v_3|^2+ |\nabla w|^2\right)\,dx \\
&= C r^2\left( |\nabla v_3(x_0)|^2+ |\nabla w(x_0)|^2 + o(1)\right) \to 0
\end{split}
\]
as $r \to 0^+$. On the other hand $H(\mf{v},x_0,r) \to C v_3^2(x_0)>0$.
\end{proof}

\begin{remark}\label{rem: on N}
If $m(x_0)=1$, then one component vanishes identically in a neighborhood of $x_0$, and the remaining ones are positive harmonic functions in such neighborhood. Thus, by direct computations as in the previous proof we have that $N(\mf{v},x_0,0^+)=0$. Therefore, $N(\mf{v},x_0,0^+)=0$ whenever the multiplicity of $x_0$ is smaller than $3$.
\end{remark}

\begin{lemma}\label{lem: 3 closed}
The function $x_0 \in \Gamma_{\mf{v}} \mapsto N(\mf{v},x_0,0^+)$ is upper semi-continuous, and $\Gamma_{\mf{v}}^3$ is a relatively closed set of $\Gamma_{\mf{v}}$.
\end{lemma}

\begin{proof}
The value $N(\mf{v},x_0,0^+)$ is upper semi-continuous with respect to $x_0$, since $N(\mf{v},x_0,0^+)$ is defined as the infimum of continuous functions. Thanks to Lemma \ref{lem: N >=} and Remark \ref{rem: on N}, the upper semi-continuity gives the rest of the thesis.
\end{proof}

Next, we establish that if two components are segregated in a ball, then the third one is harmonic in a smaller ball.

\begin{lemma}\label{lem: no |x|}
Let $\mf{v} \in \mathcal{L}_{\loc}(\Omega)$. Let $x_0 \in \Gamma_{\mf{v}}$ be such that two components, say $v_1, v_2$, satisfy $v_1\,v_2 \equiv 0$ in $B_R(x_0)$. Then $v_3$ is harmonic in $B_{R/2}$.
\end{lemma}

\begin{proof}
Without loss, let $x_0=0$. Suppose by contradiction that $v_3$ is not harmonic in $B_{R/2}$. Let $\{\mf{v}_n\}$ and $M_n \to +\infty$ be given by Definition \ref{def: L}, and let
\[
c_n:= J_{M_n}(\mf{v}_n, B_R) = \min \left\{ J_{M_n}(\mf{v}, B_R) : \mf{v} \in H^1(B_R, \R^3) \text{ such that } \mf{v}-\mf{v}_n \in H_0^1(B_R, \R^3)\right\}.
\]
We provide a competitor of ${\mf{v}_n}$ decreasing the energy, which is a contradiction. To this aim, let $\eta \in C^1(\overline{B_R})$ be a radial cut-off function with the following properties: $0 \le \eta \le 1$, $\eta$ is radially increasing, $\eta \equiv 0$ in $B_{R/2}$, $\eta =1$ on $S_{R}$. We set
\[
\tilde v_{1,n}:= \eta v_{1,n} + (1-\eta) (v_{1,n}-v_{2,n})_+, \quad \tilde v_{2,n}:= \eta v_{2,n}+ (1-\eta)(v_{1,n}-v_{2,n})_-,
\]
and
\[
\tilde v_{3,n}:= \begin{cases} v_{3,n} & \text{in $B_R \setminus B_{R/2}$} \\
\text{harmonic extension of $v_{3,n}$} & \text{in $B_{R/2}$}.
\end{cases}
\]
Notice that, since $v_1\,v_2 \equiv 0$ in $B_R$ by assumption, $(v_{1,n}-v_{2,n})_+ \to v_1$ and $(v_{1,n}-v_{2,n})_- \to v_2$ strongly in $H^1(B_R)$. Therefore, $\tilde v_{1,n} \to v_1$ and $\tilde v_{2,n} \to v_2$ strongly in $H^1$ as well, and in particular
\beq\label{2051}
\Big| \int_{B_R}\left(|\nabla \tilde v_{1,n}|^2 - |\nabla v_{1,n}|^2\right)dx\Big|+\Big|  \int_{B_R}\left(|\nabla \tilde v_{2,n}|^2 - |\nabla v_{2,n}|^2\right)dx \Big|\to 0.
\eeq
Moreover, $\tilde v_{1,n} \le v_{1,n}$ and $\tilde v_{2,n} \le v_{2,n}$ in $B_R$, and $\tilde v_{1,n}\, \tilde v_{2,n} \equiv 0$ in $B_{R/2}$, whence
\beq\label{2052}
0\le \tilde v_{1,n} \,\tilde v_{2,n} \, \tilde v_{3,n}  \le   v_{1,n} \, v_{2,n} \,  v_{3,n} \quad \implies \quad M_n \int_{B_R} \prod_{i=1}^3 \tilde v_{i,n}^2\,dx \to 0.
\eeq
And, finally, $\tilde v_{3,n} \to w_3$ strongly in $H^1(B_R)$, where $w_3$ is defined by
\[
w_{3}:= \begin{cases} v_{3} & \text{in $B_R \setminus B_{R/2}$} \\
\text{harmonic extension of $v_{3}$} & \text{in $B_{R/2}$}.
\end{cases}
\]
Since we are assuming that $v_3$ is not harmonic in $B_{R/2}$, from the fact that the Dirichlet integral of $w_3$ is strictly smaller than the one of $v_3$ we deduce that there exists $\delta>0$ such that
\beq\label{2053}
\int_{B_R} \left(|\nabla \tilde v_{3,n}|^2 - |\nabla v_{3,n}|^2\right)\,dx < - \delta
\eeq
for every $n$ large. To summarize, thanks to formulas \eqref{2051}-\eqref{2053}, we conclude that
\[
0\,\le \,J_{M_n}(\tilde{\mf{v}}_n, B_R) -J_{M_n}(\mf{v}_n, B_R) <-\delta +o(1)
\]
as $n \to \infty$, a contradiction.
\end{proof}

\section{Blow-up analysis}\label{sec: blow-up}

In order to study the free boundary regularity for functions in $\mathcal{L}_{\loc}(\Omega)$, we use, as customary, local techniques based on a blow-up analysis. Let $x_0 \in \Gamma_{\mf{v}}^3$, namely $\mf{v}(x_0)=0$, let $\rho_m \to 0^+$, and let us introduce the \emph{blow-up sequence}
\[
\mf{v}_m(x):= \frac1{H(\mf{v},x_0,\rho_m)^{1/2}} \mf{v}(x_0+\rho_m x), \quad x \in \Omega_m:= \frac{\Omega-x_0}{\rho_m}.
\]
Notice that $\|\mf{v}_m\|_{L^2(S_1)}=1$ for every $m$, and that the sets $\Omega_m$ exhaust $\R^N$ as $m \to \infty$. In this section we study the asymptotic behavior of $\{\mf{v}_m\}$.

\begin{remark}\label{rem: on measures}
We know that $v_i$ is harmonic in $\{v_i>0\}$, and moreover $v_i$ is subharmonic in $\Omega$, since it is the limit in $H^1_{\loc}(\Omega)$ of subharmonic functions. Therefore, we can write
\[
\Delta v_i = \mu_i \quad \text{in $\Omega$}
\]
in the sense of distributions, where $\mu_i$ is a non-negative Radon measure supported on $\pa\{v_i>0\}$.
\end{remark}

\begin{remark}\label{rem: bu m2}
Let $v_1(x_0)=0=v_2(x_0)$, with $v_3(x_0) > 0$. Then it is plain that $\mf{v}_m \to (0,0,c)$ for some positive constant $c>0$ as $m \to \infty$. Analogously, if $v_1(x_0)=0$ and $v_2(x_0), v_3(x_0) >0$, then $\mf{v}_m \to (0,c_1,c_2)$ for some $c_1,c_2>0$.
Therefore, blow-up limits of points with multiplicity $1$ and $2$ are well-understood.
\end{remark}

In what follows we focus on the case when $m(x_0) =3$. By scaling, we have
\beq\label{eq mis}
\Delta v_{i,m} = \mu_{i,m} \quad \text{and} \quad v_{1,m} \,v_{2,m}\, v_{3,m} \equiv 0 \quad \text{in } \Omega_m,
\eeq
where
\[
\mu_{i,m}(E) = \frac{1}{H(\mf{v},x_0,\rho_m)^{1/2} \rho_m^{N-2}} \mu_i(x_0+\rho_m E) \quad \text{for every Borel set $E \subset \R^N$},
\]
is a non-negative Radon measure supported on $\pa \{v_{i,m}>0\}$.

\begin{lemma}\label{lem: inv sc}
For each $m$ fixed, the function $\mf{v}_m \in \mathcal{L}_{\loc}(\Omega_m)$.
\end{lemma}
\begin{proof}
Let $\{\mf{v}_n\}$ and $M_n \to +\infty$ be the two sequences given by Definition \ref{def: L}. Defining
\[
\mf{v}_{m,n}(x):= \frac{1}{H(\mf{v}_n,x_0,\rho_m)^{1/2}} \mf{v}_n(x_0+\rho_m x),
\]
it is clear that, for every $m$ fixed, $\mf{v}_{m,n} \to \mf{v}_m$ locally uniformly and in $H^1_{\loc}(\Omega_m)$, as $n \to \infty$. Each function $\mf{v}_{m,n}$ minimizes a scaled functional, with competition parameter
\beq\label{def mmn}
M_{m,n}:= \rho_m^2 H(\mf{v}_n, x_0,\rho_m)^2 M_n,
\eeq
with respect to variations with compact support; we observe that $M_{m,n} \to +\infty$ as $n \to \infty$, for each $m$ fixed: indeed, $M_n \to +\infty$ and
\[
\lim_{n \to \infty} H(\mf{v}_n, x_0,\rho_m) = H(\mf{v}, x_0,\rho_m)=: c_m >0 \quad \implies \quad H(\mf{v}_n, x_0,\rho_m) \ge \frac12 c_m
\]
for large $n$, where we used the fact that $H(\mf{v},x_0,r)>0$ for every $r$. Finally, for every $\omega \ssubset \Omega_m$, by scaling
\[
M_{m,n} \int_{\omega} \prod_{i=1}^3 v_{i,m,n}^2\,dx \le \frac{2 \rho_m^{2-N} M_n}{H(\mf{v},x_0,\rho_m)} \int_{x_0+\rho_m \omega} \prod_{i=1}^3 v_{i,n}^2\,dx,
\]
and the latter tends to $0$ as $n \to \infty$.
\end{proof}

The main result of this section is the following:

\begin{proposition}\label{prop: conv blow-up}
Let $x_0 \in \Gamma_{\mf{v}}^3$, and let $\gamma = N(\mf{v},x_0,0^+)$. As $m \to \infty$
\[
\mf{v}_m \to \mf{w} \quad \text{in $C^{0,\alpha}_{\loc}(\R^N)$ and strongly in $H^1_{\loc}(\R^N)$},
\]
for every $\alpha \in (0, \bar \nu)$, up to a subsequence, where $\mf{w} \in \mathcal{L}_{\loc}(\R^N)$ is $\gamma$-homogeneous with respect to $0$. Moreover, all the components of $\mf{w}$ are non-trivial.
\end{proposition}

The proof is divided into a series of lemmas.

\begin{lemma}\label{lem: bdd H1 mis}
For every $R>0$, there exists $C>0$ such that $\|\mf{v}_m\|_{H^1(B_R)} \le C$, $\|\mf{v}_m\|_{L^\infty(B_R)} \le C$ and $\max_i \mu_{i,m}(B_R) \le C$ for every $m$.
\end{lemma}

\begin{proof}
Let $2R_0:= \dist(x_0,\pa \Omega)$ and $R>1$. By Proposition \ref{prop: alm seg}, for $m$ so large such that $\rho_m R \le R_0$ we have that
\[
N(\mf{v}_m,0,R) = N(\mf{v},x_0,\rho_m R) \le N(\mf{v},x_0,R_0)=: \gamma' <+\infty.
\]
Therefore, by Lemma \ref{lem: doubling seg},
\beq\label{H m bdd}
\frac{1}{R^{N-1}} \int_{S_R} \sum_i v_{i,m}^2\,d\sigma = H(\mf{v}_m,0,R) \le H(\mf{v}_m,0,1) R^{2\gamma'} =R^{2\gamma'},
\eeq
and moreover
\beq\label{grad m bdd}
\frac{1}{R^{N-2}} \int_{B_R} \sum_i |\nabla v_{i,m}|^2\,dx = N(\mf{v}_m,0,R) H(\mf{v}_m,0,R) \le \gamma' R^{2\gamma'}.
\eeq
Estimates \eqref{H m bdd} and \eqref{grad m bdd} ensure that $\{\mf{v}_m\}$ is bounded in $H^1(B_R)$, as claimed. As $R>1$ was arbitrarily chosen, we deduce that in fact $\{\mf{v}_m\}$ is bounded in $H^1_{\loc}(\R^N)$. Furthermore, since each $v_{i,m}$ is subharmonic, $\{\mf{v}_m\}$ is also bounded in $L^\infty_{\loc}(\R^N)$.

Finally, let $i \in \{1,2,3\}$. We test the equation for $v_{i,m}$ with a cut-off function $\varphi \in C^\infty_c(B_{2R})$ such that $0\le \varphi \le 1$, and $\varphi \equiv 1$ in $B_R$. We obtain
\[
\mu_{i,m}(B_R) \le \int_{B_{2R}} \varphi \,d\mu_{i,m} = - \int_{B_{2R}} \nabla v_{i,m} \cdot \nabla \varphi\,dx \le C \|\nabla v_{i,m}\|_{L^2(B_{2R})} \le C,
\]
since $\{\mf{v}_m\}$ is bounded in $H^1(B_{2R})$.
\end{proof}

The lemma implies that, up to a subsequence, there exist $\mf{w} \in H^1_{\loc}(\R^N)$ and a nonnegative Radon measure $\bar \mu_i$ such that: $\mf{v}_m \rightharpoonup \mf{w}$ weakly in $H^1_{\loc}(\R^N)$, $\mu_{i,m} \overset{\ast}{\rightharpoonup} \bar \mu_i$ weakly star in the sense of measures, and moreover
\[
\Delta w_i = \bar \mu_i, \quad w_i \ge 0, \quad w_1\,w_2\,w_3 \equiv 0 \quad \text{a.e. in $\R^n$}.
\]
The next step consists in showing that, up to a further subsequence, $\mf{v}_m \to \mf{w}$ locally uniformly.

\begin{lemma}\label{lem: decay1L}
Let $\mf{v} \in \mathcal{L}_{\loc}(\Omega)$, $\alpha \in (0,\bar \nu)$, and $R>0$. There exists $C>0$ and $\bar m \in \N$ such that
\beq\label{1350}
\frac1{r^{N-1}} \int_{S_r(\bar x_0)} \sum_i v_{i,m}^2\,d\sigma \le C r^{2\alpha}
\eeq
for every $\bar x_0 \in \Gamma_{\mf{v}_m}^3 \cap B_R(0)$, $r \in (0, R)$, and for $m>\bar m$.%
\end{lemma}
\begin{proof}
Let $\bar m$ so large that $B_{2R} \subset \Omega_m$ for $m>\bar m$. By Lemma \ref{lem: N >=} we have that $N(\mf{v}_m,\bar x_0,0^+) \ge \alpha$ for every $\bar x_0 \in \Gamma_{\mf{v}_m}^3 \cap B_R(0)$. Thus, by Lemma \ref{lem: doubling seg} we deduce that for every $r \in (0,R)$,
\[
\frac1{r^{N-1}} \int_{S_r(\bar x_0)} \sum_i v_{i,m}^2\,d\sigma \le \frac{H(\mf{v}_m,\bar x_0,R)}{R^{2\alpha}} r^{2\alpha} \le C r^{2\alpha},
\]
where we used that $\|\mf{v}_m\|_{L^\infty(B_{2R})} \le C$ for the last estimate.
\end{proof}

\begin{lemma}\label{lem: bu hol}
Let $\mf{v} \in \mathcal{L}_{\loc}(\Omega)$, $\alpha \in (0,\bar \nu)$, and $R>0$. There exists $C>0$ such that $\|\mf{v}_m\|_{C^{0,\alpha}(\overline{B_R})} \le C$.
\end{lemma}

\begin{proof}
Since $\|\mf{v}_m\|_{L^\infty(B_R)} \le C$, we focus now on the boundedness of the H\"older seminorm
\[
[\mf{v}_m]_{C^{0,\alpha}(\overline{B_R})} = \max_{i=1,2,3} \ \sup_{\substack{x,y \in \overline{B_R} \\ x \neq y}} \frac{|v_{i,m}(x)-v_{i,m}(y)|}{|x-y|^\alpha}.
\]
Recall that $\mf{v} \in C^{0,\alpha+\eps}(\Omega)$, provided that $\alpha+\eps<\bar \nu$. Thus, up to a relabelling and up to a subsequence, for each $m$ there exists $x_m, y_m \in \overline{B_R}$ such that
\[
[\mf{v}_m]_{C^{0,\alpha}(\overline{B_R})} = \frac{|v_{1,m}(x_m)-v_{1,m}(y_m)|}{|x_m-y_m|^\alpha}.
\]
Let now $r_m:= |x_m-y_m|$, and, without loss of generality,
\[
2R_m:= \max\{\dist(x_m,\Gamma_{\mf{v}_m}^3), \ \dist(y_m,\Gamma_{\mf{v}_m}^3)\} = \dist(x_m,\Gamma_{\mf{v}_m}^3).
\]
{If $R_m=0$, both points are triple points and their H\"older quotient is zero. We may therefore assume $R_m>0$.}

We prove the boundedness of $[\mf{v}_m]_{C^{0,\alpha}(\overline{B_R})}$ in different ways, according to the asymptotic properties of $\{r_m\}$ and $\{R_m\}$.

\smallskip

\emph{Case 1)} $r_m \ge c >0$. \\
\noindent By the boundedness in $L^\infty$, we have
\[
[\mf{v}_m]_{C^{0,\alpha}(\overline{B_R})} = \frac{|v_{1,m}(x_m)-v_{1,m}(y_m)|}{r_m^\alpha} \le \frac{2C}{c}.
\]

\smallskip

\emph{Case 2)} $r_m \to 0$ and $R_m \ge c>0$. \\
\noindent By definition of $R_m$, in the ball $B_{c}(y_m)$ there are no triple points, namely the only free-boundary points (if any) of $\Gamma_{\mf{v}_m} \cap B_c(y_m)$ are points of multiplicity $2$, and, up to a subsequence the following alternative holds: either $v_{1,m}>0$ in $B_c(y_m)$, or $v_{1,m}$ vanish somewhere. Therefore, by Lemma \ref{prop: diff arm}, we can construct a harmonic function $w_m$ in the following way: in the connected component $\omega_{1,m}$ of $\{v_{1,m}>0\} \cap B_c(y_m)$ such that $x_m, y_m \in \overline{\omega_{1,m}}$ (it is not difficult to see that, if $x_m$ and $y_m$ belong to different components, then they cannot achieve the maximum of the H\"older quotient), we set $w_m=v_{1,m}$; then we consider a component of the boundary of $\omega_{1,m} \cap B_c(y_m)$, which we call $\sigma_{1,m}$. Since $\sigma_{1,m}$ consists of double points, there exists an index $i \in \{2,3\}$, say $i=2$, such that $\sigma_{1,m}$ separates $\omega_{1,m}$ and a connected component $\omega_{2,m}$ of $\{v_{2,m}>0\} \cap B_c(y_m)$, while $v_{3,m}>0$ on $\sigma_{1,m}$. We set $w_m=0$ on $\sigma_{1,m}$, and $w=-v_{2,m}$ on $\omega_{2,m}$. Iterating the argument in all the adjacent components to $\omega_{1,m}$, $\omega_{2,m}$, \dots, we define $w_m$ as a harmonic function in $B_c(y_m)$, coinciding with $v_{1,m}$ in $\overline{\omega_{1,m}}$.

Notice that $x_m \in B_{c/2}(y_m)$ for large $m$. Therefore, to complete the proof in this case we can estimate $[v_{1,m}]_{C^{0,\alpha}(\overline{B_{c/2}(y_m)})} = [v_{1,m}]_{C^{0,\alpha}(\overline{\omega_{1,m}})}$. This can be done via standard $L^p$ estimates for harmonic functions: for every $p>N$
\[
\begin{split}
\|w_m\|_{W^{2,p}(B_{c/2}(y_m))}  \le C \|w_m\|_{L^p(B_{c}(y_m))} \le C \|w_m\|_{L^\infty(B_c(y_m))} \le C,
\end{split}
\]
by Lemma \ref{lem: bdd H1 mis}. By the Sobolev embedding $W^{2,p} \hookrightarrow C^{0,1}$ for $p>N$, we deduce that
\[
[v_{1,m}]_{C^{0,\alpha}(B_{c/2}(y_m))} = [w_{m}]_{C^{0,\alpha}(B_{c/2}(y_m))}   \le C \|w_m\|_{W^{2,p}(B_{c/2}(y_m))} \le C,
\]
which is the desired result.

\smallskip

\emph{Case 3)} $r_m, R_m \to 0$ and $R_m/r_m \to +\infty$. \\
\noindent Let $z_m \in \Gamma_{\mf{v}_m}^3$ such that $|x_m-z_m| =2R_m$. By Lemma \ref{lem: decay1L}, there exists $C>0$ such that for every $m$ large
\[
\frac{1}{r^{N-1}} \int_{S_r(z_m)} \sum_i v_{i,m}^2\,d\sigma \le C r^{2\alpha}, \quad \forall r \in (0,R).
\]
In particular, this estimate holds for $r \in (0,4R_m)$, and by integrating it we deduce that
\[
\frac{1}{|B_{4R_m}(z_m)|} \int_{B_{4R_m}(z_m)} \sum_i v_{i,m}^2\,dx \le C R_m^{2\alpha}.
\]
Hence, since $v_{i,m}$ is subharmonic, for every $y \in B_{R_m}(x_m)$,
\beq\label{1461}
\begin{split}
\sum_i v_{i,m}^2(y) &\le \frac{1}{|B_{R_m}(y)|} \int_{B_{R_m}(y)} \sum_i v_{i,m}^2\,dx\\
 &\le \frac{C}{|B_{4R_m}(z_m)|} \int_{B_{4R_m}(z_m)} \sum_i v_{i,m}^2\,dx \le C R_m^{2\alpha},
\end{split}
\eeq
with $C$ independent of $m$ and of $y$. Since $r_m \ll R_m$, the ball $B_{R_m/2}(x_m)$ contains $y_m$. Furthermore, in the ball $B_{R_m}(x_m)$ there are no triple points of $\mf{v}_m$, by definition of $R_m$, and hence, as in the previous step, we can construct a harmonic function $w_m$ in $B_{R_m}(x_m)$ such that $[v_{1,m}]_{C^{0,\alpha}(B_{R_m/2}(x_m))} = [w_{m}]_{C^{0,\alpha}(B_{R_m/2}(x_m))}$ and $\|w_m\|_{L^\infty(B_{R_m}(x_m))} \le \|\mf{v}_m\|_{L^\infty(B_{R_m}(x_m))}$. Let us define now $w_m'(x):= w_m(x_m+R_m x)$. Being $w_{m}$ harmonic in $B_{R_m}(x_m)$, we have that $w_m'$ is harmonic as well in $B_1$, and by $L^p$ estimates, Sobolev embedding, and using \eqref{1461}, we obtain
\[
\begin{split}
[w_{m}]_{C^{0,\alpha}(\overline{B_{R_m/2}(x_m)})} & = R_m^{-\alpha} [w_{m}']_{C^{0,\alpha}(\overline{B_{1/2}})}  \le C R_m^{-\alpha} \|w_m'\|_{W^{2,p}(B_{1/2})}   \le C R_m^{-\alpha}  \|w_m'\|_{L^{p}(B_1)} \\
& = C R_m^{-\alpha} \left( R_m^{-N/p} \|w_m\|_{L^{p}(B_{R_m}(x_m))}\right)  \le C R_m^{-\alpha}\|w_m\|_{L^\infty(B_{R_m}(x_m))} \le C,
\end{split}
\]
which is the desired result.

\smallskip

\emph{Case 4)} $r_m, R_m \to 0$ and $R_m/r_m \le C$. \\
\noindent Let $z_m, z_m' \in \Gamma_{\mf{v}_m}^3$ such that $|x_m-z_m| =2R_m$ and $|y_m-z_m'| = \dist(y_m,\Gamma_{\mf{v}_m}^3) \le 2R_m$. By Lemma \ref{lem: decay1L}, there exists $C>0$ such that for every $m$ large
\[
\frac{1}{r^{N-1}} \int_{S_r(z_m)} v_{1,m}^2\,d\sigma \le C r^{2\alpha} \quad \text{and} \quad  \frac{1}{r^{N-1}} \int_{S_r(z_m')} v_{1,m}^2\,d\sigma \le C r^{2\alpha},
\]
for every $r \in (0,R)$. By integrating, we deduce that
\[
\frac{1}{|B_{3R_m}(z_m)|} \int_{B_{3R_m}(z_m)} v_{1,m}^2\,dx \le C R_m^{2\alpha} \quad \text{and} \quad  \frac{1}{|B_{3R_m}(z_m')|} \int_{B_{3R_m}(z_m')} v_{1,m}^2\,dx \le C R_m^{2\alpha}.
\]
Since $v_{1,m}$ is subharmonic, as in the previous step we deduce that
\[
v_{1,m}^2(x_m) \le \frac{C}{|B_{3R_m}(z_m)|} \int_{B_{3R_m}(z_m)} v_{1,m}^2\,dx \le C R_m^{2\alpha} \le C r_m^{2\alpha}.
\]
The same estimate also holds true for $v_{1,m}^2(y_m)$. Therefore,
\[
|v_{1,m}(x_m)-v_{1,m}(y_m)| \le C r_m^{\alpha} = C |x_m-y_m|^\alpha,
\]
as desired.
\end{proof}

The lemma and the Ascoli-Arzel\`a theorem imply that, by means of a diagonal selection, there exists a convergent subsequence $\mf{v}_m \to \mf{w}$ in $C^{0,\alpha}_{\loc}(\R^N)$. Now we prove that the convergence is also strong in $H^1_{\loc}(\R^N)$.

\begin{lemma}
Up to a subsequence, $\mf{v}_m \to \mf{w}$ strongly in $H^1_{\loc}(\R^N)$.
\end{lemma}

\begin{proof}
For any $R>0$, let $\varphi \in C^\infty_c(B_R)$. On $B_R$, for every $m$ sufficiently large
\[
\Delta v_{i,m} = \mu_{i,m} \quad \text{and} \quad \Delta w_{i} = \bar \mu_{i}.
\]
We test these equations with $(v_{i,m}-w_i) \varphi$, and subtract term by term: we obtain
\begin{multline*}
\left| \int_{B_R} |\nabla (v_{i,m}-w_i)|^2\varphi\,dx\right| \\
\le \|v_{i,m}-w_i\|_{L^\infty(B_R)} \left[ \int_{B_R} |\nabla(v_{i,m}-w_i)\cdot \nabla \varphi| \,dx + \int_{B_R} |\varphi|\,d\mu_{i,m}+ \int_{B_R} |\varphi|\,d\bar \mu_{i}\right],
\end{multline*}
and the right hand side tends to $0$ as $m \to \infty$, by local uniform convergence and Lemma \ref{lem: bdd H1 mis}. Since $\varphi$ was arbitrarily chosen, the thesis follows.
\end{proof}

\begin{proof}[Proof of Proposition \ref{prop: conv blow-up}]
We proved that, up to a subsequence, $\mf{v}_m \to \mf{w}$ in $C^{0}_{\loc}(\R^N)$ and in $H^1_{\loc}(\R^N)$. We show that $\mf{w} \in \mathcal{L}_{\loc}(\R^N)$. Let $R>0$ and $\eps>0$ be arbitrarily chosen. There exists $\bar m \in \N$ such that
\beq\label{1751}
m > \bar m \quad \implies \quad \|\mf{w}-\mf{v}_m\|_{L^\infty(B_R)} + \|\mf{w}-\mf{v}_m\|_{H^1(B_R)} < \frac{\eps}{2}.
\eeq
Moreover, since $\mf{v}_m \in \mathcal{L}_{\loc}(\Omega_m)$, we can take sequences $\{\mf{v}_{m,n}: n \in \N\}$ and $M_{m,n} \to +\infty$ satisfying the three points in Definition \ref{def: L}. For each $m$ fixed, we choose $\bar n=\bar n(m) \gg 1$ so large that
\beq\label{1752}
\|\mf{v}_{m,\bar n}-\mf{v}_m\|_{L^\infty(B_R)} + \|\mf{v}_{m,\bar n}-\mf{v}_m\|_{H^1(B_R)} < \frac{\eps}{2},
\eeq
and moreover
\beq\label{1753}
M_{m, \bar n} > m \quad \text{and} \quad M_{m, \bar n} \int_{B_R} \prod_i v_{i,m,\bar n}^2\,dx < \frac1m.
\eeq
By combining \eqref{1751}-\eqref{1753}, we deduce that the definition of $\mathcal{L}(B_R)$ applies to $\mf{w}$ with the sequences $\mf{w}_m:= \mf{v}_{m, \bar n}$ and $M_{m, \bar n}$. Since $R$ was arbitrarily chosen, we finally infer that $\mf{w} \in \mathcal{L}_{\loc}(\R^N)$.

We prove now that $\mf{w}$ is $\gamma$-homogeneous, with $\gamma=N(\mf{v},x_0,0^+)$. Indeed, since $\mf{v}_m \to \mf{w}$ in $C^0_{\loc}(\R^N)$ and in $H^1_{\loc}(\R^N)$, we have that for every $R>0$
\[
N(\mf{w},0,R) = \lim_{m \to \infty} N(\mf{v}_m,0,R) = \lim_{m \to \infty} N(\mf{v},x_0,\rho_m R) = \gamma.
\]
This means that $N(\mf{w},0,\,\cdot)$ is constant, equal to $\gamma$, and the homogeneity follows by Proposition \ref{prop: alm seg}. Furthermore, Lemmas \ref{lem: N >=} and \ref{cor: comp 0 seg} imply that all the components of $\mf{w}$ are non-trivial.
\end{proof}

\section{Free boundary regularity}\label{sec: fb reg}

In this section we study the structure of the nodal set $\Gamma_{\mf{v}}$ for functions in $\mathcal{L}_{\loc}(\Omega)$. As we will see, the crucial step consists in studying the properties of the set $\Gamma_{\mf{v}}^3$. So far, we know that $\Gamma_{\mf{v}}^3$ has empty interior, that
\[
N(\mf{v},x_0,0^+) \ \begin{cases} = 0 & \text{if $m(x_0) =1, 2,$} \\ \ge \bar \nu & \text{if $m(x_0)=3$},
\end{cases}
\]
and that $\Gamma_{\mf{v}}^3$ is a relatively closed subset of $\Gamma_{\mf{v}}$, see Lemmas \ref{lem: 3 int vuo}-\ref{lem: 3 closed}. The first main result of this section is the following:

\begin{proposition}\label{prop: H dim sing}
For any $\mf{v} \in \mathcal{L}_{\loc}(\Omega)$, the Hausdorff dimension of $\Gamma_{\mf{v}}^3$ is at most $N-2$, and $\Gamma_{\mf{v}}^3$ is a locally finite set in dimension $N=2$.
\end{proposition}

To prove the proposition, we make use of the \emph{Federer's dimensional reduction principle}, in the form appeared in \cite[Theorem 8.5]{Chen} (see also \cite{TT12}). In turn, this is a generalization of the main result presented in \cite[Appendix A]{SimBook}.

\begin{theorem}[Federer's Reduction Principle]\label{teo:FRP}
Let $\mathcal{F}\subset L^{\infty}_\text{loc}(\R^N, \R^h)$, and define, for any given $U\in\mathcal{F},\ x_0\in\R^N$ and $t>0$, the rescaled and translated function $$U_{x_0,t}:=U(x_0+t\,\cdot).$$
We say that $U_n\rightarrow U$ in $\mathcal{F}$ iff $U_n\rightarrow U$ in $L^{\infty}_\text{loc}(\R^N,\R^h)$.

Assume that $\mathcal{F}$ satisfies the following conditions:
\begin{itemize}
\item[(A1)] (Closure under appropriate scalings and translations)
Given any $|x_0|\leq 1-t, 0<t<1$, $\rho>0$ and $U\in \mathcal{F}$, we have that also $\rho\cdot U_{x_0,t}\in \mathcal{F}$.
\item[(A2)] (Existence of a homogeneous ``blow--up'')
Given $|x_0|<1, t_k\downarrow 0$ and $U \in \mathcal{F}$, there exists a sequence $\rho_k\in(0,+\infty)$, a real number $\alpha\geq 0$ and a  function $\bar U\in \mathcal{F}$ homogeneous of degree\footnote{That is, $\bar U(tx)=t^\alpha U(x)$ for every $t>0$.} $\alpha$ such that, if we define $U_k(x)=U(x_0+t_k x)/\rho_k$, then
$$U_k\rightarrow \bar U \qquad {\rm in }\ \mathcal{F},\qquad \qquad \text{ up to a subsequence}.$$
\item[(A3)] (Singular Set hypotheses)
There exists a map $\Sigma:\mathcal{F}\rightarrow \cC$ (where $\cC:=\{A\subset \R^N:\ A\cap B_1 \text{ is relatively closed in } B_1\}$) such that
\begin{itemize}
 \item [(i)]  Given $|x_0|\leq 1-t$, $0<t<1$ and $\rho>0$, it holds $$\Sigma(\rho\cdot U_{x_0,t})=(\Sigma(U))_{x_0,t}:=\frac{\Sigma(U)-x_0}{t}.$$
 \item[(ii)] Given $|x_0|<1$, $t_k\downarrow 0$ and $U,\bar U\in \cF$ such that there exists $\rho_k>0$ satisfying $U_k:=\rho_k U_{x_0,t_k}\rightarrow \bar U$ in $\cF$, the following ``continuity'' property holds:
$$\forall \varepsilon>0\ \exists k(\epsilon)>0:\ k\geq k(\varepsilon) \Rightarrow \Sigma(U_k)\cap B_1\subseteq \{x\in \R^N:\ \dist(x,\Sigma(\bar U))<\varepsilon\}.$$
\end{itemize}
\end{itemize}
Then, either $\Sigma(U)\cap B_1=\emptyset$ for every $U\in \cF$, or else there exists an integer $d \in [0,N-1]$ such that ${\rm dim}_{\cH}(\Sigma(U)\cap B_1)\leq d$ for every $U\in\cF$. Moreover in the latter case there exist a function $V\in \cF$, a d-dimensional subspace $L \,\subseteq\, \R^N$ and a real number $\alpha\geq0$ such that
\begin{equation*}\label{Psi_invariant_over_L}
V_{y,t}=t^\alpha V \qquad \forall y\in L, \ t>0, \qquad \quad \text{ and } \qquad \quad \Sigma(V)\cap B_1=L\cap B_1.
\end{equation*}
If $d=0$ then $\Sigma(U)\cap B_\rho$ is a finite set for each $U\in \cF$ and $0<\rho<1$.
\end{theorem}

\begin{proof}[Proof of Proposition \ref{prop: H dim sing}]
Since any scaling or translation does not change the Hausdorff dimension of a set, and since a countable union of sets with Hausdorff dimension less than or equal to some $d \in [0,+\infty)$ also has Hausdorff dimension less than or equal to $d$, we can prove the estimate for $\Gamma_{\mf{v}}^3 \cap B_1$, supposing that $\mf{v} \in \mathcal{L}_{\loc}(B_2)$. To this end, we apply the Federer's principle to the family
\[
\mathcal{F}:= \left\{ \mf{v} \in L^\infty_{\loc}(\R^N,\R^3) \left| \begin{array}{l} \text{there exists a domain $\Omega \supset \joinrel \supset B_2$}\\ \text{such that $\mf{v} \in \mathcal{L}_{\loc}(\Omega)$}\end{array}\right.\right\}.
\]
The results of the previous sections ensure that Assumption (A1) and (A2) are satisfied, see in particular Lemma \ref{lem: inv sc}, Remark \ref{rem: bu m2} and Proposition \ref{prop: conv blow-up}. Concerning (A3), let us define $\Sigma: \mathcal{F}\rightarrow \cC$ by $\Sigma(\mf{v})=\Gamma_{\mf{v}}^3 \cap B_1$ (we have already recalled that this is a closed set). It is straightforward to check that (A3)-($i$) is fulfilled. Moreover, the validity of (A3)-($ii$) is a direct consequence of the local uniform convergence of blow-up sequences, and of the fact that $\Gamma_{\mf{v}}^3$ is defined by the condition $\mf{v}(x)=0$ (which is stable under local uniform convergence). Therefore, the Federer's principle is applicable, and implies that either $\Gamma_{\mf{v}}^3 \cap B_1=\emptyset$ for every $\mf{v} \in \mathcal{F}$, or there exists an integer $0 \le d \le N-1$ such that
\[
\dim_{\cH} (\Gamma_{\mf{v}}^3 \cap B_1) \le d \qquad \text{for every $\mf{v} \in \cF$}.
\]
Moreover, in this latter case there exists a $d$-dimensional linear subspace $E \subset \R^N$, and an $\alpha$-homogeneous function $\mf{v} \in \cF$ (for some $\alpha\ge 0$) such that $\Gamma_{\mf{v}}^3=E$, and
\begin{equation}\label{d-dim}
\frac{\mf{v}(x_0+\lambda x)}{\lambda^\alpha} = \mf{v}(x) \qquad \text{for every $x_0 \in E$ and $\lambda>0$};
\end{equation}
identity \eqref{d-dim} means that $\mf{v}$ is $\alpha$-homogeneous with respect to all the points in $E$, and hence, up to a rotation, it depends only on $N-d$ variables.

Let us suppose by contradiction that $d=N-1$. Then, without loss of generality, we can suppose that $\mf{v}(x_1,\dots,x_N) = \mf{w}(x_1)$ for a function $\mf{w} \in \mathcal{L}_{\loc}(\R)$ with $\Gamma_{\mf{w}}^3 = \{x_1=0\}$. Notice also that we can extend $\mf{v}$ outside of $B_2$ as a function in $\mathcal{L}_{\loc}(\R^N)$, by homogeneity. Now, by Lemmas \ref{lem: N >=} and \ref{cor: comp 0 seg}, we have that $N(\mf{v},0,0^+) \ge \bar \nu$, and hence all the components of $\mf{v}$ and $\mf{w}$ are non-trivial. By homogeneity, $w_i(1)>0$ if and only $w_i(x_1)>0$ for every $x_1>0$, and $w_i(-1)>0$ if and only $w_i(x_1)>0$ for every $x_1<0$. By partial segregation, at least one component must vanish on $(-\infty,0]$, and at least one component must vanish on $[0,+\infty)$. This means that two components, say $w_1$ and $w_2$, satisfy the condition $w_1\,w_2 \equiv 0$ in $\R$. By Lemma \ref{lem: no |x|}, we deduce that $w_3$ is a non-negative harmonic function in $\R$, with $w_3(0) = 0$. Therefore, $w_3 \equiv 0$, a contradiction. This shows that $d \le N-2$, as desired.
\end{proof}

The previous proposition allows us to show that the zero sets of different components do not share common interior points.

\begin{lemma}\label{lem: tot emp int}
For any $\mf{v} \in \mathcal{L}_{\loc}(\Omega)$, we have that $\mathrm{int}\{v_i=0\} \cap \mathrm{int}\{v_j=0\} = \emptyset$ for every $i \neq j$, unless two components of $\mf{v}$ vanish identically and the other is strictly positive.
\end{lemma}

\begin{proof}
Let us suppose that $x_0 \in \mathrm{int}\{v_i=0\} \cap \mathrm{int}\{v_j=0\}$ for a pair of indexes, say $i=1$ and $j=2$. Then there exists $R>0$ such that $v_1 = 0 = v_2$ in $B_R(x_0)$. If $v_3(x_0)=0$, since $\Gamma_{\mf{v}}^3$ has empty interior, we must have that $v_3|_{S_\rho(x_0)}\neq 0$ for every $\rho <R$. Then $v_3(x_0)>0$ by Lemma \ref{lem: no self}, a contradiction. Therefore, $v_3(x_0)>0$, and by Lemma \ref{cor: diff arm} we have that $v_1-v_2$ is harmonic in the connected component $\omega$ of $\{v_3>0\}$ containing $x_0$. Since $v_1-v_2 = 0$ in $B_R(x_0)$, by unique continuation $v_1-v_2 = 0$ in $\omega$, and moreover $v_1v_2 \equiv 0$ in $\omega$, whence $v_1=v_2=0$ in $\omega$. We claim that $\omega = \Omega \setminus \Gamma^3_{\mathbf{v}}$. Let $y \in \partial \omega \cap \Omega$, so that $v_3(y) = 0$; since also $v_1(y) = v_2(y) = 0$ by
continuity, we have $\partial \omega \cap \Omega \subseteq \Gamma^3_{\mathbf{v}}$.
By Proposition \ref{prop: H dim sing} we have $\dim_{\mathcal{H}}
(\Gamma^3_{\mathbf{v}}) \le N-2$, and in particular $\mathcal{H}^{N-1}
(\Gamma^3_{\mathbf{v}}) = 0$; therefore $\Gamma^3_{\mathbf{v}}$ does not disconnect
$\Omega$\footnote{If $F \subseteq \Omega$ is relatively closed with $\mathcal{H}^{N-1}(F) = 0$ and $\Omega$ is open and connected, then $\Omega \setminus F$ is path-connected. Indeed, given a ball $B \subset \Omega$ and $p,q \in B \setminus F$, the cone with vertex $p$ over $F \cap B$ is Lebesgue-negligible (the radial projection onto a sphere centeredat $p$ is Lipschitz away from $p$, hence it maps $\mathcal{H}^{N-1}$-null sets into $\mathcal{H}^{N-1}$-null subsets of the sphere), and the same holds with $q$ in place of $p$; thus for a.e. $z \in B$ the broken line $p \to z \to q$ is contained in $B \setminus F$. The general case follows by covering a path joining two points of $\Omega \setminus F$ with a finite chain of balls.}, namely $\Omega \setminus \Gamma^3_{\mathbf{v}}$ is open
and connected. Since $\omega \subset \Omega \setminus \Gamma^3_{\mathbf{v}}$ is
non-empty, open and relatively closed in $\Omega \setminus \Gamma^3_{\mathbf{v}}$ (its
relative boundary being empty, by the above inclusion), we deduce that $\omega = \Omega
\setminus \Gamma^3_{\mathbf{v}}$, as claimed.

As $\Gamma^3_{\mathbf{v}}$ is closed and has empty interior, $\omega$ is dense in $\Omega$,
and hence $v_1 \equiv v_2 \equiv 0$ in $\Omega$ by continuity. In particular $v_1 v_2
\equiv 0$ in $B_R(y)$ for every ball $B_R(y) \subseteq \Omega$, so that $v_3$ is harmonic
in $\Omega$ by Lemma \ref{lem: no |x|}. Finally, $v_3 \ge 0$ is harmonic in
the connected open set $\Omega$ and $v_3(x_0)>0$; by the strong maximum principle, $v_3 >
0$ in $\Omega$.
\end{proof}

We are ready to prove the main result regarding the structure of the nodal set $\Gamma_{\mf{v}}$ and of the free-boundary $\Gamma_{\mf{v}}'$ (recall definitions \eqref{def fb} and \eqref{def fb2}).

\begin{proof}[Proof of Theorem \ref{thm: nodal set}]
($i$) Assume that $\Gamma_{\mf{v}} \neq \Omega$. If $x_0 \in \Gamma_{\mf{v}}'$, then $x_0 \in \pa \{v_i>0\}$ for at least an index $i$, say $x_0 \in \pa \{v_1>0\}$. Then $v_1(x_0)=0$, and moreover for every $r>0$ there exists $x_r \in B_r(x_0)$ such that $v_1(x_r)>0$. If $v_2(x_0), v_3(x_0)>0$, then by continuity $v_2,v_3>0$ in a small ball $B_\rho(x_0)$. Therefore, such ball contains points where all the components are positive, in contradiction with the partial segregation condition. It is then necessary that at least one of the equalities $v_2(x_0)=0$ and $v_3(x_0)=0$ is satisfied, and hence $x_0 \in \Gamma_{\mf{v}}$. This proves that $\Gamma_{\mf{v}}' \subset \Gamma_{\mf{v}}$. On the other hand, let $x_0 \in \Gamma_{\mf{v}}$; thus, $x_0 \in \{v_i=0\}$ for at least two indexes, say $i=1,2$. If $x_0 \not \in \Gamma_{\mf{v}}'$, then $x_0 \not \in \pa \{v_j=0\}$ for any index $j=1,2,3$, and hence $x_0 \in \mathrm{int}\{v_1=0\} \cap \mathrm{int}\{v_2=0\}$, in contradiction with Lemma \ref{lem: tot emp int}. Thus, $x_0 \in \Gamma_{\mf{v}}'$, and hence $\Gamma_{\mf{v}} \subset \Gamma_{\mf{v}}'$.\\
($ii$) It is an easy consequence of Proposition \ref{prop: H dim sing}, and Lemmas \ref{prop: diff arm} and \ref{lem: tot emp int}. Indeed, any point $x_0 \in \Gamma_{\mf{v}}^2$ has a neighborhood where $v_i-v_j$ is harmonic for some $i \neq j$, namely $\Gamma_{\mf{v}}^2$ has locally the same structure as the nodal set of a harmonic function. This means that, unless two components of $\mf{v}$ vanish identically and the other is strictly positive, $\Gamma_{\mf{v}}^2$ splits as $\mathcal{R}_{\mf{v}}^2 \cup \Sigma_{\mf{v}}^2$, where $\mathcal{R}_{\mf{v}}^2$ is the union of smooth hypersurfaces of dimension $N-1$, $\Sigma_{\mf{v}}^2$ has Hausdorff dimension at most $N-2$, and is locally finite in \(\Omega\setminus\Gamma_{\mf{v}}^3\) when \(N=2\).
Suppose that \(x_n\in\Sigma_{\mf{v}}^2\) converge to \(x_0\in\Gamma_{\mf{v}}^3\).
Blowing up at \(x_0\) at scales \(r_n=|x_n-x_0|\), we obtain a homogeneous
limit \(\mf{w}\) by Proposition~\ref{prop: conv blow-up}, with \((x_n-x_0)/r_n\to e\in S^1\)
up to a subsequence. By homogeneity and Proposition~\ref{prop: H dim sing},
\(\Gamma_{\mf{w}}^3=\{0\}\). Thus some \(w_k>0\) near \(e\), where the
rescaled differences of the other two components are harmonic by
Lemma~\ref{cor: diff arm} and converge locally in \(C^1\). Their critical zeros therefore
give \(e\in\Sigma_{\mf{w}}^2\). By homogeneity the whole ray through \(e\)
lies in \(\Sigma_{\mf{w}}^2\), contradicting its local finiteness away
from the origin. Therefore, $\Sigma_{\mf{v}}^2$ is locally finite.

Now it is sufficient to pose $\mathcal{R}_{\mf{v}}:=\mathcal{R}_{\mf{v}}^2$ and $\Sigma_{\mf{v}} = \Sigma_{\mf{v}}^2 \cup \Gamma_{\mf{v}}^3$, and the thesis follows.  \\
 ($iii$) For concreteness, we suppose by contradiction that $0 \in \{v_1=0\}$ and $0 \not \in \overline{\mathrm{int}\{v_1=0\}}$. Then there exists $R>0$ such that $\mathrm{int}\{v_1=0\} \cap B_R =\emptyset$. Hence, by subharmonicity, $v_1|_{S_\rho} \not \equiv 0$ for every $\rho>0$. If $v_2(0),v_3(0)>0$, then by continuity the partial segregation condition would be violated. Therefore, at least another component, say $v_2$, vanishes at $0$. We distinguish two cases, $v_3(0)>0$ and $v_3(0) = 0$, and we obtain a contradiction in both of them. This completes the proof.

If $v_3(0)>0$, up to replacing $R$ with a smaller quantity, we can also assume that $\inf_{B_R} v_3>0$. We claim that $v_2 \equiv 0$ in $B_R$. Indeed, if not, then we should have
\[
v_1 \not \equiv 0, \quad v_2 \not \equiv 0, \quad v_1\,v_2 \equiv 0 \quad \text{in $B_R$}.
\]
But this implies by continuity that any point in $\{v_2>0\}$ is in $\mathrm{int}\{v_1=0\} \cap B_R \neq \emptyset$, in contradiction with the definition of $R$. Therefore, $v_2 \equiv 0$ in $B_R$, as claimed. Since instead $v_1|_{S_{R-\eps}} \not \equiv 0$, Lemma \ref{lem: no self} implies that $v_1>0$ in $B_R$, in contradiction with the fact that $v_1(0) = 0$. Thus, the case $v_3(0)>0$ cannot take place.

On the other hand, if $v_3(0)=0$, we claim that $v_2 \, v_3 \equiv 0$ in $B_R$. Indeed, if not, then we should have
\[
v_1 \not \equiv 0, \quad v_2 \,v_3 \not \equiv 0, \quad v_1\,v_2\,v_3 \equiv 0 \quad \text{in $B_R$}.
\]
As before, this implies by continuity that any point in $\{v_2 v_3>0\}$ is in $\mathrm{int}\{v_1=0\} \cap B_R \neq \emptyset$, in contradiction with the definition $R$. Therefore, $v_2\, v_3 \equiv 0$ in $B_R$, and Lemma \ref{lem: no |x|} implies that $v_1$ is harmonic in $B_{R/2}$. In particular, either $v_1>0$ in $B_{R/2}$, or $v_1 \equiv 0$ in $B_{R/2}$, and both of these alternatives lead to a contradiction.
\end{proof}

\begin{remark}\label{rem: on triple}
Let $\mf{v} \in \mathcal{L}_{\loc}(\Omega)$ be such that $v_i \not \equiv 0$ for every $i$, and let $x_0 \in \Gamma_{\mf{v}}^3$ be a triple point. Then we have the following basic properties:

($i$) For every $r>0$ and every $i =1,2,3$ we have $B_r(x_0) \cap \{v_i>0\} \neq \emptyset$. Indeed, if by contradiction $v_i \equiv 0$ in $B_{r_0}(x_0)$ for some $r_0>0$ and an index $i$, then by Lemma \ref{lem: no self} one easily deduces that any other component must be strictly positive at $x_0$, a contradiction.

($ii$) For every $r>0$ and every $i =1,2,3$ we also have $B_r(x_0) \cap \mathrm{int}\{v_i=0\} \neq \emptyset$. Indeed, if not $\{v_i=0\} \cap B_{r_0}(x_0)=\pa \{v_i=0\} \cap B_{r_0}(x_0)$ for some $r_0>0$ and an index $i$, say $i=3$. Since $\Gamma_{\mf{v}}$ has dimension $N-1$, we deduce that $v_1\,v_2 \equiv 0$ in $B_{r_0}(x_0)$. Thus, Lemma \ref{lem: no |x|} implies that $v_3$ is harmonic in $B_{\,r_0/2}(x_0)$ and, being non-negative and non-trivial, it cannot vanish at $x_0$, a contradiction again.
\end{remark}

\section{Improved H\"older bounds and $C^{0,3/4}$ regularity}\label{sec: ihb}

The main result in \cite{ST24p1} establishes the validity of uniform H\"older estimates for sequences $\{\mf{u}_\beta\}$ of uniformly bounded local minimizers, for ``small" exponents $0<\alpha < \bar \nu \le 2/3$. In this section we aim at improving this result, by demonstrating Theorem \ref{thm: improved Holder bounds}: we will obtain uniform H\"older estimates up to the exponent $\hat \nu =3/4$, and show that any function in $\mathcal{L}_{\loc}$ is $C^{0,3/4}$ regular in $\Omega$. Subsequently, in Section \ref{sec: mer}, we will also prove Theorem \ref{prop: mer min} which demonstrates the optimality of the result.

To prove the validity of $C^{0,3/4-}$ estimates, we recall from \cite[Remark 4.14]{ST24p1} that the limitation $\alpha < \bar \nu$ in \cite[Theorem 1.1]{ST24p1} enters into the game only when applying two Liouville-type theorems for entire solutions of certain systems modeling partial segregation, see \cite[Corollary 3.5 and Theorem 3.6]{ST24p1}. We recall the statements for the reader's convenience.

\begin{corollary}
[Corollary 3.5 in\cite{ST24p1}]\label{cor: liou seg}
Let $N \ge 2$ be a positive integer, and let $\alpha \in (0, \bar \nu)$. For $i=1,2,3$, let $u_i \in H^1_{\loc}(\R^N) \cap C(\R^N)$ be globally $\alpha$-H\"older continuous functions in $\R^N$, such that
\[
-\Delta u_i \le 0 \quad \text{and} \quad u_i(x) \ge 0 \quad \text{in $\R^N$},
\]
and moreover the partial segregation condition holds:
\[
\prod_{j=1}^3 u_j  \equiv 0 \quad \text{in $\R^N$},
\]
Then at least one component $u_i$ vanishes identically.
\end{corollary}

\begin{theorem}
[Follows from Theorem 3.6 in \cite{ST24p1}]\label{thm: liou sub}
Let $N \ge 2$ be a positive integer, and, for $i=1,2,3$, let $u_i \in H^1_{\loc}(\R^N) \cap C(\R^N)$ be nonnegative functions satisfying
\beq\label{ent sys}
\Delta u_i = M u_i \prod_{j \neq i} u_j^2 \quad \text{in $\R^N$},
\eeq
for some constant $M >0$. Assume moreover that
\[
0\le u_i(x) \le C(1+|x|^{\alpha}) \quad \text{for every $x \in \R^N$},
\]
with $\alpha  \in (0,\bar \nu)$. Then at least one component $u_i$ vanishes identically.
\end{theorem}

The proof of these Liouville-type theorems is based upon an Alt-Caffarelli-Friedman monotonicity formula for partially segregated function. In fact, the exponent $\bar \nu$ is characterized in terms of an optimal ``overlapping partition problem" on the sphere which is a key ingredient to derive the monotonicity of the ACF-type functional, see \cite[Section 2]{ST24p1}.

In \cite{ST24p1}, we applied the previous results to entire functions obtained as limit of scaled local minimizers of the associated energy. In particular, Corollary \ref{cor: liou seg} was applied to globally H\"older continuous functions in $\mathcal{L}_{\loc}(\R^N)$, and Theorem \ref{thm: liou sub} was applied to solutions of \eqref{ent sys} which are minimal for the associated energy with respect to variations with compact support. We stress that the minimality property was not used in \cite{ST24p1}. In this section we show how to leverage minimality to obtain improved Liouville theorems, that is, with a threshold $\hat \nu=3/4$, greater than $\bar \nu \le 2/3$ for the growth of the solutions considered.

This allows us to pass from uniform $C^{0,\alpha}$ bounds for small $\alpha$, to uniform $C^{0,3/4-}$ bounds.

\begin{proposition}\label{prop: imp liou}
Let $\nu \in (0,3/4)$, and suppose that one of the following assumptions is satisfied:
\begin{itemize}
\item[(1)] $\mf{v} \in \mathcal{L}_{\loc}(\R^N)$ and $\mf{v}$ is globally $\nu$-H\"older continuous.
\item[(2)] $\mf{v} \not \equiv 0$, $\mf{v} \in H^1_{\loc}(\R^N) \cap C(\R^N)$ is a minimizer of $J_{M}$ with respect to variations with compact support in $\R^N$, and $0 \le v_i(x) \le C(1+|x|^\nu)$ for every $x \in \R^N$, for some positive constants $C, M > 0$.
\end{itemize}
Then at least one component $v_i$ vanishes identically.
\end{proposition}

Once that Proposition \ref{prop: imp liou} is proved, Theorem \ref{thm: improved Holder bounds} follows rather easily. The rest of the section is organized as follows: in Subsection \ref{sub: imp liou} we prove Proposition \ref{prop: imp liou}. In Subsection \ref{sub: conc imp} we complete the proof of Theorem \ref{thm: improved Holder bounds}.

\subsection{Proof of Proposition \ref{prop: imp liou}}\label{sub: imp liou}

The proof is divided into several intermediate steps; we shall often use the results of Sections \ref{sec: Alm} and \ref{sec: fb initial study}, and the following:

\begin{proposition}\label{prop: liou seg full}
Let $\alpha \in (0,1)$, and let $(u,v) \in H^1_{\loc}(\R^N) \cap C(\R^N)$ be such that
\[
\begin{cases}
\Delta u \ge 0   \\
\Delta v \ge 0 \\
u \, v \equiv 0,
\end{cases}
\qquad 0 \le u(x),\, v(x) \le C(1+|x|^{\alpha}) \quad \text{for every $x \in \R^N$},
\]
for some constant $C>0$. Then at least one component of $(u,v)$ must vanish identically.
\end{proposition}

\begin{remark}\label{rem: liou growth}
In \cite[Proposition~7.2]{CTV05} the result is stated under a global
H\"older continuity assumption. The version with the growth bound used
here follows from \cite[Theorem~3.1 and Remark~3.3]{ST23} by taking
\(k=2\) and \(A_1=A_2=I\). This formulation is useful in the blow-down
argument below, since homogeneity gives the required growth bound
directly.
\end{remark}

The following blow-down lemma reduces the finite-parameter case of Proposition \ref{prop: imp liou} to the study of homogeneous segregated profiles.

\begin{lemma}\label{lem: blow-down}
Let $M>0$, and let $\mf{v} \not \equiv 0$ be a minimizer of $J_{M}$ with respect to variations with compact support in $\R^N$, such that $0 \le v_i(x) \le C(1+|x|^\nu)$ for every $x \in \R^N$ and every $i$. Then there exist $\nu' \in [0,\nu]$ and $\mf{w} \in \mathcal{L}_{\loc}(\R^N)$, $\nu'$-homogeneous with respect to $0$; in particular
\[
0 \le w_i(x) \le C' |x|^{\nu'} \qquad \text{for every $x \in \R^N$ and every $i$},
\]
with $C' = \max_{S_1} |\mf{w}|$. Moreover, if all the components $v_i$ are non-trivial, then the same holds for the components $w_i$ (and $\nu'>0$).
\end{lemma}

\begin{proof}
By minimality, the function $\mf{v}$ solves
\[
\Delta v_{i} = M v_i \prod_{j \neq i} v_j^2 \quad \text{in $\R^N$}.
\]
We consider the \emph{blow-down family} of $\mf{v}$, defined by
\[
\mf{v}_r(x):= \frac{\mf{v}(rx)}{H(\mf{v},0,r)^{1/2}} \quad \forall r >1.
\]
Plainly, each $\mf{v}_r$ is a minimizer of the associated energy functional with respect to variations with compact support, $H(\mf{v}_r,0,1) = 1$, and
\[
\Delta v_{i,r} = M_r v_{i,r} \prod_{j \neq i} v_{j,r}^2, \quad v_i \ge 0 \quad \text{in $\R^N$},
\]
where $M_r:= r^2 H(\mf{v},0,r)^2 M \to +\infty$ as $r \to +\infty$ (notice that $H$ is nondecreasing with respect to $r$, by subharmonicity). Now, by Lemma \ref{lem: N infty coex} and the monotonicity of $N$, we have that
\[
N(\mf{v},0,r) \le N(\mf{v},0,+\infty) =:\nu' \le \nu \quad \text{for every $r>0$},
\]
so that by the doubling property in Lemma \ref{lem: doubling coex} we also have
\[
\int_{S_R} \sum_i v_{i,r}^2\,d\sigma = R^{N-1} H(\mf{v}_r,0,R) = R^{N-1}\frac{H(\mf{v},0,Rr)}{H(\mf{v},0,r)} \le e^{2\nu'} R^{N-1+2\nu'} \qquad \forall R>1.
\]
By subharmonicity, this implies that $\{\mf{v}_r\}$ is uniformly bounded on any compact set of $\R^N$. We are then in position to apply \cite[Theorem 1.1]{ST24p1}: up to a diagonal selection, there exists a subsequence of radii $r \to +\infty$ and a function $\mf{w} \in \mathcal{L}_{\loc}(\R^N) \cup \{0\}$ such that
\[
\mf{v}_{r} \to \mf{w} \quad \text{locally uniformly, and in $H^1_{\loc}(\R^N)$},
\]
and for every $R>1$
\[
M_r \int_{B_R} \prod_{j=1}^3 v_{j,r}^2\,dx \to 0.
\]
In fact, $\mf{w} \not \equiv 0$ and hence $\mf{w} \in \mathcal{L}_{\loc}(\R^N)$, since $H(\mf{w},0,1)=1$ by uniform convergence.

Now, by scaling and convergence, for every $R>0$ we have that
\[
N(\mf{w},0,R) = \lim_{r \to +\infty} N(\mf{v}_r,0,R) = \lim_{r \to +\infty} N(\mf{v},0,rR) = N(\mf{v},0,+\infty) = \nu',
\]
and the second part of Proposition \ref{prop: alm seg} ensures that $\mf{w}$ is $\nu'$-homogeneous with respect to $0$.
The growth estimate is then an immediate consequence of the homogeneity and of the continuity of $\mf{w}$.
It remains to show that, if all the components $v_i$ are non-trivial, then the same holds for the components $w_i$. Suppose that $v_i \not \equiv 0$ for every $i$. Then, by Lemma \ref{lem: N infty coex}, we have that $\nu' = N(\mf{v},0,+\infty)>0$. But we also have $\nu' = N(\mf{w},0,+\infty)$, so that also $w_i \not \equiv 0$ for every $i$, by Lemma \ref{cor: comp 0 seg}.
\end{proof}

\begin{remark}\label{rem: Lloc props}
By \cite[Theorem 1.7]{ST24p1} and Remark \ref{rem: on measures}, every $\mf{v} \in \mathcal{L}_{\loc}(\R^N)$ is a triplet of nonnegative subharmonic functions, continuous and in $H^1_{\loc}$, such that $\Delta v_i = 0$ in $\{v_i>0\}$ for every $i$ and $v_1v_2v_3 \equiv 0$. Lemmas \ref{lem: connect} and \ref{lem: triple on sphere} are therefore stated directly for such triplets, since we shall need to apply them also to restrictions of functions in $\mathcal{L}_{\loc}$ to linear subspaces.
\end{remark}

\begin{lemma}\label{lem: connect}
Let $\mf{v} \in H^1_{\loc}(\R^N,\R^3) \cap C^0(\R^N,\R^3)$ be a triplet of nonnegative subharmonic functions such that $\Delta v_i = 0$ in $\{v_i>0\}$ for every $i$, and such that $0 \le v_i(x) \le C(1+|x|^\nu)$ for every $x \in \R^N$, for some $C>0$ and $\nu \in (0,1)$. Then $\{v_i>0\}$ is connected for every $i$.
\end{lemma}
\begin{proof}
By contradiction, let $\{v_i>0\} = A \cup B$ for two non-empty open disjoint sets $A$ and $B$, each of which is a union of connected components of $\{v_i>0\}$. Then $w_1:= v_i \chi_A$ and $w_2:= v_i \chi_B$ are non-trivial subharmonic functions with disjoint positivity sets, satisfying the same growth estimate as $v_i$; namely, they satisfy the assumptions of Proposition \ref{prop: liou seg full}, whence we deduce that one of them vanishes identically, a contradiction.
\end{proof}

It is crucial to show that the set $\Gamma_{\mf{v}}^3$ of points of multiplicity 3 is not empty, and to characterize the behavior of $\mf{v}$ with respect to triple points. This is the content of the next three lemmas.

We recall from Theorem \ref{thm: nodal set} that, for every $\mf{v} \in \mathcal{L}_{\loc}(\R^N)$ with non-trivial components,
\beq\label{st 1 e 2}
\text{$\overline{\mathrm{int}\{v_i=0\}} = \{v_i=0\}$ for every $i$, and $\mathrm{int}\{v_i=0\} \cap \mathrm{int}\{v_j=0\} = \emptyset$ for every $i \neq j$.}
\eeq

\begin{lemma}\label{lem: trip non emp}
Let $\mf{v} \in \mathcal{L}_{\loc}(\R^N)$ be such that $0 \le v_i(x) \le C(1+|x|^\nu)$ for every $x \in \R^N$, for some $C>0$ and $\nu \in (0,1)$, and suppose that all the components of $\mf{v}$ are non-trivial. Then $\Gamma_{\mf{v}}^3 \neq \emptyset$.
\end{lemma}

\begin{proof}
Suppose by contradiction that $\Gamma_{\mf{v}}^3 = \emptyset$, and set
\[
\Omega_i := \{v_i>0\}, \qquad Z_k := \Omega_i \cap \Omega_j \quad \text{for } \{i,j,k\}=\{1,2,3\}.
\]

\medskip

\emph{Step 1)} $\Omega_1 \cup \Omega_2 \cup \Omega_3 = \R^N$, $\Omega_1 \cap \Omega_2 \cap \Omega_3 = \emptyset$, and $Z_1, Z_2, Z_3$ are pairwise disjoint open sets.

\smallskip

\noindent Since $\Gamma_{\mf{v}}^3 = \emptyset$, at every point at most two components vanish, and hence at least one component is positive: this gives the first identity. The second one is exactly the partial segregation condition $v_1v_2v_3 \equiv 0$. Finally, for $k \neq l$ we have $Z_k \cap Z_l \subset \Omega_1 \cap \Omega_2 \cap \Omega_3 = \emptyset$.

\medskip

\emph{Step 2)} $Z_k \neq \emptyset$ for every $k$.

\smallskip

\noindent Suppose for instance that $Z_1 = \Omega_2 \cap \Omega_3 = \emptyset$, that is, $v_2 v_3 \equiv 0$ in $\R^N$. Since $v_2$ and $v_3$ are nonnegative, subharmonic, and satisfy the growth assumption, Proposition \ref{prop: liou seg full} implies that one between $v_2$ and $v_3$ vanishes identically, against our assumptions.

\medskip

\emph{Step 3)} Conclusion of the proof.

\smallskip

\noindent Let $A:= \Omega_1$ and $B:= \Omega_2 \cup \Omega_3$. By Lemma \ref{lem: connect} each $\Omega_i$ is connected; since $\Omega_2 \cap \Omega_3 = Z_1 \neq \emptyset$ by Step 2, also $B$ is connected, being the union of two connected sets with non-empty intersection. Moreover $A \cup B = \R^N$ by Step 1, while
\[
A \cap B = (\Omega_1 \cap \Omega_2) \cup (\Omega_1 \cap \Omega_3) = Z_3 \sqcup Z_2
\]
is the union of two non-empty \emph{disjoint} open sets, and hence is disconnected. This is impossible: $A$ and $B$ are open and connected, $A \cup B = \R^N$ and $A \cap B \neq \emptyset$, so that the Mayer-Vietoris sequence in reduced homology gives
\[
0 = H_1(\R^N) \longrightarrow \widetilde{H}_0(A \cap B) \longrightarrow \widetilde{H}_0(A) \oplus \widetilde{H}_0(B) = 0,
\]
whence $\widetilde{H}_0(A \cap B) = 0$, that is, $A \cap B$ is connected.
\end{proof}

\begin{lemma}\label{lem: homog}
Let $\mf{v} \in \mathcal{L}_{\loc}(\R^N)$ be globally $\nu$-H\"older continuous for some $\nu \in (0,1)$, and suppose that all the components of $\mf{v}$ are non-trivial. Then the set $\Gamma_{\mf{v}}^3$ is an affine space of dimension at most $N-2$, and $\mf{v}$ is $\nu$-homogeneous with respect to any of its points.
\end{lemma}

\begin{proof}
Let $x_0 \in \Gamma_{\mf{v}}^3$. At first, we show that $N(\mf{v},x_0,r) = \nu$ for every $r>0$. Then, by Proposition \ref{prop: alm seg}, $\mf{v}$ is $\nu$-homogeneous with respect to $x_0$.

Let us suppose by contradiction that $N(\mf{v},x_0,\bar r) = \nu+\eps$ for some $\eps>0$. By monotonicity, $N(\mf{v},x_0,r) \ge \nu+\eps$ for every $r>\bar r$, and by Lemma \ref{lem: doubling seg} we have that
\[
H(\mf{v},x_0,r) \ge \frac{H(\mf{v},x_0,\bar r)}{\bar r^{2(\nu+\eps)}} r^{2(\nu+\eps)} \qquad \forall r>\bar r,
\]
whereas we have also
\beq\label{H nu}
H(\mf{v},x_0,r) \le C r^{2 \nu} \qquad \forall r>0
\eeq
by the global $\nu$-H\"older continuity of $\mf{v}$. This gives a contradiction for large $r$, and shows that $N(\mf{v},x_0,r) \le \nu$ for every $r>0$. Similarly, if $N(\mf{v},x_0,\bar r) \le \nu-\eps$ for some $\eps>0$, then
\[
H(\mf{v},x_0,r) \ge \frac{H(\mf{v},x_0,\bar r)}{\bar r^{2(\nu-\eps)}} r^{2(\nu-\eps)} \qquad \forall r<\bar r,
\]
in contradiction with \eqref{H nu} for $r>0$ small enough.

The previous argument proves that $\mf{v}$ is $\nu$-homogeneous with respect to any point in $\Gamma_{\mf{v}}^3$. Therefore, given two points in the set, the line connecting them must be contained within the set itself, namely $\Gamma_{\mf{v}}^3$ must be an affine space of $\R^N$.

Suppose that $\Gamma_{\mf{v}}^3$ has dimension $N-1$. Then it is a hyperplane, and $\R^N \setminus \Gamma_{\mf{v}}^3$ consists of two open half-spaces $H_1$ and $H_2$. Let $x_1 \in H_1$ and $x_2 \in H_2$. Since $H(\mf{v},x_j,r) >0$ for every $r>0$ (see Proposition \ref{prop: alm seg}), in each neighborhood of $x_j$ one component must be non-trivial, say $v_{i_j} \not \equiv 0$ in $B_{\rho_j}(x_j)$ for $j=1,2$ (the index $i_j$ could be the same for $j=1,2$). We set $w_1 := v_{i_1} \chi_{H_1}$ and $w_2:= v_{i_2} \chi_{H_2}$. By definition, these are non-trivial subharmonic functions in $\R^N$ with disjoint positivity sets, and they satisfy the growth assumption of Proposition \ref{prop: liou seg full}, being globally H\"older continuous with exponent $\nu \in (0,1)$. Therefore one between $v_{i_1}$ and $v_{i_2}$ must vanish identically, a contradiction. This shows that the dimension of $\Gamma_{\mf{v}}^3$ is at most $N-2$.
\end{proof}

\begin{lemma}\label{lem: triple on sphere}
Let $N>2$, and let $\mf{v} \in H^1_{\loc}(\R^N,\R^3) \cap C^0(\R^N,\R^3)$ be a triplet of nonnegative subharmonic functions with the following properties:
\begin{itemize}
\item[($i$)] $v_i \not \equiv 0$ and $\Delta v_i = 0$ in $\{v_i>0\}$ for every $i$; moreover, $v_1 \,v_2\,v_3 \equiv 0$ in $\R^N$;
\item[($ii$)] $\mf{v}$ is $\nu$-homogeneous with respect to a point $x_0$, for some $\nu \in (0,1)$.
\end{itemize}
Then $S_1(x_0) \cap \Gamma_{\mf{v}}^3 \neq \emptyset$.
\end{lemma}

\begin{proof}
Without loss of generality $x_0 = 0$. Notice that, by ($ii$) and by continuity, $0 \le v_i(x) \le C|x|^{\nu}$ with $C = \max_{S_1}|\mf{v}|$, so that $\mf{v}$ satisfies the growth assumptions of Proposition \ref{prop: liou seg full} and of Lemma \ref{lem: connect}.

Suppose by contradiction that $\Gamma_{\mf{v}}^3 \cap S_1(0) = \emptyset$, and set
\[
\omega_i := \{v_i>0\} \cap S_1(0), \qquad \zeta_k := \omega_i \cap \omega_j \quad \text{for } \{i,j,k\} = \{1,2,3\},
\]
which are relatively open subsets of $S_1(0)$. Exactly as in Step 1 of the proof of Lemma \ref{lem: trip non emp} we have that
\[
\omega_1 \cup \omega_2 \cup \omega_3 = S_1(0), \qquad \omega_1 \cap \omega_2 \cap \omega_3 = \emptyset,
\]
and that $\zeta_1, \zeta_2, \zeta_3$ are pairwise disjoint.

Each $\zeta_k$ is non-empty. Indeed, if for instance $\zeta_1 = \omega_2 \cap \omega_3 = \emptyset$, then $v_2 v_3 \equiv 0$ on $S_1(0)$, and hence on the whole of $\R^N$ by homogeneity; Proposition \ref{prop: liou seg full} then implies that one between $v_2$ and $v_3$ vanishes identically, in contradiction with ($i$).

Moreover, each $\omega_i$ is connected: by homogeneity $\{v_i>0\}$ is a cone with vertex $0$, so that $\omega_i$ is connected if and only if $\{v_i>0\}$ is, and the latter holds by Lemma \ref{lem: connect}.

We can now conclude as in Step 3 of the proof of Lemma \ref{lem: trip non emp}. Let $A:= \omega_1$ and $B:= \omega_2 \cup \omega_3$: these are relatively open and connected subsets of $S_1(0)$ (for $B$, recall that $\omega_2 \cap \omega_3 = \zeta_1 \neq \emptyset$), with $A \cup B = S_1(0)$, while
\[
A \cap B = (\omega_1 \cap \omega_2) \cup (\omega_1 \cap \omega_3) = \zeta_3 \sqcup \zeta_2
\]
is disconnected. Since $N>2$, the sphere $S_1(0)$ is homeomorphic to $\mathbb{S}^{N-1}$ with $N-1 \ge 2$, and hence $H_1(S_1(0)) = 0$; the Mayer-Vietoris sequence in reduced homology then gives
\[
0 = H_1(S_1(0)) \longrightarrow \widetilde{H}_0(A \cap B) \longrightarrow \widetilde{H}_0(A) \oplus \widetilde{H}_0(B) = 0,
\]
so that $A \cap B$ is connected, a contradiction.
\end{proof}

\begin{lemma}\label{lem: 3/4}
Let $N \ge 2$, let $0 \le d \le N-2$, and let $\mf{v} \in \mathcal{L}_{\loc}(\R^N)$ be such that:
\begin{itemize}
\item[($i$)] all the components $v_i$ are non-trivial;
\item[($ii$)] $\mf{v}$ is $\nu$-homogeneous with respect to $0$, for some $\nu \in (0,1)$;
\item[($iii$)] there exists a linear subspace $V \subset \R^N$ of dimension $d$ such that $\mf{v}(x+e) = \mf{v}(x)$ for every $x \in \R^N$ and every $e \in V$.
\end{itemize}
Then $\nu \ge 3/4$.
\end{lemma}

\begin{proof}
We argue by induction on $k:= N-d \ge 2$.

\medskip

\noindent \emph{Step 1) $k=2$.} Let $P := V^{\perp}$, a $2$-dimensional plane, and let $\mf{g}:= \mf{v}|_{P}$. By ($iii$), $\mf{v}(x) = \mf{g}(\pi_P x)$, where $\pi_P$ denotes the orthogonal projection onto $P$; consequently $\mf{g}$ is a triplet of nonnegative continuous functions on $P \simeq \R^2$, with non-trivial components, harmonic in $\{g_i>0\}$, satisfying $g_1g_2g_3 \equiv 0$, and $\nu$-homogeneous with respect to $0$. In polar coordinates on $P$ we write $\mf{g}(r,\theta) = r^\nu \mf{h}(\theta)$, and each component $h_i$ satisfies
\[
-h_i''  = \nu^2 h_i \quad \text{in $\{h_i>0\} \cap S_{1}$}.
\]
By continuity and positivity of $g_i$, up to a rotation $h_i(\theta) = c \sin(\nu \theta)$ in each connected component of $\{h_i>0\} \cap S_{1}$, and any such component is an arc with opening $\pi/\nu$. If $\{h_i>0\}$ is disconnected for an index $i$, then at least two such arcs must stay in $S_1 \simeq [0,2\pi)$, whence
\[
2\frac{\pi}{\nu} \le 2\pi \quad \implies \quad \nu \ge 1.
\]
If instead $\{h_i>0\}$ is connected for every $i$, by partial segregation we deduce that the union of the complements $\{h_i=0\}$ must cover $S_1$, namely
\[
3\left(2\pi - \frac{\pi}{\nu}\right) \ge 2\pi \quad \implies \quad \nu \ge \frac34.
\]
In conclusion, $\nu \ge 3/4$.

\medskip

\noindent \emph{Step 2) $k \ge 3$.} Suppose that the thesis holds whenever $N-d = k-1$, and let $\mf{v}$, $V$ be as in the statement, with $N-d = k$.

Since $\nu>0$, homogeneity gives $\mf{v}(0)=0$, and thus $V \subset \Gamma_{\mf{v}}^3$ by ($iii$). Consider the restriction of $\mf{v}$ to $V^{\perp} \simeq \R^{k}$: by ($i$)--($iii$) it is a triplet of nonnegative subharmonic functions, harmonic where positive, satisfying the partial segregation condition, with non-trivial components, and $\nu$-homogeneous with respect to $0$. Since $k>2$, Lemma \ref{lem: triple on sphere} provides a point
\[
x_1 \in \Gamma_{\mf{v}}^3 \cap S_1(0) \cap V^{\perp}.
\]

Let $\gamma := N(\mf{v},x_1,0^+)$. By Lemma \ref{lem: N >=} we have $\gamma \ge \bar \nu >0$, while $\gamma \le \nu$ by Lemma \ref{lem: N infty seg} and the monotonicity of $N$. Let $\rho_m \to 0^+$, and let
\[
\mf{v}_m(x) := \frac{\mf{v}(x_1+\rho_m x)}{H(\mf{v},x_1,\rho_m)^{1/2}}
\]
be the associated blow-up sequence. By Proposition \ref{prop: conv blow-up}, up to a subsequence $\mf{v}_m \to \bar{\mf{v}}$ locally uniformly and in $H^1_{\loc}(\R^N)$, where $\bar{\mf{v}} \in \mathcal{L}_{\loc}(\R^N)$ is $\gamma$-homogeneous with respect to $0$ and has all non-trivial components.

We claim that $\bar{\mf{v}}$ is invariant with respect to the linear subspace $V \oplus \R x_1$, whose dimension is $d+1$ (recall that $x_1 \in V^\perp$, $|x_1|=1$). On the one hand, for every $e \in V$ we have $\rho_m e \in V$, so that by ($iii$)
\[
\mf{v}_m(x+e) = \frac{\mf{v}(x_1+\rho_m x + \rho_m e)}{H(\mf{v},x_1,\rho_m)^{1/2}} = \mf{v}_m(x),
\]
and the same identity passes to the limit. On the other hand, for every $h \in \R$,
\[
x_1+\rho_m \big(x+h(x_1+\rho_m x)\big) = (1+\rho_m h)\,(x_1+\rho_m x),
\]
whence, by the $\nu$-homogeneity of $\mf{v}$ with respect to $0$,
\[
\big|\mf{v}_m\big(x+h(x_1+\rho_m x)\big) - \mf{v}_m(x)\big| = \big|(1+\rho_m h)^{\nu}-1\big|\,\big|\mf{v}_m(x)\big| \longrightarrow 0
\]
locally uniformly in $x$, since $\{\mf{v}_m\}$ is locally uniformly bounded and $\rho_m \to 0^+$. As $x+h(x_1+\rho_m x) \to x+hx_1$, the local uniform convergence $\mf{v}_m \to \bar{\mf{v}}$ yields
\[
\bar{\mf{v}}(x+hx_1) = \bar{\mf{v}}(x) \qquad \text{for every $x \in \R^N$ and every $h \in \R$},
\]
and the claim is proved.

Since $k \ge 3$ we have $d+1 \le N-2$, and therefore $\bar{\mf{v}}$ satisfies the assumptions of the lemma with $N-(d+1) = k-1$. By the inductive assumption, $\gamma \ge 3/4$, and hence $\nu \ge \gamma \ge 3/4$.
\end{proof}

\begin{proof}[Proof of Proposition \ref{prop: imp liou}]
(1) Let $\mf{v}$ satisfy assumption (1) in the proposition, and suppose by contradiction that all the components $v_i$ are non-trivial. Since $\mf{v}$ is globally $\nu$-H\"older continuous, it satisfies the growth assumption of Lemma \ref{lem: trip non emp}, and hence $\Gamma_{\mf{v}}^3 \neq \emptyset$. Let $x_0 \in \Gamma_{\mf{v}}^3$: by Lemma \ref{lem: homog}, $\mf{v}$ is $\nu$-homogeneous with respect to $x_0$. Up to a translation we may assume that $x_0 = 0$, so that $\mf{v}$ satisfies the assumptions of Lemma \ref{lem: 3/4} with $d=0$ and $V=\{0\}$, and we deduce that $\nu \ge 3/4$, a contradiction.

\smallskip

\noindent (2) Let $\mf{v}$ satisfy assumption (2) in the proposition, and suppose by contradiction that all the components $v_i$ are non-trivial. By Lemma \ref{lem: blow-down}, there exist $\nu' \in (0,\nu]$ and $\mf{w} \in \mathcal{L}_{\loc}(\R^N)$, $\nu'$-homogeneous with respect to $0$ and with all non-trivial components. Then $\mf{w}$ satisfies the assumptions of Lemma \ref{lem: 3/4} with $d=0$ and $V = \{0\}$, whence $\nu' \ge 3/4$, in contradiction with $\nu' \le \nu < 3/4$.
\end{proof}

\subsection{Proof of Theorem \ref{thm: improved Holder bounds}}\label{sub: conc imp}

\begin{proof}[Proof of \eqref{eq: imp hol}]
As already observed in \cite[Remark 4.14]{ST24p1}, in order to prove the uniform H\"older estimate, the contradiction and blow-up schemes are exactly the same described in \cite[Section 4]{ST24p1}. The only difference is that, in \cite[Lemmas 4.12 and 4.13]{ST24p1}, we use Proposition \ref{prop: imp liou} instead of Corollary \ref{cor: liou seg} and Theorem \ref{thm: liou sub}. This allows to enlarge the threshold for $\alpha$ from $\bar \nu$ to $3/4$.
\end{proof}

Estimate \eqref{eq: imp hol} implies that functions in $\mathcal{L}_{\loc}(\Omega)$ are of class $C^{0,\alpha}(\Omega)$ for every $\alpha \in (0,3/4)$. To complete the proof of Theorem \ref{thm: improved Holder bounds}, it remains to show that actually they are of class $C^{0,3/4}(\Omega)$. This requires some preliminary lemmas.

\begin{lemma}\label{lem: lbn3}
Let $\mf{v} \in \mathcal{L}_{\loc}(\Omega)$. If $m(x_0) =3$, then $N(\mf{v},x_0,0^+) \ge 3/4$.
\end{lemma}

\begin{proof}
Thanks to Theorem \ref{thm: improved Holder bounds}, the proof is analogue to the one of Lemma \ref{lem: N >=}, with $\bar \nu$ replaced by $3/4$.
\end{proof}

\begin{lemma}\label{lem: decayopt}
Let $\mf{v} \in \mathcal{L}_{\loc}(\Omega)$, and let $K \subset \Omega$ be a compact set. There exists $C>0$ such that
\[
\frac{1}{r^{N-1}} \int_{S_r(x_0)} \sum_i v_i^2\,d\sigma \le C r^{3/2}
\]
for every $r \in (0, \dist(K, \pa \Omega) / 2 )$ and for every $x_0 \in \Gamma_{\mf{v}}^3 \cap K$.
\end{lemma}

\begin{proof}
One can proceed exactly as in the proof of Lemma \ref{lem: decay1L}, using Lemma \ref{lem: lbn3} instead of Lemma \ref{lem: N >=}.
\end{proof}

\begin{proof}[Conclusion of the proof of Theorem \ref{thm: improved Holder bounds}]
By contradiction, suppose that $\mf{v} \in \mathcal{L}_{\loc}(\Omega)$ but $\mf{v} \not \in C^{3/4}(K)$, for a compact set $K \subset \Omega$. Then there exist sequences $\{x_m\}$ and $\{y_m\}$ in $K$ such that (up to a relabelling) $x_m \neq y_m$ and
\[
\frac{|v_{1}(x_m)-v_1(y_m)|}{|x_m-y_m|^{3/4}} \to +\infty.
\]
Let $r_m:= |x_m-y_m|$ and $2R_m:= \max\{\dist(x_m,\Gamma_{\mf{v}}^3), \dist(\,y_m\,, \Gamma_{\mf{v}}^3)\} = \dist(x_m,\Gamma_{\mf{v}}^3)$. We can consider the same four cases as in Lemma \ref{lem: bu hol}, and reach a contradiction in each of them following the same strategy (by using Lemma \ref{lem: decayopt} instead of Lemma \ref{lem: decay1L}).
\end{proof}

\section{Further results on the nodal set in dimension 2}\label{sec: fb dim 2}

In this section, we prove further properties of the nodal set in dimension $N=2$. We consider the problem with fixed traces in a simply connected bounded domain $\Omega \subset \R^2$ with regular boundary. The starting points of our analysis are Theorem \ref{thm: nodal set}, Remark \ref{rem: part nodal}, and the basic properties of $\Gamma_{\mf{v}}^3$ collected in Remark \ref{rem: on triple}.

In order to state our main result, we make some assumptions on the boundary datum.

\begin{definition}\label{def: simple boundary}
We say that a function $\boldsymbol{\psi} =(\psi_1,\psi_2,\psi_3) \in \mathrm{Lip}(\overline{\Omega})$, with $\psi_1\,\psi_2\,\psi_3 \equiv 0$ in $\overline{\Omega}$, is a \emph{simple boundary datum}, if
\begin{itemize}
\item[($i$)] $\{\psi_i>0\} \cap \pa \Omega \neq \emptyset$ is connected for every $i$;
\item[($ii$)] $\Gamma_{\boldsymbol{\psi}} \cap \pa \Omega$ contains exactly three points, one from $\{\psi_1=0=\psi_2\}$, one from $\{\psi_1=0=\psi_3\}$, and one from $\{\psi_2=0=\psi_3\}$.
\item[($iii$)] $\psi_i \in C^{1,\alpha}(\supp \,\psi_i \cap \pa \Omega)$, and moreover, if $x_0 \in \{\psi_i=0=\psi_j\}$ for two indexes $i$ and $j$, then $\nabla_{\theta} \psi_i(x_0) = - \nabla_{\theta} \psi_j(x_0) \neq 0$. \end{itemize}
\end{definition}
We also require $\partial\Omega$ to be of class $C^{1,\alpha}$ in a neighborhood of each of the three boundary points in $(ii)$. The derivatives in $(iii)$ are taken from the respective supports, with the same orientation of the arclength coordinate.

In particular, $\Gamma_{\boldsymbol{\psi}}^3 \cap \pa \Omega=\emptyset$. A prototypical example is the function defined by \eqref{tr mer} on $S_1$ (conveniently extended within $B_1$).

Given a simple boundary datum, we consider the minimization problem
\beq\label{def cinfty'}
c_\infty:= \min\left\{ \sum_{j=1}^3 \int_{\Omega} |\nabla u_j|^2\,dx\left|\begin{array}{l}  \mf{u}-\boldsymbol{\psi} \in H_0^1(\Omega,\R^3) \\ u_1\,u_2\,u_3 \equiv 0 \ \text{in $\Omega$}\end{array}\right.\right\}.
\eeq
Theorem 1.2 in \cite{ST24p1} ensures the existence of a minimizer $\mf{v}$, and the fact that $\mf{v} \in C^{0,\alpha}(\overline{\Omega})$ for some $\alpha >0$. We denote by $\omega_i:=\{v_i>0\}$ the positivity set of $v_i$, and we define the \emph{interfaces}
\beq\label{def Gamma ij}
\Gamma_{ij}:= \pa \omega_i \cap \pa \omega_j \cap \{x \in \Omega: \ m(x) = 2\}.
\eeq
We say that $\omega_i$ and $\omega_j$ are \emph{adjacent} if $\Gamma_{ij} \neq \emptyset$.

Since the boundary datum has no triple zero, compactness gives
\[
\min_{\partial\Omega}(\psi_1+\psi_2+\psi_3)>0.
\]
By continuity of $\mf v$ on $\overline\Omega$, the sum $v_1+v_2+v_3$ is positive in an interior neighborhood of the boundary. Thus all triple points lie in a compact subset of $\Omega$, and their local finiteness in dimension two implies that there are only finitely many.

Let $p$ be one of the three transition points, say $\psi_i(p)=\psi_j(p)=0$. The remaining component $v_k$ is positive near $p$, so $v_iv_j=0$ there and $h=v_i-v_j$ is harmonic in $\Omega$ near $p$ by Lemma~\ref{cor: diff arm}. On the two sides of $p$ along $\partial\Omega$, the trace of $h$ is given by $\psi_i$ and $-\psi_j$, respectively. Condition $(iii)$ makes this trace $C^{1,\alpha}$ across $p$, with $\partial_\tau h(p)\ne0$. Local boundary Schauder estimates for the Dirichlet problem give $h\in C^{1,\alpha}$ up to the boundary near $p$. In a boundary chart, the implicit function theorem therefore represents $\{h=0\}$ as a single arc transverse to $\partial\Omega$. This arc coincides with $\Gamma_{\mf v}$ near $p$. At every other boundary point two components are strictly positive, and remain so in a neighborhood. By compactness there is a collar of the boundary in which the nodal set consists of the three regular arcs just described. Consequently the entire singular set lies in a compact subset of $\Omega$ and is finite by Theorem~\ref{thm: nodal set}.

Furthermore, we observe that $\Gamma_{\mf{v}}$ may contain some piecewise smooth \emph{loops}, defined as the boundary of an open non-empty region which does not contain points of $\Gamma_{\mf{v}}$, or the bouquet of piecewise smooth loops attached together at singular points (possibly sharing some regular arcs), each of which containing an open region where one component \(v_i = 0\) (we cannot have a nodal loop for $v_i$ containing a region where $v_i>0$, since this would contradict the harmonicity of $v_i$ and the maximum principle). It remains to determine whether there can be infinitely many such loops.

\begin{remark}
The fact that any loop $\gamma$ of $\Gamma_{\mf{v}}$ surrounds a region where one component $v_i=0$ is an easy consequence of Theorem \ref{thm: nodal set}, see in particular Remark \ref{rem: part nodal}. Indeed, if this were not true, then there were (at least) three arcs on $\gamma$ belonging to $\Gamma_{12}$, $\Gamma_{13}$ and $\Gamma_{23}$. Inside the loop, close to the arc of $\Gamma_{ij}$, one component between $v_i$ and $v_j$ must vanish. Thus, at least two components must vanish inside the loops, say $v_1$ and $v_2$. Since $\mathrm{int}\{v_1=0\} \cap \mathrm{int}\{v_2=0\}=\emptyset$, this forces the existence of another arc of $\Gamma_{12} \subset \Gamma_{\mf{v}}$ in the open region surrounded by $\gamma$, in contradiction with the definition of loop.
\end{remark}

In the main result of this section, we answer the above issues.

\begin{theorem}\label{thm: fb dim 2 1}
Let $\boldsymbol{\psi}$ be a simple boundary datum. Then $\omega_i$ is connected for every $i=1,2,3$, and $\Gamma_{\mf{v}}$ has finitely many singular points, and finitely many loops.
\end{theorem}

The proof is divided into intermediate lemmas.

\begin{lemma}\label{lem: connect 1}
If $\boldsymbol{\psi}$ is a simple boundary datum, the set $\omega_i:=\{v_i>0\}$ is connected for every $i$.
\end{lemma}

\begin{proof}
By contradiction, let $\omega_i= A \cup B$, with $A$ and $B$ disjoint, open, and non-empty. Recall that $v_i$ is continuous on $\overline{\Omega}$, hence it vanishes (continuously) on $\partial \omega_i \setminus  \{\psi_i > 0\}$. Since $\{\psi_i>0\}$ is connected on $\pa \Omega$, and $v_i$ is strictly positive there, it cannot intersect both $A$ and $B$. We infer that $v_i$ vanishes, for instance, on $\pa B$. But $v_i \ge 0$ is a harmonic function in $B$, continuous up to the boundary, and such that $v_i=0$ on $\pa B$; then $v_i \equiv 0$ in $B$, a contradiction.
\end{proof}

\begin{lemma}\label{lem: 2 triple on loops}
On any loop of $\pa \omega_i$ there must be at least two triple points.
\end{lemma}
\begin{proof}
Let $\alpha \subset \pa \omega_i$ be a loop for some index $i$, say $i=1$. Being a nodal point, at each point of $\alpha$ another component of $\mf{v}$ must vanish. If $v_2=0$ on $\alpha$, then recalling that the nodal set $\Gamma_{\mf{v}}$ has empty interior, and that $v_1=0$ in the region surrounded by $\alpha$, we have that $\alpha$ must be the boundary of a positivity set of $v_2$, which is not possible (since $v_2$ is harmonic where positive). Thus, there must be at least an arc of $\alpha$ where $v_3=0$ and $v_2>0$, and at least an arc of $\alpha$ where $v_2=0$ and $v_3>0$. Any of these arcs is delimited by two triple points.
\end{proof}

\begin{lemma}\label{lem: loops and triple}
The number of loops in $\Gamma_{\mf{v}}$ is finite.
\end{lemma}

\begin{proof}
This follows directly from Lemma \ref{lem: 2 triple on loops} and from the fact that $\Gamma_{\mf{v}}^3$ is finite: if the number of loops $m \to +\infty$, then we must have infinitely many arcs connecting a finite number of triple points (indeed, the same finitely many triple points must belong to infinitely many loops). In particular, there are two triple points connected by three arcs of regular points. These arcs must belong to at least two different $\Gamma_{ij}$. In any case, for all the possible combinations of indexes we find an open non-empty region where two components vanish identically (see Figure \ref{fig:3}), in contradiction with the fact that $\mathrm{int}\{v_i=0\} \cap \mathrm{int}\{v_j=0\} = \emptyset$, being the nodal set of dimension $N-1$.
\end{proof}

\begin{figure}[ht]
\centering

\tikzset{every picture/.style={line width=0.75pt}}

\begin{tikzpicture}[x=0.75pt,y=0.75pt,yscale=-1,xscale=1]

\draw [color={rgb, 255:red, 208; green, 2; blue, 27 }  ,draw opacity=1 ][line width=1.5]  [dash pattern={on 5.63pt off 4.5pt}]  (202,152) .. controls (152,114) and (158,97) .. (201,79) ;
\draw [line width=1.5]    (201,79) .. controls (244,67) and (254,147) .. (202,152) ;
\draw [shift={(202,152)}, rotate = 174.51] [color={rgb, 255:red, 0; green, 0; blue, 0 }  ][fill={rgb, 255:red, 0; green, 0; blue, 0 }  ][line width=1.5]      (0, 0) circle [x radius= 4.36, y radius= 4.36]   ;
\draw [shift={(201,79)}, rotate = 344.41] [color={rgb, 255:red, 0; green, 0; blue, 0 }  ][fill={rgb, 255:red, 0; green, 0; blue, 0 }  ][line width=1.5]      (0, 0) circle [x radius= 4.36, y radius= 4.36]   ;
\draw [color={rgb, 255:red, 208; green, 2; blue, 27 }  ,draw opacity=1 ][line width=1.5]  [dash pattern={on 5.63pt off 4.5pt}]  (201,79) .. controls (272,42) and (356,157) .. (202,152) ;
\draw [color={rgb, 255:red, 208; green, 2; blue, 27 }  ,draw opacity=1 ][line width=1.5]  [dash pattern={on 5.63pt off 4.5pt}]  (404,152) .. controls (354,114) and (360,97) .. (403,79) ;
\draw [line width=1.5]    (403,79) .. controls (446,67) and (456,147) .. (404,152) ;
\draw [shift={(404,152)}, rotate = 174.51] [color={rgb, 255:red, 0; green, 0; blue, 0 }  ][fill={rgb, 255:red, 0; green, 0; blue, 0 }  ][line width=1.5]      (0, 0) circle [x radius= 4.36, y radius= 4.36]   ;
\draw [shift={(403,79)}, rotate = 344.41] [color={rgb, 255:red, 0; green, 0; blue, 0 }  ][fill={rgb, 255:red, 0; green, 0; blue, 0 }  ][line width=1.5]      (0, 0) circle [x radius= 4.36, y radius= 4.36]   ;
\draw [color={rgb, 255:red, 74; green, 144; blue, 226 }  ,draw opacity=1 ][line width=1.5]    (403,79) .. controls (404.45,76.44) and (406.23,75.61) .. (408.35,76.51) .. controls (410.36,77.5) and (411.79,76.97) .. (412.64,74.9) .. controls (413.61,72.85) and (415.4,72.34) .. (418.02,73.37) .. controls (419.78,74.69) and (421.21,74.4) .. (422.3,72.51) .. controls (423.51,70.64) and (425.28,70.42) .. (427.63,71.85) .. controls (429.12,73.43) and (430.52,73.37) .. (431.85,71.66) .. controls (434,70.01) and (435.73,70.06) .. (437.05,71.8) .. controls (438.91,73.65) and (440.61,73.83) .. (442.15,72.35) .. controls (444.47,71.06) and (446.13,71.37) .. (447.13,73.28) .. controls (448.63,75.37) and (450.24,75.8) .. (451.95,74.57) .. controls (454.38,73.62) and (455.93,74.15) .. (456.6,76.18) .. controls (457.73,78.45) and (459.22,79.1) .. (461.06,78.11) .. controls (463.55,77.5) and (465.24,78.39) .. (466.12,80.78) .. controls (466.31,82.84) and (467.62,83.67) .. (470.06,83.28) .. controls (472.56,83.03) and (473.79,83.94) .. (473.74,86.01) .. controls (474.01,88.42) and (475.35,89.59) .. (477.75,89.54) .. controls (480.19,89.63) and (481.37,90.89) .. (481.3,93.31) .. controls (481.04,95.63) and (482.06,96.95) .. (484.35,97.28) .. controls (486.65,97.77) and (487.48,99.15) .. (486.84,101.4) .. controls (486.01,103.49) and (486.65,104.89) .. (488.75,105.62) .. controls (490.98,107.09) and (491.46,108.76) .. (490.19,110.64) .. controls (488.76,112.27) and (488.93,113.94) .. (490.7,115.66) .. controls (492.31,117.23) and (492.16,118.89) .. (490.23,120.64) .. controls (488.29,121.57) and (487.78,123.19) .. (488.71,125.5) .. controls (489.52,127.65) and (488.79,128.98) .. (486.53,129.51) .. controls (484.28,129.76) and (483.27,131.04) .. (483.49,133.33) .. controls (483.38,135.78) and (482.06,136.98) .. (479.54,136.93) .. controls (477.06,136.68) and (475.73,137.61) .. (475.54,139.72) .. controls (474.51,142.23) and (472.96,143.09) .. (470.87,142.29) .. controls (468.82,141.41) and (467.4,142.04) .. (466.62,144.18) .. controls (465.67,146.33) and (464.1,146.91) .. (461.92,145.9) .. controls (459.76,144.83) and (458.03,145.34) .. (456.74,147.44) .. controls (455.28,149.53) and (453.39,149.97) .. (451.08,148.78) .. controls (449.77,147.35) and (448.24,147.64) .. (446.51,149.66) .. controls (445.72,151.48) and (444.11,151.73) .. (441.66,150.41) .. controls (440.33,148.9) and (438.61,149.11) .. (436.52,151.04) .. controls (435.49,152.81) and (433.68,152.98) .. (431.08,151.53) .. controls (429.72,149.96) and (428.46,150.04) .. (427.29,151.78) .. controls (424.74,153.57) and (422.76,153.66) .. (421.35,152.04) .. controls (419.95,150.41) and (418.57,150.44) .. (417.22,152.13) .. controls (415.82,153.81) and (413.67,153.81) .. (410.77,152.14) .. controls (409.32,150.45) and (407.82,150.43) .. (406.29,152.07) -- (404,152) ;
\draw (146,104.4) node [anchor=north west][inner sep=0.75pt]    {\footnotesize $\Gamma_{12}$};
\draw (217,105.4) node [anchor=north west][inner sep=0.75pt]    {\footnotesize $ \Gamma_{13}$};
\draw (265,104.4) node [anchor=north west][inner sep=0.75pt]    {\footnotesize $\Gamma_{12}$};
\draw (348,104.4) node [anchor=north west][inner sep=0.75pt]    {\footnotesize $\Gamma_{12}$};
\draw (419,105.4) node [anchor=north west][inner sep=0.75pt]    {\footnotesize $\Gamma_{13}$};
\draw (467,104.4) node [anchor=north west][inner sep=0.75pt]    {\footnotesize $\Gamma_{23}$};
\end{tikzpicture}
\caption{\small In both pictures we have two triple points connected by three regular arcs. On the left, the region delimited by the two arcs of $\Gamma_{12}$ is a loop delimiting a region where both $v_1$ and $v_2$ must vanish identically. This is a not possible. On the right, the arcs of $\Gamma_{12}$ and of $\Gamma_{23}$ delimit a region where necessarily $v_2 \equiv 0$, by harmonicity. On the other hand, the two arcs of $\Gamma_{12}$ and of $\Gamma_{13}$ delimit a region where $v_1 \equiv 0$, for the same reason. As a consequence, we find an open non-empty region where two components vanish identically, which is not possible again.}\label{fig:3}
\end{figure}
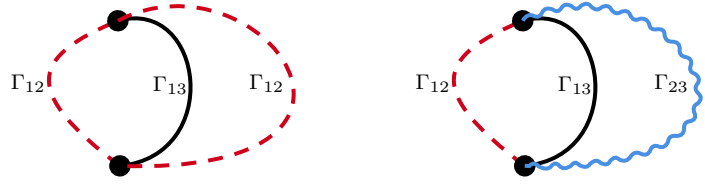

\begin{proof}[Proof of Theorem \ref{thm: fb dim 2 1}]
{Lemma~\ref{lem: connect 1} gives the connectedness of each $\omega_i$. The collar argument above and Theorem~\ref{thm: nodal set} give finiteness of the singular set. Finiteness of the loops follows from Lemma~\ref{lem: loops and triple}. Figure~\ref{fig:2} illustrates how a singular double point can occur on the boundary of a loop.}
\end{proof}

\begin{figure}[h]
\centering

\tikzset{every picture/.style={line width=0.75pt}}

\begin{tikzpicture}[x=0.75pt,y=0.75pt,yscale=-0.7,xscale=0.7]

\draw [line width=1.5]    (193.69,82.34) .. controls (221.93,87.4) and (239.88,173.17) .. (276.6,145.5) ;
\draw [line width=1.5]    (195,157.49) .. controls (227.64,129.82) and (224.38,102.15) .. (257.02,74.49) ;
\draw [line width=1.5]    (174.6,115.99) .. controls (211.32,113.22) and (248.04,122.44) .. (274.15,107.69) ;
\draw [line width=1.5]    (388.39,70.51) .. controls (427.24,76.84) and (451.95,179.48) .. (502.49,144.88) ;
\draw [line width=1.5]    (390.18,164.49) .. controls (435.11,129.89) and (430.61,95.29) .. (475.54,60.7) ;
\draw [line width=1.5]    (362.1,112.59) .. controls (412.64,109.13) and (463.18,120.66) .. (499.12,102.21) ;
\draw [line width=1.5]    (475.54,60.7) .. controls (513.72,34.17) and (523.83,88.37) .. (499.12,102.21) ;
\draw [line width=1.5]    (362.1,112.59) .. controls (302.58,119.51) and (348.63,63) .. (388.39,70.51) ;

\draw (190.96,95.03) node [anchor=north west][inner sep=0.75pt]    {$1$};
\draw (222.78,134.68) node [anchor=north west][inner sep=0.75pt]    {$1$};
\draw (245.63,91.34) node [anchor=north west][inner sep=0.75pt]    {$1$};
\draw (218.7,77.5) node [anchor=north west][inner sep=0.75pt]    {$2$};
\draw (192.59,124.54) node [anchor=north west][inner sep=0.75pt]    {$2$};
\draw (251.34,120.85) node [anchor=north west][inner sep=0.75pt]    {$2$};
\draw (386.88,88.28) node [anchor=north west][inner sep=0.75pt]    {$1$};
\draw (430.68,137.87) node [anchor=north west][inner sep=0.75pt]    {$1$};
\draw (462.13,83.67) node [anchor=north west][inner sep=0.75pt]    {$1$};
\draw (425.06,66.37) node [anchor=north west][inner sep=0.75pt]    {$2$};
\draw (389.12,125.19) node [anchor=north west][inner sep=0.75pt]    {$2$};
\draw (469.99,120.57) node [anchor=north west][inner sep=0.75pt]    {$2$};

\end{tikzpicture}
\caption{\small The local picture and the global picture of an admissible singular point with multiplicity $2$. The numbers $1$ and $2$ identify the regions where $v_1=0$ or $v_2=0$, respectively. The connectedness of the positivity sets $\{v_i>0\}$ forces the occurrence of loops.}\label{fig:2}
\end{figure}
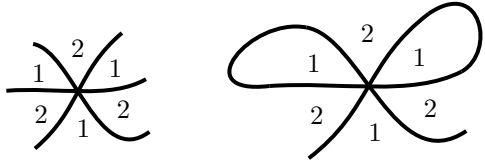

\section{Optimal regularity}\label{sec: mer}

In this section we prove Theorem \ref{prop: mer min}. We argue by contradiction and suppose that $\tilde{\mf{v}}$ is not a minimizer for problem \eqref{min fix traces intro 1} in $B_1$ with fixed traces given by \eqref{tr mer}:
\[
c_\infty:= \min\left\{ \sum_{j=1}^3 \int_{B_1} |\nabla u_j|^2\,dx\left|\begin{array}{l}  \mf{u}-\boldsymbol{\psi} \in H_0^1(B_1,\R^3) \\ u_1\,u_2\,u_3 \equiv 0 \ \text{in $B_1$}\end{array}\right.\right\}.
\]
We denote by
\[
\tilde c:= J(\tilde{\mf{v}},B_1), \quad \text{where} \quad J(\mf{u},\Omega):=\sum_{i=1}^3 \int_{\Omega} |\nabla u_i|^2\,dx,
\]
and, in view of the fact that $\tilde{\mf{v}}$ is not a minimizer, we observe that $c_\infty<\tilde c$. Moreover, we recall that the value $c_\infty$ is achieved by a minimizer $\mf{v}$ (see \cite[Theorem 1.2]{ST24p1}).

\begin{remark}\label{rem: scale min}
For any $x_0 \in B_1$ and $r>0$ small, let
\[
\boldsymbol{\psi}_{x_0,r}(x) := r^{3/4}\boldsymbol{\psi}\left(\frac{x-x_0}{r}\right).
\]
A simple scaling argument shows that
\[
\min\left\{J(\mf{u},B_r(x_0))\left|\begin{array}{l}  \mf{u}-\sigma \boldsymbol{\psi}_{x_0,r} \in H_0^1(B_r(x_0),\R^3) \\ u_1\,u_2\,u_3 \equiv 0 \ \text{in $B_r(x_0)$}\end{array}\right.\right\} = \sigma^2 c_\infty\, r^{3/2},
\]
for any $\sigma>0$. Furthermore $\tilde{\mf{v}}(x-x_0) = \boldsymbol{\psi}_{x_0,r}(x)$ on $S_r(x_0)$, and
\[
J(\sigma \tilde{\mf{v}}(\cdot\,-x_0),B_r(x_0)) = \sigma^2 \tilde c \, r^{3/2} .
\]
\end{remark}

The results of Section \ref{sec: fb dim 2} apply in the present setting, since $\boldsymbol{\psi}$ defined by \eqref{tr mer} on $S_1$ (conveniently extended within $B_1$) is a simple boundary datum, according to Definition \ref{def: simple boundary}. Therefore, summarizing what we know from the results already proven, the set $\omega_i =\{v_i>0\}$ is connected, $\pa \omega_i$ has a finite number of loops for every $i$, and $\Gamma_{\mf{v}}$ has a finite number of singular points in $B_1$. The boundary $\pa \omega_i$ has one component, say $\gamma_i$, intersecting the boundary in the two points of $\pa \{\psi_i>0\} \cap \pa B_1$. We can fix an orientation on each $\gamma_i$, in such a way that $\omega_i$ stays on the right of $\gamma_i$, and we call $\tilde \omega_i$ such region on the right, which may contain $\omega_i$ with strict inclusion. Regarding the set $\pa \omega_i \setminus \gamma_i$, it consists of a finite number of connected components $\sigma_i^1, \dots, \sigma_i^{m_i}$, each of which consists of a finite number of loops connected through singular points; these loops enclose closed regions where $v_i=0$. Both $\gamma_i$ and $\sigma_i^j$ are piecewise smooth curves, at each point of $\pa\omega_i$ at least one component different from $v_i$ vanishes, and the set of simgular points (including $\Gamma_{\mf{v}}^3$) is finite. In particular, $\gamma_1$, $\gamma_2$ and $\gamma_3$ cannot have a common arc. Finally, we recall the definition of $\Gamma_{ij}$ given by \eqref{def Gamma ij}.

\begin{lemma}\label{lem: nod mer}
There exists a point $x_0 \in B_1$ with multiplicity $3$ for the function $\mf{v}$, and $r_0>0$ sufficently small, such that the nodal set $\Gamma_{\mf{v}} \cap B_{r_0}(x_0)$ consists of exactly three regular arcs, one of $\Gamma_{12}$, one of $\Gamma_{13}$, one of $\Gamma_{23}$, meeting at $x_0$; furthermore, each of these arcs intersects $\pa B_{r_0}(x_0)$ transversally.
\end{lemma}

\begin{proof}
Notice that there must be a triple point $x_0$ in $\gamma_1 \cap \gamma_2 \cap \gamma_3$: indeed, one of the endpoints of $\gamma_1$ belongs to $\gamma_{2}$, and the other to $\gamma_{3}$; if on the other hand there were no triple point, then by continuity either $\gamma_1 =\gamma_2$, or $\gamma_{1} =\gamma_3$, a contradiction.

In addition any component of $\pa \omega_2$ cannot ``cross" a component of $\pa \omega_1$, in the following sense: denoting by $A_i^j$ the region delimited by $\sigma_i^j$, we have that
\[
\gamma_2 \cap (B_1 \setminus \tilde \omega_1)= \emptyset, \quad \gamma_2 \cap A_1^j= \emptyset, \quad \sigma_2^j \cap A_1^k = \emptyset
\]
for every $j$ and $k$. Indeed, if this were not true, then we would find a non-empty open set where both $v_1$ and $v_2$ vanish identically (see Figure \ref{fig:7}), in contradiction with the fact that the set of points with multiplicity $2$ has dimension $1$ (Theorem \ref{thm: nodal set}). Clearly, the same holds if we consider other pair of indexes.

\begin{figure}[ht]

\tikzset{every picture/.style={line width=0.75pt}}

\begin{tikzpicture}[x=0.75pt,y=0.75pt,yscale=-0.7,xscale=0.7]

\draw   (434,130.65) .. controls (434.16,93.29) and (464.58,63.14) .. (501.95,63.3) .. controls (539.31,63.46) and (569.46,93.88) .. (569.3,131.25) .. controls (569.14,168.61) and (538.72,198.76) .. (501.35,198.6) .. controls (463.99,198.44) and (433.84,168.02) .. (434,130.65) -- cycle ;
\draw   (256,130.95) .. controls (256,93.59) and (286.29,63.3) .. (323.65,63.3) .. controls (361.01,63.3) and (391.3,93.59) .. (391.3,130.95) .. controls (391.3,168.31) and (361.01,198.6) .. (323.65,198.6) .. controls (286.29,198.6) and (256,168.31) .. (256,130.95) -- cycle ;
\draw   (82,129.95) .. controls (82,92.59) and (112.29,62.3) .. (149.65,62.3) .. controls (187.01,62.3) and (217.3,92.59) .. (217.3,129.95) .. controls (217.3,167.31) and (187.01,197.6) .. (149.65,197.6) .. controls (112.29,197.6) and (82,167.31) .. (82,129.95) -- cycle ;
\draw [color={rgb, 255:red, 208; green, 2; blue, 27 }  ,draw opacity=1 ][line width=1.5]    (108.6,183.3) .. controls (131.6,119.3) and (136.6,127.3) .. (162.6,124.3) ;
\draw [color={rgb, 255:red, 208; green, 2; blue, 27 }  ,draw opacity=1 ][line width=1.5]    (282.6,184.3) .. controls (291.6,109.3) and (344.6,124.3) .. (391.3,130.95) ;
\draw [color={rgb, 255:red, 74; green, 144; blue, 226 }  ,draw opacity=1 ][line width=1.5]  [dash pattern={on 5.63pt off 4.5pt}]  (109.6,75.3) .. controls (137.6,194.3) and (144.6,143.3) .. (162.6,124.3) ;
\draw [color={rgb, 255:red, 189; green, 16; blue, 224 }  ,draw opacity=1 ][line width=1.5]    (162.6,124.3) .. controls (178.6,123.3) and (208.6,130.3) .. (217.3,129.95) ;
\draw  [color={rgb, 255:red, 74; green, 144; blue, 226 }  ,draw opacity=1 ][dash pattern={on 5.63pt off 4.5pt}][line width=1.5]  (321,123.15) .. controls (321,109.43) and (327.85,98.3) .. (336.3,98.3) .. controls (344.75,98.3) and (351.6,109.43) .. (351.6,123.15) .. controls (351.6,136.87) and (344.75,148) .. (336.3,148) .. controls (327.85,148) and (321,136.87) .. (321,123.15) -- cycle ;
\draw  [color={rgb, 255:red, 208; green, 2; blue, 27 }  ,draw opacity=1 ][line width=1.5]  (457.63,129.96) .. controls (457.63,115.78) and (463.43,104.29) .. (470.6,104.29) .. controls (477.77,104.29) and (483.58,115.78) .. (483.58,129.96) .. controls (483.58,144.14) and (477.77,155.64) .. (470.6,155.64) .. controls (463.43,155.64) and (457.63,144.14) .. (457.63,129.96) -- cycle ;
\draw  [color={rgb, 255:red, 74; green, 144; blue, 226 }  ,draw opacity=1 ][dash pattern={on 5.63pt off 4.5pt}][line width=1.5]  (446.6,129.96) .. controls (446.6,122.97) and (457.35,117.3) .. (470.6,117.3) .. controls (483.85,117.3) and (494.6,122.97) .. (494.6,129.96) .. controls (494.6,136.96) and (483.85,142.63) .. (470.6,142.63) .. controls (457.35,142.63) and (446.6,136.96) .. (446.6,129.96) -- cycle ;

\draw (119,160.4) node [anchor=north west][inner sep=0.75pt]   {\footnotesize $\gamma_1$};
\draw (122,78.4) node [anchor=north west][inner sep=0.75pt]    {\footnotesize $\gamma_2$};
\draw (294,157.4) node [anchor=north west][inner sep=0.75pt]    {\footnotesize $\gamma_1$};
\draw (300,97.4) node [anchor=north west][inner sep=0.75pt]    {\footnotesize $\sigma_2^j$};
\draw (482,146.4) node [anchor=north west][inner sep=0.75pt]    {\footnotesize $\sigma_1^k$};
\draw (500,118.4) node [anchor=north west][inner sep=0.75pt]    {\footnotesize $\sigma_2^j$};

\end{tikzpicture}
\caption{\small It is not possible that two components of $\pa \omega_i$ and $\pa \omega_j$ cross each other, otherwise we always find an open non-empty region where $v_i=0$ and $v_j=0$.}
\label{fig:7}
\end{figure}
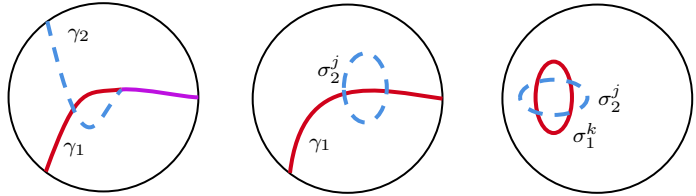

As a consequence, there exists a triple point $x_0$ where all the curves $\gamma_i$ intersect without crossing each other; moreover, since both the loops and the singular points of $\Gamma_{\mf{v}}$ are in finite number (Theorem \ref{thm: fb dim 2 1}), there exists a neighborhood $B_{r_0}(x_0)$ of $x_0$ such that $(\Gamma_{\mf{v}} \cap B_{r_0}(x_0))  = (\gamma_1 \cup \gamma_2 \cup \gamma_3) \cap B_{r_0}(x_0)$, and $x_0$ is the only singular point therein. There are two possibilities: either $x_0$ is a simple point for each curve $\gamma_i$ (Case 1), or $x_0$ is a self-intersection point for (at least) one of the curves (Case 2).

\smallskip

\noindent \underline{Case 1.} In this case, the thesis of the lemma holds: we have exactly $3$ arcs, one of $\Gamma_{12}$, one of $\Gamma_{13}$, one of $\Gamma_{23}$, meeting at $x_0$. Moreover, recalling that the curves are piecewise regular, provided that we suitably restrict the choice of $r_0$, we can assume that the transversality property of the statement is satisfied.

\smallskip

\noindent \underline{Case 2.}  Assume now that $x_0$ is a self-intersection point for one of the curves, say $\gamma_1$. This means that $\gamma_1$ contains a loop starting from $x_0$. By Lemma \ref{lem: 2 triple on loops}, there exists at least a second triple point $x_1$ on $\gamma_1$, and two relatively open arcs of $\gamma_1$ belonging to $\Gamma_{12}$ and $\Gamma_{13}$, respectively. Now, for $x_1$, we have the same alternatives already considered for $x_0$: either $x_1$ is a simple point for each curve $\gamma_i$, or $x_1$ is a self-intersection point for (at least) one of the curves. In the latter case, we have a loop for one curve $\gamma_i$ starting at $x_1$, and such that on the loop there is another triple point $x_2$. If $x_2 = x_0$, then we find an open non-empty region where 2 components vanish, which is not
possible; by the same argument, the triple point reached at any step of the iteration cannot coincide with one of those found at the previous steps, since otherwise the loops constructed so far would close up, delimiting an open non-empty region where two components vanish identically. Therefore the iteration produces mutually different triple points, and we deduce that either after finitely many steps we return to Case 1, and the proof is complete, or $\Gamma^3_{\mathbf v}$ contains infinitely many points in $B_1$. By Theorem \ref{thm: fb dim 2 1}, however, this is not possible.
\end{proof}

\begin{figure}[ht]
\centering

\tikzset{every picture/.style={line width=0.75pt}}

\tikzset{every picture/.style={line width=0.75pt}}

\begin{tikzpicture}[x=0.75pt,y=0.75pt,yscale=-0.7,xscale=0.7]

\draw [line width=1.5]  [dash pattern={on 5.63pt off 4.5pt}]  (192.45,192.46) .. controls (136.45,157.43) and (135.45,117.08) .. (193.45,105.4) ;
\draw [shift={(193.45,105.4)}, rotate = 348.61] [color={rgb, 255:red, 0; green, 0; blue, 0 }  ][fill={rgb, 255:red, 0; green, 0; blue, 0 }  ][line width=1.5]      (0, 0) circle [x radius= 4.36, y radius= 4.36]   ;
\draw [shift={(192.45,192.46)}, rotate = 212.03] [color={rgb, 255:red, 0; green, 0; blue, 0 }  ][fill={rgb, 255:red, 0; green, 0; blue, 0 }  ][line width=1.5]      (0, 0) circle [x radius= 4.36, y radius= 4.36]   ;
\draw [line width=1.5]  [dash pattern={on 5.63pt off 4.5pt}]  (121.45,217.95) .. controls (147.45,203.08) and (154.45,222.19) .. (192.45,192.46) ;
\draw [color={rgb, 255:red, 208; green, 2; blue, 27 }  ,draw opacity=1 ][line width=1.5]    (193.45,105.4) .. controls (280.45,118.14) and (223.45,177.6) .. (192.45,192.46) ;
\draw [color={rgb, 255:red, 208; green, 2; blue, 27 }  ,draw opacity=1 ][line width=1.5]    (192.45,192.46) .. controls (219.45,202.02) and (203.45,223.26) .. (267.45,223.26) ;
\draw [color={rgb, 255:red, 74; green, 146; blue, 226 }  ,draw opacity=1 ][line width=2.25]    (193.45,105.4) .. controls (193.23,103.17) and (194.23,101.78) .. (196.44,101.21) .. controls (198.56,100.03) and (198.93,98.42) .. (197.54,96.38) .. controls (195.9,94.83) and (195.85,93.19) .. (197.4,91.45) .. controls (198.8,89.42) and (198.52,87.73) .. (196.57,86.39) .. controls (194.58,85.1) and (194.22,83.53) .. (195.49,81.66) .. controls (196.67,79.51) and (196.26,77.88) .. (194.25,76.75) .. controls (192.24,75.59) and (191.83,73.91) .. (193.02,71.72) .. controls (194.31,69.81) and (193.99,68.22) .. (192.05,66.97) .. controls (190.15,65.62) and (189.92,63.92) .. (191.35,61.89) .. controls (192.92,60.25) and (192.86,58.68) .. (191.18,57.19) .. controls (189.62,55.26) and (189.8,53.52) .. (191.71,51.95) .. controls (193.67,50.92) and (194.14,49.34) .. (193.11,47.21) -- (193.8,45.69) ;
\draw [line width=1.5]  [dash pattern={on 5.63pt off 4.5pt}]  (406.45,202.25) .. controls (350.45,172.57) and (349.45,138.4) .. (407.45,128.51) ;
\draw [shift={(407.45,128.51)}, rotate = 350.32] [color={rgb, 255:red, 0; green, 0; blue, 0 }  ][fill={rgb, 255:red, 0; green, 0; blue, 0 }  ][line width=1.5]      (0, 0) circle [x radius= 4.36, y radius= 4.36]   ;
\draw [shift={(406.45,202.25)}, rotate = 207.92] [color={rgb, 255:red, 0; green, 0; blue, 0 }  ][fill={rgb, 255:red, 0; green, 0; blue, 0 }  ][line width=1.5]      (0, 0) circle [x radius= 4.36, y radius= 4.36]   ;
\draw [line width=1.5]  [dash pattern={on 5.63pt off 4.5pt}]  (335.45,223.83) .. controls (361.45,211.24) and (368.45,227.43) .. (406.45,202.25) ;
\draw [color={rgb, 255:red, 208; green, 2; blue, 27 }  ,draw opacity=1 ][line width=1.5]    (407.45,128.51) .. controls (494.45,139.3) and (437.45,189.66) .. (406.45,202.25) ;
\draw [color={rgb, 255:red, 208; green, 2; blue, 27 }  ,draw opacity=1 ][line width=1.5]    (406.45,202.25) .. controls (433.45,210.34) and (417.45,228.33) .. (481.45,228.33) ;
\draw [line width=1.5]  [dash pattern={on 5.63pt off 4.5pt}]  (401.8,68.57) .. controls (350.8,73.07) and (339.8,113.54) .. (407.45,128.51) ;
\draw [shift={(407.45,128.51)}, rotate = 12.48] [color={rgb, 255:red, 0; green, 0; blue, 0 }  ][fill={rgb, 255:red, 0; green, 0; blue, 0 }  ][line width=1.5]      (0, 0) circle [x radius= 4.36, y radius= 4.36]   ;
\draw [shift={(401.8,68.57)}, rotate = 174.96] [color={rgb, 255:red, 0; green, 0; blue, 0 }  ][fill={rgb, 255:red, 0; green, 0; blue, 0 }  ][line width=1.5]      (0, 0) circle [x radius= 4.36, y radius= 4.36]   ;
\draw [color={rgb, 255:red, 208; green, 2; blue, 27 }  ,draw opacity=1 ][line width=1.5]    (407.45,128.51) .. controls (458.8,113.54) and (471.8,72.17) .. (401.8,68.57) ;
\draw [line width=1.5]  [dash pattern={on 5.63pt off 4.5pt}]  (370.8,55.98) .. controls (380.8,61.38) and (378.8,65.88) .. (401.8,68.57) ;
\draw [color={rgb, 255:red, 208; green, 2; blue, 27 }  ,draw opacity=1 ][line width=1.5]    (401.8,68.57) .. controls (417.8,66.77) and (434.8,55.08) .. (442.8,49.69) ;

\draw (120,138.4) node [anchor=north west][inner sep=0.75pt]    {\footnotesize $\Gamma _{12}$};
\draw (246,138.4) node [anchor=north west][inner sep=0.75pt]    {\footnotesize$\Gamma _{13}$};
\draw (205,74.4) node [anchor=north west][inner sep=0.75pt]    {\footnotesize$\Gamma _{23}$};
\draw (339,123.4) node [anchor=north west][inner sep=0.75pt]    {\footnotesize $\Gamma _{12}$};
\draw (455,122.4) node [anchor=north west][inner sep=0.75pt]    {\footnotesize $\Gamma _{13}$};

\end{tikzpicture}

\caption{\small The two possible situations described by Case 2 in Lemma \ref{lem: nod mer}. We can have a finite ``tower of loops", but after finitely many loops the presence of a simple triple point is necessary.}
\label{fig:0}
\end{figure}
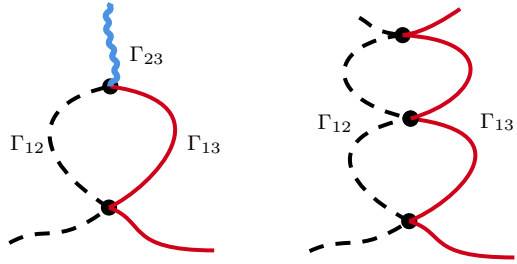

\begin{lemma}\label{lem: expansion}
Let $x_0, r_0$ be given by Lemma \ref{lem: nod mer}. Then, denoting by $(\rho,\theta)$ polar coordinates centered at $x_0$, we have that up to a rotation
\begin{align*}
v_1(\rho,\theta) = \sigma\rho^{3/4}\sin\left(\frac34 \theta\right) + O(\rho) \quad &\text{in $\omega_1$}\\
v_2(\rho,\theta) = \sigma\rho^{3/4}\sin\left(\frac34 \left(\theta - \frac23 \pi\right)\right) + O(\rho) \quad &\text{in $\omega_2$}\\
v_3(\rho,\theta)  = \sigma\rho^{3/4}\sin\left(\frac34 \left(\theta - \frac43\pi\right)\right) + O(\rho) \quad &\text{in $\omega_3$},
\end{align*}
and, with respect to the orthonormal frame $(e_\rho, e_\theta)$,
\begin{align*}
\nabla v_1(\rho,\theta) = \frac{3}{4} \sigma \rho^{-1/4} \left(\sin\left(\frac{3}{4}\theta\right), \cos\left(\frac{3}{4}\theta\right) \right)  + O(1) \quad &\text{in $\omega_1$}\\
\nabla v_2(\rho,\theta) = \frac{3}{4} \sigma\rho^{-1/4} \left(\sin\left(\frac{3}{4}\left(\theta - \frac23 \pi\right)\right), \cos\left(\frac{3}{4}\left(\theta - \frac23 \pi\right)\right) \right)  + O(1)
\quad &\text{in $\omega_2$}\\
\nabla v_3(\rho,\theta)  = \frac{3}{4} \sigma\rho^{-1/4} \left(\sin\left(\frac{3}{4}\left(\theta - \frac43 \pi\right)\right), \cos\left(\frac{3}{4}\left(\theta - \frac43 \pi\right)\right) \right)  + O(1)
\quad &\text{in $\omega_3$}
\end{align*}
as $\rho \to 0^+$.
\end{lemma}

\begin{proof}
By Lemma \ref{lem: nod mer}, in $B_{r_0}(x_0)$ the nodal set of $\mf{v}$ consists of one arc of $\Gamma_{12}$, one of $\Gamma_{13}$, and one of $\Gamma_{23}$ meeting at $x_0$. Up to a rotation and a relabelling, we can suppose that the point of $\Gamma_{12} \cap \pa B_{r_0}(x_0)$ stays on the ray $\theta=0$, and that, moving on $\pa B_{r_0}(x_0)$ in counterclockwise sense, the first nodal point belongs to $\Gamma_{23}$. This way, $\omega_1 \cap B_{r_0}(x_0)$ is delimited by a portion of $\pa B_{r_0}(x_0)$ and by $\Gamma_{12} \cup \Gamma_{13} \cup \{x_0\}$, and $\Gamma_{23} \subset \omega_1 \cap B_{r_0}(x_0)$; $\omega_2 \cap B_{r_0}(x_0)$ is delimited by a portion of $\pa B_{r_0}(x_0)$ and by $\Gamma_{23} \cup \Gamma_{12} \cup \{x_0\}$, and $\Gamma_{13} \subset \omega_2 \cap B_{r_0}(x_0)$; and $\omega_3 \cap B_{r_0}(x_0)$ is delimited by a portion of $\pa B_{r_0}(x_0)$ and by $\Gamma_{13} \cup \Gamma_{23} \cup \{x_0\}$, and $\Gamma_{12} \subset \,\omega_3\, \cap B_{r_0}(x_0)$; see the picture on the left in Figure \ref{fig:1}.

In polar coordinates, $\mf{v}$ is defined on $[0,r_0] \times [0,2\pi)$, and we can extend it by periodicity on the strip $[0,r_0] \times \R$. Each positivity set $\omega_1$, $\omega_2$ and $\omega_3$ is identified with a sequence of topological rectangles in the strip, and we notice that each component of $\omega_1$ is adjacent to $\omega_3$ on the right and to $\omega_2$ on the left; each component of $\omega_2$ is adjacent to $\omega_1$ on the right and to $\omega_3$ on the left; and each component of $\omega_3$ is adjacent to $\omega_2$ on the right and to $\omega_1$ on the left.

We can then define a function $\hat w$ as follows:\\
$\hat w:=v_1$ on a component $\omega_{1,1}$ of $\omega_1$; \\
$\hat w:=-v_3$ in the component $\omega_{3,1}$ of \( \omega_3 \) adjacent to \( \omega_{1,1} \) on the right; \\
$\hat w:=v_2$ in the component $\omega_{2,1}$ of \( \omega_2 \) adjacent to \( \omega_{3,1} \) on the right; \\
$\hat w:=-v_1$ in the component $\omega_{1,2}$ of \( \omega_1 \) adjacent to \( \omega_{2,1} \) on the right; \\
$\hat w:=v_3$ in the component $\omega_{3,2}$ of \( \omega_3 \) adjacent to \( \omega_{1,2} \) on the right; \\
$\hat w:=-v_2$ in the component $\omega_{2,2}$ of \( \omega_2 \) adjacent to \( \omega_{3,2} \) on the right.\\
We refer again to Figure \ref{fig:1} for a graphical representation.

By construction, recalling that $v_i -v_j$ is harmonic across $\Gamma_{ij}$ (Lemma \ref{prop: diff arm}), we have that $\hat w$ is a $8\pi$-periodic-in-$\theta$ harmonic function on the strip; namely it is a harmonic function on the quadruple covering of the punctured disc $B_{r_0}(x_0) \setminus \{0\}$, bounded at $0$ (in fact, $\hat w(x_0) = 0$). Thus, $w(\rho,\theta) = \hat w(\rho^4,4\theta)$ is a harmonic function on the punctured disc $B_{r_0}(x_0) \setminus \{0\}$, vanishing at $0$, and hence it is also harmonic in the full disc $B_{r_0}(x_0)$. By harmonicity, classical results \`a la Hartman-Wintner \cite{HaWi} ensures that $w$ admits an expansion of type
\[
w(x) = P_d(x-x_0)+R_{d+1}(x-x_0), \quad \nabla w(x) = \nabla P_d(x-x_0)+\nabla R_{d+1}(x-x_0),
\]
where $P_d \not \equiv 0$ is a homogeneous harmonic polynomial of some degree $d \in \N$, and
\[
R_{d+1}(x-x_0) = O(|x-x_0|^{d+1}), \quad \nabla R_{d+1}(x-x_0) = O(|x-x_0|^{d}).
\]
Since by construction $w$ has exactly $6$ nodal lines meeting at the origin, we deduce that $d=3$. Thus, up to a rotation $P_d(x-x_0) = \sigma \mathrm{Re}((z-z_0)^3)$ for some $\sigma>0$.
Passing to polar coordinates, and translating the above expansion in terms of the original function $\hat w$, the thesis follows.
\end{proof}

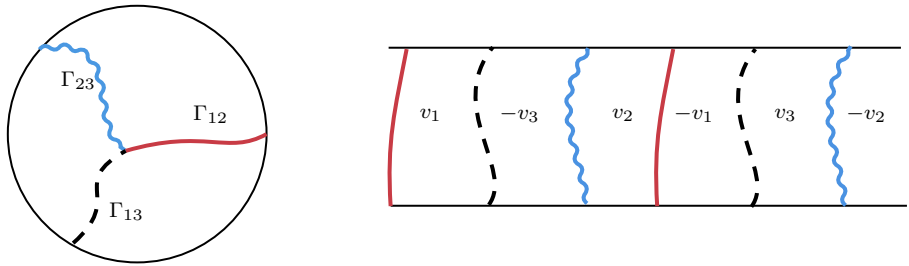
\begin{figure}[ht]
\begin{center}

\tikzset{every picture/.style={line width=0.75pt}}

\begin{tikzpicture}[x=0.75pt,y=0.75pt,yscale=-1,xscale=1]

\draw   (100,114.45) .. controls (100,78.8) and (128.9,49.9) .. (164.55,49.9) .. controls (200.2,49.9) and (229.1,78.8) .. (229.1,114.45) .. controls (229.1,150.1) and (200.2,179) .. (164.55,179) .. controls (128.9,179) and (100,150.1) .. (100,114.45) -- cycle ;
\draw [color={rgb, 255:red, 74; green, 152; blue, 226 }  ,draw opacity=1 ][line width=1.5]    (115.8,71.9) .. controls (117.25,69.91) and (118.94,69.6) .. (120.88,70.97) .. controls (122.92,72.42) and (124.58,72.28) .. (125.87,70.55) .. controls (127.6,68.9) and (129.31,69.01) .. (131.01,70.86) .. controls (132.16,72.81) and (133.67,73.23) .. (135.54,72.12) .. controls (138.1,71.61) and (139.53,72.52) .. (139.83,74.87) .. controls (139.68,77.18) and (140.69,78.46) .. (142.86,78.73) .. controls (145.19,79.69) and (145.88,81.26) .. (144.93,83.44) .. controls (143.79,85.47) and (144.2,87.01) .. (146.15,88.04) .. controls (148.13,89.49) and (148.44,91.26) .. (147.09,93.33) .. controls (145.68,95.12) and (145.91,96.69) .. (147.78,98.02) .. controls (149.65,99.34) and (149.91,101.03) .. (148.55,103.1) .. controls (147.2,104.96) and (147.53,106.61) .. (149.55,108.04) .. controls (151.54,109.03) and (152,110.57) .. (150.95,112.68) .. controls (150.11,114.95) and (150.84,116.45) .. (153.13,117.17) .. controls (155.37,117.44) and (156.4,118.7) .. (156.22,120.94) -- (158.8,122.9) ;
\draw [color={rgb, 255:red, 200; green, 59; blue, 69 }  ,draw opacity=1 ][line width=1.5]    (158.8,122.9) .. controls (198.6,111.3) and (209.6,125.3) .. (229.1,114.45) ;
\draw [line width=1.5]  [dash pattern={on 5.63pt off 4.5pt}]  (132.6,169.3) .. controls (158.6,146.3) and (129.6,139.3) .. (158.8,122.9) ;
\draw    (290,71) -- (545.6,71) ;
\draw    (291,150.5) -- (548.6,150.5) ;
\draw [color={rgb, 255:red, 200; green, 59; blue, 69 }  ,draw opacity=1 ][line width=1.5]    (291,150.7) .. controls (289,117.7) and (293,102.7) .. (299,71.7) ;
\draw [line width=1.5]  [dash pattern={on 5.63pt off 4.5pt}]  (340,150) .. controls (353,132.3) and (318,101.3) .. (342,70.7) ;
\draw [color={rgb, 255:red, 74; green, 146; blue, 226 }  ,draw opacity=1 ][line width=1.5]    (389,149.3) .. controls (387.05,147.76) and (386.75,146.09) .. (388.11,144.3) .. controls (389.38,142.19) and (389.01,140.45) .. (387,139.08) .. controls (385.05,138.06) and (384.69,136.53) .. (385.92,134.5) .. controls (387.13,132.47) and (386.73,130.87) .. (384.71,129.68) .. controls (382.7,128.53) and (382.27,126.85) .. (383.43,124.64) .. controls (384.67,122.73) and (384.27,121.1) .. (382.24,119.75) .. controls (380.27,118.63) and (379.93,117.04) .. (381.21,114.98) .. controls (382.54,113.01) and (382.25,111.32) .. (380.35,109.91) .. controls (378.52,108.65) and (378.35,107.09) .. (379.86,105.24) .. controls (381.45,103.52) and (381.42,101.84) .. (379.75,100.19) .. controls (378.18,98.36) and (378.31,96.66) .. (380.15,95.08) .. controls (382.06,93.66) and (382.36,92.04) .. (381.04,90.23) .. controls (379.85,88.25) and (380.34,86.59) .. (382.52,85.25) .. controls (384.66,84.31) and (385.24,82.86) .. (384.27,80.91) .. controls (383.56,78.53) and (384.31,77.03) .. (386.52,76.41) .. controls (388.92,75.58) and (389.76,74.17) .. (389.04,72.17) -- (389.6,71.3) ;
\draw [color={rgb, 255:red, 200; green, 59; blue, 69 }  ,draw opacity=1 ][line width=1.5]    (424,150.7) .. controls (422,117.7) and (426,102.7) .. (432,71.7) ;
\draw [line width=1.5]  [dash pattern={on 5.63pt off 4.5pt}]  (471,151) .. controls (484,133.3) and (449,102.3) .. (473,71.7) ;
\draw [color={rgb, 255:red, 74; green, 146; blue, 226 }  ,draw opacity=1 ][line width=1.5]    (518,150.3) .. controls (516.05,148.75) and (515.75,147.07) .. (517.12,145.28) .. controls (518.41,143.18) and (518.1,141.57) .. (516.17,140.44) .. controls (514.18,139.08) and (513.81,137.38) .. (515.08,135.35) .. controls (516.34,133.33) and (515.97,131.68) .. (513.98,130.41) .. controls (512,129.16) and (511.65,127.56) .. (512.93,125.62) .. controls (514.24,123.71) and (513.89,122.02) .. (511.9,120.55) .. controls (509.98,119.3) and (509.71,117.77) .. (511.1,115.96) .. controls (512.49,113.97) and (512.26,112.21) .. (510.41,110.68) .. controls (508.62,109.3) and (508.5,107.68) .. (510.06,105.81) .. controls (511.7,104.06) and (511.7,102.44) .. (510.07,100.93) .. controls (508.52,99.03) and (508.66,97.4) .. (510.49,96.03) .. controls (512.41,94.55) and (512.71,92.89) .. (511.39,91.06) .. controls (510.25,88.79) and (510.72,87.1) .. (512.81,85.99) .. controls (514.94,85) and (515.49,83.54) .. (514.48,81.6) .. controls (513.71,79.26) and (514.41,77.76) .. (516.6,77.1) .. controls (518.97,76.25) and (519.84,74.7) .. (519.23,72.46) -- (520.6,70.3) ;

\draw (125,81.4) node [anchor=north west][inner sep=0.75pt]    {\footnotesize $\Gamma_{23}$};
\draw (149,147.4) node [anchor=north west][inner sep=0.75pt]    {\footnotesize $\Gamma_{13}$};
\draw (191,98.4) node [anchor=north west][inner sep=0.75pt]    {\footnotesize $\Gamma_{12}$};
\draw (304,100.4) node [anchor=north west][inner sep=0.75pt]    {\footnotesize $v_1$};
\draw (344,99) node [anchor=north west][inner sep=0.75pt]    {\footnotesize $-v_3$};
\draw (400,100.4) node [anchor=north west][inner sep=0.75pt]    {\footnotesize $v_2$};
\draw (431,99) node [anchor=north west][inner sep=0.75pt]    {\footnotesize $-v_1$};
\draw (481,100.4) node [anchor=north west][inner sep=0.75pt]    {\footnotesize $v_3$};
\draw (517,99) node [anchor=north west][inner sep=0.75pt]    {\footnotesize $-v_2$};

\end{tikzpicture}
\end{center}
\caption{\small On the left, the structure of the nodal set in $B_{r_0}(x_0)$. On the right, the quadruple covering on the punctured ball $B_{r_0}(x_0) \setminus \{x_0\}$, with the correct ordering of the nodal regions in order to define the function $\hat w$. Notice that we did not draw all the nodal lines in the second picture (for instance, there is a copy of $\Gamma_{23}$ in the first region where $\hat w=v_1$, but we did not draw it since it plays no role in the definition of $\hat w$).}
\label{fig:1}
\end{figure}

\begin{remark}\label{rem: covering}
It is convenient to record here some facts which are implicitly contained in the
previous proof, and which will be used in the construction of a competitor below.
For $0<r_1<r_2$, let
\[
\mathcal{C}(r_1,r_2) := \left\{ (\rho,\theta) :\ r_1 < \rho < r_2, \ \theta \in \R/8\pi \mathbb{Z} \right\},
\qquad \pi(\rho,\theta) := x_0 + \rho e^{i\theta},
\]
denote the quadruple covering of the annulus $B_{r_2}(x_0) \setminus \overline{B_{r_1}(x_0)}$;
its deck transformations, namely the homeomorphisms $h$ of $\mathcal{C}$ such that $\pi \circ h
= \pi$, form the cyclic group of order four generated by $T(\rho,\theta) := (\rho, \theta +
2\pi)$, so that every fibre $\pi^{-1}(x)$ consists of the four points $\tilde x$, $T \tilde x$,
$T^2 \tilde x$, $T^3 \tilde x$,
and let
\[
\hat{\tilde w}(\rho,\theta) := \rho^{3/4} \sin\Big(\frac34 \theta\Big)
\]
be the function associated with $\tilde{\mf{v}}$ by the same recipe used above for $\mf{v}$.
Then the following hold.
\begin{itemize}
\item[($a$)] Both $\hat w$ and $\hat{\tilde w}$ are \emph{anti-invariant} with respect to $T^2$:
\beq\label{antisym}
\hat w \circ T^2 = -\hat w, \qquad \hat{\tilde w} \circ T^2 = - \hat{\tilde w}.
\eeq
For $\hat{\tilde w}$ this is the identity
$\sin\big(\frac34(\theta+4\pi)\big) = \sin\big(\frac34\theta + 3\pi\big) = -\sin\big(\frac34\theta\big)$;
for $\hat w$ it follows from its very definition, since in the list
$\omega_{1,1},\omega_{3,1},\omega_{2,1},\omega_{1,2},\omega_{3,2},\omega_{2,2}$ the fourth,
fifth and sixth region are the images through $T^2$ of the first, second and third one
respectively, while the corresponding signs are opposite.

\item[($b$)] The expansion of Lemma \ref{lem: expansion} holds at the level of the covering:
\beq\label{exp covering}
\hat w = \sigma \hat{\tilde w} + O(\rho), \qquad
\nabla \hat w = \sigma \nabla \hat{\tilde w} + O(1) \qquad \text{as $\rho \to 0^+$},
\eeq
this being exactly what the Hartman--Wintner expansion of $w$ gives, before descending to
the disc.

\item[($c$)] The gradient of $\hat{\tilde w}$ never vanishes; more precisely
\beq\label{grad model}
\big|\nabla \hat{\tilde w}\big|(\rho,\theta) = \frac34\, \rho^{-1/4}
\qquad \text{for every $(\rho,\theta)$},
\eeq
since $\nu^2\sin^2(\nu\theta)+\nu^2\cos^2(\nu\theta) = \nu^2$ with $\nu = 3/4$. In particular
the nodal set of $\hat{\tilde w}$ on $\mathcal{C}(r_1,r_2)$ consists of the six radial arcs
$\{\theta = 4\pi j/3\}$, $j=0,\dots,5$, and $T^2$ maps each of them onto the one three steps
ahead.
\end{itemize}
\end{remark}

\begin{lemma}\label{lem: reconstruction}
Let $0<r_1<r_2$, let $A:= B_{r_2}(x_0) \setminus \overline{B_{r_1}(x_0)}$, and let
$\mathcal{C}:=\mathcal{C}(r_1,r_2)$, $T$, $\pi$ be as in Remark \ref{rem: covering}. Let
$\hat W \in C^0(\mathcal{C}) \cap H^1(\mathcal{C})$ be such that
\begin{itemize}
\item[($i$)] $\hat W \circ T^2 = - \hat W$;
\item[($ii$)] the open set $\{\hat W \neq 0\}$ has exactly six connected components
$\mathcal{N}_1, \dots, \mathcal{N}_6$, with $T^2 \mathcal{N}_j = \mathcal{N}_{j+3}$ (indexes
modulo $6$), and $\pi|_{\mathcal{N}_j}$ is injective for every $j$.
\end{itemize}
Let moreover $t : \{1,\dots,6\} \to \{1,2,3\}$ be a surjective map such that $t(j+3) = t(j)$ for
every $j$ (indexes modulo 6), to be thought of as a labelling of the components $\mathcal{N}_j$,
and define $\mathbf{V} = (V_1,V_2,V_3)$ on $A$ by
\beq\label{def V}
V_i := \big|\hat W\big| \circ \big(\pi|_{\mathcal{N}_j}\big)^{-1} \ \text{ on } \pi(\mathcal{N}_j),
\ \text{ for } t(j)=i, \qquad V_i:=0 \ \text{ elsewhere in } A.
\eeq
Then $\mf{V}$ is well defined, non-negative, continuous, belongs to $H^1(A,\R^3)$, satisfies
the partial segregation condition
\[
V_1\,V_2\,V_3 \equiv 0 \qquad \text{in $A$},
\]
and
\beq\label{energy covering}
J(\mf{V},A) = \frac12 \int_{\mathcal{C}} \big|\nabla \hat W\big|^2 .
\eeq
\end{lemma}

\begin{proof}
\emph{Well-posedness.} The two indexes $j$ with $t(j)=i$ are $j$ and $j+3$, and
$\pi(\mathcal{N}_{j+3}) = \pi(T^2\mathcal{N}_j) = \pi(\mathcal{N}_j)$; moreover
$|\hat W| \circ T^2 = |\hat W|$ by ($i$). Hence the two formulas in \eqref{def V} agree,
and $V_i$ is well defined.

\emph{Regularity.} On $\mathcal{N}_j$ the function $|\hat W|$ is continuous and vanishes on
$\pa \mathcal{N}_j$, since $\mathcal{N}_j$ is a connected component of $\{\hat W \neq 0\}$;
therefore its extension by $0$ is continuous and belongs to $H^1$. As $\pi$ is a local
isometry and $\pi|_{\mathcal{N}_j}$ is injective, the same holds for $V_i$.

\emph{Partial segregation.} Let $x \in A$, and let
$\pi^{-1}(x) = \{\tilde x,\, T\tilde x,\, T^2 \tilde x,\, T^3 \tilde x\}$. By ($i$),
\[
\hat W(T^2 \tilde x) = - \hat W(\tilde x), \qquad \hat W(T^3 \tilde x) = -\hat W(T\tilde x),
\]
so that $\tilde x \in \mathcal{N}_j$ if and only if $T^2 \tilde x \in T^2\mathcal{N}_j =
\mathcal{N}_{j+3}$, and analogously for the pair $\{T\tilde x, T^3\tilde x\}$. Since
$t(j+3)=t(j)$, the four points of the fibre are contained in components carrying \emph{at
most two} distinct labels. Consequently there exists $i \in \{1,2,3\}$ such that
$\pi^{-1}(x)$ does not meet any component of label $i$, and for such an $i$ we have
$V_i(x)=0$ by \eqref{def V}. We stress that this argument does not use any information on
the nodal set of $\hat W$: it only rests on ($i$) and on the fact that the labels are
constant along the $T^2$-orbits.

\emph{Energy.} For each $i$, choose one of the two regions $\mathcal N_j$ with $t(j)=i$. The local isometry $\pi$ maps this region injectively onto $\{V_i>0\}$, and its partner $T^2\mathcal N_j$ has the same Dirichlet energy by $(i)$. Since the weak gradient of an $H^1$ function vanishes almost everywhere on its zero set, summing the integrals over the six regions yields
\[
\int_{\mathcal C}|\nabla\hat W|^2
=2\sum_{i=1}^3\int_{\{V_i>0\}}|\nabla V_i|^2
=2J(\mf V,A).
\]
This proves \eqref{energy covering}.
\end{proof}

\begin{proof}[Conclusion of the proof of Theorem \ref{prop: mer min}]
We are supposing by contradiction that $\tilde{\mf{v}}$ is not a minimizer, namely $c_\infty<\tilde c $. Now, for $0<r < r_0$ given by Lemma \ref{lem: nod mer}, we consider the minimization problem
\[
c_r:= \min\left\{J(\mf{u},B_r(x_0)) \left|\begin{array}{l}  \mf{u}-\mf{v} \in H_0^1(B_r(x_0),\R^3) \\ u_1\,u_2\,u_3 \equiv 0 \ \text{in $B_r(x_0)$}\end{array}\right. \right\},
\]
where $\mf{v}$ is a minimizer for $c_\infty$. A fortiori, its restriction on $B_r(x_0)$ is a minimizer for $c_r$, and hence an estimate for $c_r$ can be easily obtained by using the expansion in Lemma \ref{lem: expansion} and what we observed in Remark \ref{rem: scale min}: we have that
\beq\label{stima en 0 mer}
c_r = \sigma^2 \tilde c\, r^{3/2} + O(r^{7/4})  \qquad \text{as $r \to 0^+$.}
\eeq
Indeed, by Lemma \ref{lem: expansion},
\[
\sum_i |\nabla v_i|^2 = \frac{9}{16}\sigma^2 \rho^{-1/2}
\sum_i \chi_{\omega_i} + O(\rho^{-1/4}) \quad \text{in $B_{r_0}(x_0)$}
\]
(recall that $\nabla v_i \equiv 0$
outside $\omega_i$). Since $\int_{B_r(x_0)} \rho^{-1/4}\,dx = O(r^{7/4})$, while the leading term
integrates, by Remark \ref{rem: scale min}, to $J(\sigma \tilde{\mathbf v}(\cdot - x_0), B_r(x_0)) =
\sigma^2 \tilde c\, r^{3/2}$, we obtain \eqref{stima en 0 mer}.

In order to produce a competitor of $\mf{v}$ decreasing the energy, obtaining a
contradiction, we consider at first a Lipschitz function $\eta = \eta(|x-x_0|)$ with the
properties that, for $\varepsilon>0$ to be chosen later, $\eta \equiv 0$ in
$B_{(1-\varepsilon)r}(x_0)$, $\eta \equiv 1$ on $\pa B_r(x_0)$, $0 \le \eta \le 1$, and
$|\nabla \eta| = 1/(\varepsilon r)$ on the annulus
$A_{\varepsilon,r}:= B_r(x_0) \setminus \overline{B_{(1-\varepsilon)r}(x_0)}$. Since $\eta$
is radial, it lifts to a deck-invariant function (i.e. a function invariant under $T$) on the covering
$\mathcal{C}:=\mathcal{C}((1-\varepsilon)r,\,r)$ of $A_{\varepsilon,r}$, and we may set
\beq\label{def W hat}
\hat W := \eta \, \hat w + (1-\eta)\,\sigma \hat{\tilde w}
= \sigma \hat{\tilde w} + \eta \, \hat u, \qquad \hat u := \hat w - \sigma \hat{\tilde w},
\eeq
on $\mathcal{C}$. We denote by $\mf{V}$ the triplet associated with $\hat W$ through
Lemma \ref{lem: reconstruction}, and we finally define
\[
\hat{\mf{v}}(x):= \begin{cases}
\sigma r^{3/4} (1-\varepsilon)^{3/4}\, \mf{v}\left(\dfrac{x-x_0}{(1-\varepsilon)r}\right)
& \text{in $B_{(1-\varepsilon)r}(x_0)$},\\[3mm]
\mf{V}(x) & \text{in $A_{\varepsilon,r}$}.
\end{cases}
\]
Namely, $\hat{\mf{v}}$ is a competitor obtained by pasting a scaled copy of the actual
minimizer $\mf{v}$ in a neighborhood of the triple point, replacing the expansion given by
Lemma \ref{lem: expansion}, which is not minimal for the problem on $B_1$. The interpolation
in the annulus is performed at the level of the covering, and this is what guarantees that
the partial segregation condition is preserved: a naive convex combination
$\eta \mf{v} + (1-\eta)\sigma\tilde{\mf{v}}$ would not be admissible, since the nodal sets of
$\mf{v}$ and of $\sigma\tilde{\mf{v}}(\,\cdot\,-x_0)$ only agree up to an error of order
$O(\rho)$.

Let us check that $\hat{\mf{v}}$ is well defined and admissible for $c_r$. By
\eqref{antisym} and the deck-invariance of $\eta$ we have $\hat W \circ T^2 = -\hat W$, so
that assumption ($i$) of Lemma \ref{lem: reconstruction} is satisfied. As for ($ii$), by
\eqref{exp covering} we have, on $\mathcal{C}$,
\[
\big|\hat W - \sigma \hat{\tilde w}\big| \le |\hat u| \le C\rho, \qquad
\big|\nabla (\hat W - \sigma\hat{\tilde w})\big| \le \frac{|\hat u|}{\varepsilon r}
+ |\nabla \hat u| \le \frac{C}{\varepsilon} + C ,
\]
where we used that $\rho \le r$ on $\mathcal{C}$. Moreover, since $\eta$ is radial, the cut-off does not affect the angular derivative, whence
\[
\big|\partial_\theta (\hat W - \sigma \hat{\tilde w})\big| = \eta \,\big|\partial_\theta \hat u\big|
\le \rho\, \big|\nabla \hat u\big| \le C \rho \qquad \text{on } \mathcal{C}.
\]
Let now $Z := \{4\pi j /3 : \ j = 0, \dots, 5\}$ denote the set of the nodal angles of
$\hat{\tilde w}(\rho, \cdot) = \rho^{3/4} \sin\big(\frac34 \theta\big)$, see Remark \ref{rem:
covering}-$(c)$, and let us fix $\delta \in (0, \pi/3)$. There exists $c_\delta > 0$ such
that, for every $\rho > 0$,
\[
\big|\hat{\tilde w}(\rho,\theta)\big| \ge c_\delta\, \rho^{3/4} \ \ \text{if } \operatorname{dist}
(\theta, Z) \ge \delta, \qquad \big|\partial_\theta \hat{\tilde w}(\rho,\theta)\big| \ge c_\delta\,
\rho^{3/4} \ \ \text{if } \operatorname{dist}(\theta, Z) \le \delta.
\]
Since $\rho \le r$ on $\mathcal{C}$, and hence $\rho = o(\rho^{3/4})$ as $r \to 0^+$, the previous
estimates imply that, for $r$ sufficiently small, $\hat W(\rho, \cdot)$ does not vanish at angular
distance larger than $\delta$ from $Z$, while it is strictly monotone in each connected component
of the $\delta$-neighborhood of $Z$. Therefore, for every $\rho \in ((1-\varepsilon)r, r)$, the
function $\hat W(\rho, \cdot)$ has exactly six zeros $\theta_1(\rho) < \dots < \theta_6(\rho)$ in
$\mathbb{R}/8\pi\mathbb{Z}$, one in each of such components, and they depend continuously on $\rho$,
by the strict monotonicity and the continuity of $\hat W$. This means that $\{\hat W \neq 0\}$ has
exactly six connected components
\[
\mathcal{N}_j = \big\{ (\rho,\theta) \in \mathcal{C} : \ \theta_j(\rho) < \theta < \theta_{j+1}
(\rho) \big\}, \qquad j = 1, \dots, 6
\]
(indexes modulo 6), each of angular width smaller than $\frac{4\pi}{3} + 2\delta < 2\pi$; in
particular $\pi|_{\mathcal{N}_j}$ is injective. Finally, $T^2 \mathcal{N}_j = \mathcal{N}_{j+3}$,
since $\hat W \circ T^2 = - \hat W$ and $Z + 4\pi = Z$.
It remains to fix the labelling. Since $\eta \equiv 1$ on $\partial B_r(x_0)$, we have $\hat W =
\hat w$ on the outer boundary $\{\rho = r\}$ of $\mathcal{C}$; therefore each $\mathcal{N}_j$ meets
$\{\rho = r\}$ in the lift of exactly one of the six regions $\omega_{1,1}, \omega_{3,1},
\omega_{2,1}, \omega_{1,2}, \omega_{3,2}, \omega_{2,2}$ listed in the proof of Lemma \ref{lem:
expansion}. We number the components $\mathcal{N}_1, \dots, \mathcal{N}_6$ accordingly, and we
choose
\[
t(1) = t(4) = 1, \qquad t(2) = t(5) = 3, \qquad t(3) = t(6) = 2,
\]
which is an admissible labelling in the sense of Lemma \ref{lem: reconstruction}, and is consistent
with the definition of $\hat w$. With this choice $|\hat W| = |\hat w| = v_{t(j)}$ on $\mathcal{N}_j
\cap \{\rho = r\}$ for every $j$, so that $\mathbf{V} = \mathbf{v}$ on $\partial B_r(x_0)$.

The previous discussion shows that Lemma \ref{lem: reconstruction} applies, and in particular
\[
\hat v_1\, \hat v_2\, \hat v_3 \equiv 0 \qquad \text{in $A_{\varepsilon,r}$}.
\]
As for the traces, we already checked that $\mathbf{V} = \mathbf{v}$ on $\partial B_r(x_0)$; on
$\partial B_{(1-\varepsilon)r}(x_0)$ we have $\eta = 0$,
hence
$\hat W = \sigma\hat{\tilde w}$ and $\mf{V} = \sigma \tilde{\mf{v}}(\,\cdot\,-x_0)$, which
agrees with the inner definition since
\[
\sigma r^{3/4} (1-\varepsilon)^{3/4} \tilde{\mf{v}}\left(\frac{x-x_0}{(1-\varepsilon)r}\right)
= \sigma \tilde{\mf{v}}(x-x_0) \qquad \text{on $\pa B_{(1-\varepsilon)r}(x_0)$},
\]
being $\mf{v}-\tilde{\mf{v}} \in H_0^1(B_1,\R^3)$ and $\tilde{\mf{v}}$ $3/4$-homogeneous.
Therefore $\hat{\mf{v}} \in H^1(B_r(x_0),\R^3)$, $\hat{\mf{v}} - \mf{v} \in
H_0^1(B_r(x_0),\R^3)$, and $\hat{\mf{v}}$ is admissible for $c_r$.

Finally, to estimate the energy of $\hat{\mf{v}}$, we choose
\begin{equation}\label{choice eps}
\varepsilon = r^\beta \qquad \text{with } 0 < \beta < \frac12 .
\end{equation}
Now, clearly
\beq\label{en hat v}
J(\hat{\mf{v}},B_r(x_0)) = J(\hat{\mf{v}},B_{(1-\varepsilon)r}(x_0))
+ J(\hat{\mf{v}},A_{\varepsilon,r})
= \sigma^2 c_\infty (1-\varepsilon)^{3/2} r^{3/2} + J(\hat{\mf{v}},A_{\varepsilon,r}),
\eeq
where we used again the computations in Remark \ref{rem: scale min}. Concerning the energy on
the annulus, by \eqref{energy covering}, \eqref{grad model} and the estimates above,
\[
J(\hat{\mf{v}},A_{\varepsilon,r}) = \frac12 \int_{\mathcal{C}} \big|\nabla \hat W\big|^2
\le C \int_{(1-\varepsilon)r}^{r}
\left( \rho^{-1/4} + \frac{\rho}{\varepsilon r} + 1 \right)^2 \rho \, d\rho
\le C\left( \varepsilon r^{3/2} + \frac{r^2}{\varepsilon} + \varepsilon r^2 \right),
\]
for $\varepsilon$ and $r$ sufficiently small. By \eqref{choice eps} this gives
\[
J(\hat{\mf{v}},A_{\varepsilon,r}) \le
C\left( r^{3/2+\beta} + r^{2-\beta} + r^{2+\beta}\right) = o\big(r^{3/2}\big),
\]
since $0<\beta<1/2$. Coming back to \eqref{en hat v}, and using
$(1-\varepsilon)^{3/2} = 1 + O(r^{\beta})$, we deduce that
\[
J(\hat{\mf{v}},B_r(x_0)) = \sigma^2 c_\infty r^{3/2} + o\big(r^{3/2}\big).
\]
But then a comparison with \eqref{stima en 0 mer} gives
\[
\sigma^2 \tilde c\, r^{3/2} + o\big(r^{3/2}\big) \le
\sigma^2 c_\infty r^{3/2} + o\big(r^{3/2}\big),
\]
finally implying that $\tilde c \le c_\infty$, which is the desired contradiction.
\end{proof}

\end{document}